\documentclass[oneside]{amsart}
\usepackage[dvipsnames]{xcolor} 
\usepackage{algorithmic,algorithm} 
\usepackage[utf8]{inputenc}
\usepackage[T1]{fontenc}
\usepackage{siunitx}
\usepackage{geometry}

\usepackage{amsfonts,amssymb}
\usepackage{amsmath}
\usepackage{amsthm}
\usepackage{color}
\usepackage{graphicx}
\usepackage{stmaryrd}
\usepackage[position=t]{subcaption}

\usepackage{mathtools}
\mathtoolsset{showonlyrefs}
\usepackage{pgfplots}
\usepackage{pgfplotstable}
\usepackage{booktabs}
\usepackage{placeins}

\pgfplotsset{compat=1.14}

\usepackage[hidelinks]{hyperref}

\usepackage{microtype}
\usepackage{enumitem}

\newcommand{\jump}[1]{\llbracket #1 \rrbracket}
\newcommand{\avg}[1]{\lbrace\!\!\lbrace #1 \rbrace\!\!\rbrace}
\newcommand{\dgnorm}[1]{\lVert #1 \rVert_\mathrm{DG}}
\newcommand{\fullnorm}[2][]{\left\vert\kern-0.25ex\left\vert\kern-0.25ex\left\vert #2 
      \right\vert\kern-0.25ex\right\vert\kern-0.25ex\right\vert_{\mathrm{DG} #1}}
\newcommand{\fullnormfixedsize}[2][]{\vert\kern-0.25ex\vert\kern-0.25ex\vert #2 
      \vert\kern-0.25ex\vert\kern-0.25ex\vert_{\mathrm{DG} #1}}

\newtheorem{theorem}{Theorem}
\newtheorem{lemma}{Lemma}
\newtheorem{remark}{Remark}
\newtheorem{ind-assumption}{Induction Assumption}
\newtheorem{proposition}[lemma]{Proposition}

\let\epsilon\varepsilon

\let\phi\varphi

\let\theta\vartheta

\let\myempty\varnothing

\def\hp{$hp$}
\newcommand{\de}[1]{#1_\delta}
\newcommand{\spacenameinnorm}{\mathcal{D}}
\newcommand{\spacenameinnormhom}{\mathcal{G}}

\newcommand{\wnormdg}[5]{\left\vert\kern-0.25ex\left\vert\kern-0.25ex\left\vert #1 
      \right\vert\kern-0.25ex\right\vert\kern-0.25ex\right\vert_{\spacenameinnorm^{#3}_{#4}(#5)} }
\newcommand{\wnormdghom}[5]{\left\vert\kern-0.25ex\left\vert\kern-0.25ex\left\vert #1 
      \right\vert\kern-0.25ex\right\vert\kern-0.25ex\right\vert_{\spacenameinnormhom^{#3}_{#4}(#5)} }
\newcommand{\dhe}{\mathtt{h_e}}
\newcommand{\dpe}{\mathtt{p_e}}

\newcommand{\slope}{\mathfrak{s}}

\newcommand{\psiwd}{{{\psi}_{w_\delta}}}

\newcommand{\udso}{{{U}_{\delta}^{\star\perp}}}

\newcommand{\Ad}{A_{\delta}}

\newcommand{\ud}{{u_\delta}}

\newcommand{\hK}{{\hat{K}}}

\newcommand{\fc}{{\mathfrak{c}}}
\newcommand{\fC}{{\mathfrak{C}}}

\newcommand{\tv}{{\tilde{v}}}

\newcommand{\cK}{\mathcal{K}}
\newcommand{\cJ}{\mathcal{J}}

\newcommand{\dalpha}{\partial^\alpha}
\newcommand{\dzeta}{\partial^\zeta}

\newcommand{\dOmega}{{\partial\Omega}}

\newcommand{\dbeta}{\partial^\beta}
\newcommand{\dab}{\partial^{\alpha+\beta}}
\newcommand{\betam}{{|\beta|}}
\newcommand{\zetam}{{|\zeta|}}
\newcommand{\xim}{{|\xi|}}
\newcommand{\alpham}{{|\alpha|}}

\newcommand{\hB}{\widehat{B}}
\newcommand{\hr}{\hat{r}}
\newcommand{\Ctil}{{\widetilde{C}}}
\newcommand{\Atil}{{\widetilde{A}}}
\newcommand{\Ahat}{{\widehat{A}}}
\newcommand{\tgamma}{{\tilde{\gamma}}}
\newcommand{\vd}{{v_\delta}}

\newcommand{\wdel}{{w_\delta}}
\newcommand{\lds}{{\lambda_\delta^*}}
\newcommand{\ld}{{\lambda_\delta}}
\newcommand{\uds}{{u_\delta^*}}
\newcommand{\Xd}{{X_\delta}}

\newcommand{\Creg}{C_{\mathrm{reg}}}
\newcommand{\Cinterp}{C_{\mathrm{interp}}}
\newcommand{\Cinterpone}{C_{\mathrm{interp}, 1}}
\newcommand{\Cinterptwo}{C_{\mathrm{interp}, 2}}
\newcommand{\Cinterpthree}{C_{\mathrm{interp}, 3}}
\newcommand{\CSinterp}{C_{W, \mathrm{interp}}^S}
\newcommand{\CBinterp}{C_{W, \mathrm{interp}}^{\widehat{B}}}
\newcommand{\Ccomm}{C_{\mathrm{com}}}
\newcommand{\pstar}{{p_\star}}
\newcommand{\gammastar}{\gamma_\star}
\newcommand{\Rstar}{{R_\star}}

\newcommand{\Wper}{W_{\mathrm{per}}}
\newcommand{\Wloc}{W_{\mathrm{loc}}}

\newcommand{\Hper}{H_{\mathrm{per}}}
\newcommand{\Lper}{L_{\mathrm{per}}}
\newcommand{\Layer}{\mathcal{T}}
\newcommand{\R}{\mathbb{R}}
\newcommand{\Z}{\mathbb{Z}}

\renewcommand{\emptyset}{\varnothing}

\graphicspath{{./figures/}}
\pgfplotstableset{
    columns/N/.style={int detect,column type=r,column name=$N$},
    columns/eDG/.style={
	sci sep align, precision=2,
	column name=$\|u-u_\delta\|_{\mathrm{DG}}$,
      },
    columns/eL2/.style={
	sci sep align, precision=2,
	column name=$\|u-u_\delta\|_{L^2(\Omega)}$,
      },
    columns/eLinf/.style={
	sci sep align, precision=2,
	column name=$\|u-u_\delta\|_{L^\infty(\Omega)}$,
    },
    columns/elambda/.style={
	sci sep align, precision=2,
        column name=$|\lambda - \lambda_\delta|$,
      },
    columns/bL2/.style={
	sci sep align, precision=2,
	column name=$b_{L^2}$,
      },
    columns/bDG/.style={
	sci sep align, precision=2,
	column name=$b_{\mathrm{DG}}$,
    },
    columns/bLinf/.style={
	sci sep align, precision=2,
	column name=$b_{L^\infty}$,
    },
    columns/blambda/.style={
	sci sep align, precision=2,
	column name=$b_{\lambda}$,
    },
    columns/slope/.style={
	sci sep align, precision=4,
	column name=$\mathfrak{s}$,
    },
    empty cells with={--}, 
    every head row/.style={before row=\toprule,after row=\midrule},
    every last row/.style={after row=\bottomrule}
  }

\usepackage[foot]{amsaddr}
\newcommand{\refcite}[1]{\cite{#1}}

\author{Yvon Maday$^{\star}$}
\address[$^\star$]{Sorbonne Université, CNRS, Université de Paris, Laboratoire Jacques-Louis Lions (LJLL), F-75005 Paris, France \and
  Institut Universitaire de France}
\email{yvon.maday@sorbonne-universite.fr}
\author{Carlo Marcati$^\diamondsuit$}
\address[$^\diamondsuit$]{Dipartimento di Matematica, Universit\`a di Pavia, I-27100 Pavia, Italy}
\email{carlo.marcati@unipv.it}
\thanks{Yvon Maday acknowledges funding from the European Research Council (ERC) under the European Union's Horizon 2020 research and innovation programme (grant agreement No 810367).}
\title[Analyticity and $hp$ dG approximation of nonlinear Schrödinger
eigenproblems]{Analyticity and $hp$ discontinuous
  Galerkin approximation of nonlinear Schrödinger
  eigenproblems}
\begin{document}
\begin{abstract}
 We study a class of nonlinear eigenvalue problems of Schr\"{o}dinger type,
 where the potential is singular on a set of  points. Such problems are
 widely present in physics and chemistry, and their analysis is of both
 theoretical and practical interest.
 In particular, we study the regularity of the eigenfunctions of the operators considered, and we propose and
 analyze the approximation of the solution via an isotropically refined $hp$ discontinuous
 Galerkin (dG) method.

 We show that, for weighted analytic potentials and
 for up-to-quartic polynomial nonlinearities, the eigenfunctions belong to analytic-type non
 homogeneous weighted Sobolev spaces. We also prove quasi optimal \emph{a
   priori} estimates on the error of the dG finite element method; when using an isotropically
 refined $hp$ space the numerical solution is shown to converge with
 exponential rate towards the exact eigenfunction.  We conclude with a series of numerical tests to validate the
 theoretical results.
\end{abstract}

\maketitle
\tableofcontents
\section{Introduction}
\label{sec:introduction}
This paper concerns the analysis of an elliptic nonlinear eigenvalue
problem and its approximation with an $hp$ discontinuous Galerkin finite element
method.
Specifically, we consider the problem of finding, in a domain $\Omega\subset
\mathbb{R}^d$, $d=2,3$, the smallest eigenvalue and associated eigenfunction $(\lambda, u)$ such that $\|u\|_{L^2(\Omega)}=1$ and
\begin{equation}
  \label{eq:intro-prob}
  (-\Delta  + V + f(u^2))  u = \lambda u,
\end{equation}
for a (singular) potential $V$ and a nonlinearity $f$. We refer to
  Ref.~\refcite{Cances2010} and the references therein for a discussion of problem
  \eqref{eq:intro-prob} and its connections with models in physics.
  In this manuscript, we will
  perform a theoretical analysis of \eqref{eq:intro-prob} with periodic boundary
conditions, and we will investigate the behavior of the numerical scheme with
homogeneous Dirichlet boundary conditions.

 We prove weighed analytic regularity and exponential convergence of
   the $hp$ method for nonlinearities $f(u^2)u  = |u|^{n-1} u$ with $n=2, 3,
     4$, i.e.,  an up-to-quartic polynomial nonlinearity.
   Problems of this kind correspond to the Euler-Lagrange equations
of energy minimization problems and are therefore widely present in physics and
chemistry.
Equations of the form \eqref{eq:intro-prob} are also often referred to as nonlinear
Schr\"odinger equations.
Problem \eqref{eq:intro-prob} also constitutes a model for problems of wide
  interest in quantum chemistry, such as the Hartree-Fock model. The latter is,
  indeed, a nonlinear elliptic eigenvalue problem, with a singular potential,
  and with a nonlinearity that is cubic in nature. The main difference with what
we consider here is that our nonlinearity is local, and that we only consider
the smallest eigenvalue.

Our analysis is centered mainly on potentials that are singular at a set of
isolated points; this includes the electric attraction generated by a
Coulomb potential, i.e., $V(x) = 1/d(x,\fc)$, where $d(\cdot, \cdot)$ is
  the euclidean distance between two points in $\mathbb{R}^d$, for some fixed point $\fc\in\Omega$, but
applies more generally to any potential that, in the vicinity of the singular
point, behaves as
\begin{equation}
  \label{eq:intro-V}
  V(x) \sim \frac{1}{d(x,\fc)^{\xi}},
\end{equation}
for a $\xi < 2$. Clearly, $V$ is not very regular in classical Sobolev spaces, thus we
cannot expect the solution to be regular in those spaces either. Nonetheless, we
can alternatively work in weighted Kondrat'ev-Babu\v{s}ka spaces, prove that the solution is
sufficiently regular in these spaces, and thus design an appropriate \hp{}
discretization that converges exponentially to the exact solution, see
Ref.~\refcite{Schotzau2013b}.

The nonlinear Schrödinger equation \eqref{eq:intro-prob} and the weighted spaces are
introduced in detail in Section \ref{sec:statement}. There, we also introduce our basic
assumptions on the nonlinearity, which are similar to those introduced in
Ref.~\refcite{Cances2010}, and on the potential $V$. As the analysis progresses, we will
introduce more restrictive hypotheses.

In Section \ref{sec:apriori}, we then prove \emph{a priori} convergence estimates on the
eigenvalue and eigenfunction. We consider a wider class of nonlinearities
  than polynomial ones,  as detailed in equations \eqref{eq:cond1} to \eqref{eq:cond4}.
Even though our focus is on \hp{} methods, the
abstract convergence proofs in Section 3 do not exploit specific features of
\hp{} methods, and cover also other discretizations. Suppose we consider a simpler $h$-type finite element
method: the proof of Theorem \ref{theorem:convergence1} --- i.e., convergence and
quasi optimality of the numerical solution --- holds, since we do not use
any specific feature of \hp{} refinement. The proof of convergence of the
discontinuous Galerkin method for a nonlinear eigenvalue problem of the form
\eqref{eq:intro-prob} is a new result as far as we are aware. Previous results
include the convergence of the discontinuous Galerkin method for linear
eigenproblems \cite{Antonietti2006} and the convergence of conforming methods
for the nonlinear problem \cite{Cances2010}.
The main difference with the latter paper is that the
discontinuous Galerkin method is not conforming, thus some relations between
exact and numerical quantities, e.g., between the exact eigenvalue $\lambda$ and
the numerical one $\ld$, are less straightforward.
In general, the convergence
  and quasi optimality of the numerical eigenvalue--eigenfunction pair proven in
Theorem \ref{theorem:convergence1} should be readily extendable to any nonconforming symmetric
method such that the thesis of Lemma \ref{lemma:estimates}, akin to coercivity
and continuity of the numerical bilinear form, holds.

In Section \ref{chap:nonlinear-regularity}, we restrict the
analysis to the case of polynomial, up-to-quartic polynomial nonlinearities. In this
setting, the solutions of problem \eqref{eq:intro-prob} are analytic in weighted
Sobolev spaces: specifically, if the potential is of type \eqref{eq:intro-V}
and $f(u^2) = |u|^{n-1} $, $n= 2, 3, 4$, then for $\gamma< d/2 +2
-\xi$ there exist constants
$C_u$ and $A_u$ such that for all $k\in \mathbb{N}$
\begin{equation*}
  \sum_{\alpham = k}\| d(x, \fc)^{k - \gamma} \dalpha u \|_{L^2(\Omega)} \leq C_u A_u^{k}k!,
\end{equation*}
where $u$ is the ground state of \eqref{eq:intro-prob} and $\alpha \in
\mathbb{N}_0^d$ is a multi index. This weighted, analytic regularity estimate is proven in Theorem
\ref{theorem:analytic} and constitutes a novel result of independent interest.
For previous weighted analytic
regularity results for elliptic problems, we
refer, among others, to Refs.~\refcite{Guo2006b,Costabel2012a}.

As a consequence of the quasi-optimal \emph{a priori} estimates introduced
above and of the weighted analytic regularity of the ground state for
up-to-quartic nonlinearities, we obtain exponential convergence of the numerical
solution computed with the $hp$ dG method. This is briefly discussed in Section
\ref{sec:exp-convergence} and presented in Theorem \ref{theorem:eigenvalue-exponential}.

Finally, in Section \ref{sec:numerical}, we investigate the
performance of the scheme in two and three dimensional numerical tests. We
confirm our theoretical estimates, while also showing the effect of sources of
numerical error that have not been taken into consideration in the theoretical analysis.

\section{Statement of the problem and notation}
\label{sec:statement}
\subsection{Functional setting and notation}
Let $\Omega = \mathbb{R}^d/(L\mathbb{Z}) ^d$ be the flat torus of period 
$0 < L\leq 1$ and let $S= (0, L)^d$.
We use the standard notation for Sobolev spaces, that we indicate, for any
  non periodic set
  $D\subset \mathbb{R}^d$, as
$W^{k,p}(D)$, with $W^{k,2}(D) = H^k(D)$ and $W^{0,p}(D) =
L^p(D)$. We write $v\in \Wloc^{k, p}(\R^d)$ if, for all bounded $D\subset
  \R^d$, $v\in W^{k, p}(D)$. Furthermore, we indicate periodized Sobolev spaces on $\Omega$
  by
  \begin{equation*}
    W^{k, p}(\Omega) = \Wper^{k, p}(S)  = \left\{ v_{|_{S}} : v\in \Wloc^{k, p}(\R^d) \text{ and } v_{|_{S+i L}} = v_{|_S}\text{ for all }i\in \Z^d \right\},
  \end{equation*}
with 
  $H^k(\Omega) = \Hper^k(S) = W^{k, 2}(\Omega)$ and $L^p(\Omega) =
    \Lper^p(S) = W^{0,p}(\Omega)$.
We denote the scalar product in $L^2(D)$, for $D\subset\R^d$, as
  $(\cdot, \cdot)_{D}$, the scalar product in
$\Lper^2(\Omega)$ as $(\cdot, \cdot)$, and the respective norm
as $\lVert u\rVert = (u,u)^{1/2}$.  Furthermore, we denote by
  $\mathbb{N}$ the set of strictly positive natural numbers, and by
  $\mathbb{N}_0 = \mathbb{N}\cup\{0\}$.
For two quantities $a$ and $b$, we write $a\lesssim
b$ (respectively $a\gtrsim b$) if there exists $C>0$ independent of the
discretization, such that $a\leq C b$ (resp. $a\geq Cb $). We write $a\simeq b$
if $a\lesssim b$ and $a\gtrsim b$.

We now recall the definition of the weighted Sobolev spaces, introduced in Ref.~\refcite{Kondratev1967}, that will be central to our regularity analysis.
Given a set of isolated points ${\mathfrak{C}} \subset S$, we
write $\widetilde{\mathfrak{C}} = \{\mathfrak{C} + kL, \, k\in \Z^d\}$.
we introduce the homogeneous Kondrat'ev-Babu\v{s}ka space
$\mathcal{K}^{k,p}_\gamma (\Omega, \mathfrak{C})$, defined as \begin{multline*}
  \mathcal{K}^{k,p}_\gamma (\Omega,\mathfrak{C})=\lbrace u_{|_S} : u \in
    \Wloc^{k, p}(\R^d\setminus \widetilde{\mathfrak{C}}), \, u_{|_S} = u_{|S+iL}
    \, \forall i\in \Z^d,\\
    \text{ and } r^{|\alpha|-\gamma}(\partial^\alpha u)_{|_S} \in L^p(S)\; \forall \alpha \in \mathbb{N}_0^d : |\alpha| \leq  k\rbrace,
\end{multline*}%
where $r= r(x)$ is any smooth function defined on $S$ which is, in the vicinity of every point
$\fc\in \mathfrak{C}$, equal to the euclidean distance $d(x,\fc)$ from the
point and nonzero elsewhere. For example,
  \begin{equation*}
    r(x) = \prod_{\fc\in\fC} d(x, \fc), \qquad \forall x \in S.
  \end{equation*}The
nonhomogeneous Kondrat'ev-Babu\v{s}ka space is defined by \begin{multline*}
\mathcal{J}^{k,p}_\gamma (\Omega,\mathfrak{C})=\lbrace u_{|_S}: u \in \Wloc^{\lfloor
  \gamma-d/p \rfloor, p}(\R^d) \cap \Wloc^{k,p}(\R^d\setminus \widetilde{\fC}), \, u_{|_S} = u_{|_{S+iL}} \,\forall i\in \Z^d\\
\text{and }r^{\alpham-\gamma}(\partial^\alpha u )_{|_S}\in L^p(S)\; \forall \alpha \in \mathbb{N}^d : \lfloor \gamma-d/p \rfloor +1 \leq |\alpha| \leq  k\rbrace,
\end{multline*}
for $\gamma>d/p$.
We define the associated seminorm as
\begin{equation*}
|u|_{\mathcal{J}^{k,p}_{\gamma}(\Omega)}
=|u|_{\mathcal{K}^{k,p}_{\gamma}(\Omega)} =\sum_{\alpham = k} \|
r^{\alpham-\gamma}\partial^{\alpha}u\|_{L^p(S)} .
\end{equation*}
We also introduce the
spaces of regular functions with weighted analytic type estimates as
\begin{equation*}
\mathcal{K}^{\varpi,p}_\gamma (\Omega,\mathfrak{C})=\{ v \in \mathcal{K}^{\infty,p}_{\gamma}(\Omega,\mathfrak{C}): |v |_{\mathcal{K}^{k,p}_\gamma(\Omega)} \leq CA^{k}k! \;\forall k\in \mathbb{N}_0\},
\end{equation*}
and 
\begin{equation*}
\mathcal{J}^{\varpi,p}_\gamma (\Omega,\mathfrak{C})=\{ v \in \mathcal{J}^{\infty,p}_{\gamma}(\Omega,\mathfrak{C}): |v |_{\mathcal{K}^{k,p}_\gamma(\Omega)} \leq CA^{k}k!, \;\forall k \in \mathbb{N}_0:k\geq \lfloor \gamma-d/p \rfloor +1\},
\end{equation*}
where $\mathcal{K}^{\infty,p}_{\gamma} (\Omega)= \bigcap_k \mathcal{K}^{k,p}_{ \gamma}(\Omega)$,
$\mathcal{J}^{\infty,p}_\gamma(\Omega)$ defined similarly. To simplify the
notation, we will suppose that there is only one singular point per cell
$S$, i.e.,
${\mathfrak{C}}=\{\fc\}$ and omit $\mathfrak{C}$ from the notation of the spaces.
Furthermore, we write $\mathcal{K}^{k}_\gamma (\Omega)=
\mathcal{K}^{k,2}_\gamma(\Omega)$, $\mathcal{J}^{k}_\gamma (\Omega)=
\mathcal{J}^{k,2}_\gamma(\Omega)$, $\mathcal{K}^{\varpi}_\gamma (\Omega)=
\mathcal{K}^{\varpi,2}_\gamma(\Omega)$, and $\mathcal{J}^{\varpi}_\gamma (\Omega)= \mathcal{J}^{\varpi,2}_\gamma(\Omega)$.
For a thorough treatment of Kondrat'ev-Babu\v{s}ka spaces, see
Refs.~\refcite{Kozlov1997,Costabel2010b,Costabel2010a,Costabel2012a}.
Note that the results obtained in the sequel can be trivially extended to the
case where ${\mathfrak{C}}$ contains more than one point, as long as 
${\mathfrak{C}}$ is a finite set of points.
When, for any non periodic set $D\subset\mathbb{R}^d$, we refer to the
  spaces $\mathcal{J}^{k,p}_\gamma(D)$, we implicitly refer to their non
  periodized version.

Finally, let $X = \mathcal{J}^2_\gamma(\Omega)$, for $\gamma \in (d/2, d/2
+\varepsilon)$, where $0<\varepsilon<1$ will be
specified later, namely in hypothesis \eqref{eq:hyp-potential-2}.

\begin{remark}
    We perform our analysis in the torus $\Omega = \mathbb{R}^d/(L\mathbb{Z})^d$
    to avoid having to deal with the singularities that arise at the edges of
    polyhedral domains. A theory of weighted analytic regularity for nonlinear
    elliptic problems in polyhedral domains is, indeed, not available in the
    literature; the analysis we perform here can be also applied to corner
    singularities of solutions to nonlinear elliptic problems in polyhedral domain, and
    can be therefore seen as a first building block for the analysis of
    nonlinear elliptic problems in polyhedral domains.
  \end{remark}

\subsection{Statement of the problem}
\label{sec:statement-sub}
We introduce the problem under consideration. From the 
``physical'' point of view, it consists in the minimization of an energy
  composed by a kinetic term, an interaction
with a singular potential $V$ and a nonlinear self-interaction term. Under
  the unitary norm constraint, using Euler's equation, the energy minimization problem translates into a nonlinear
elliptic eigenvalue problem. This is the form under which
most of the analysis will be carried out.

We start therefore by introducing the bilinear form over $H^1(\Omega)\times H^1(\Omega)$
\begin{equation}
  \label{eq:bilin-cont}
  a(u,v) = \int_\Omega\nabla u\cdot \nabla v + \int_\Omega V u v
\end{equation}
and a function $F:\mathbb{R}^+\to \mathbb{R}$, whose
properties have been introduced in Section \ref{sec:introduction}.
Let
\begin{equation}
  \label{eq:energy-cont}
  E(v) = \frac{1}{2}a(v,v) + \frac{1}{2} \int_\Omega F(v^2) .
\end{equation}
Let us denote by $u$ the minimizer of \eqref{eq:energy-cont} (unique up to a
sign change under the hypotheses that follow) over the space $\{v \in H^1(\Omega):\lVert v\rVert=1\}$: then, there exists
$\lambda \in \mathbb{R}$ such that $u$ is the solution of
\begin{equation}
  \label{eq:eq-cont}
  \langle A^u u - \lambda u, v \rangle = 0\quad \forall v \in H^1(\Omega)
\end{equation}
where for given $u\in H^1(\Omega)$, $A^u\in \mathcal{L}(H^1(\Omega),
(H^1(\Omega))')$ is defined as
\begin{equation*}
  \langle A^uv,w \rangle= a(v,w) + \int_\Omega f(u^2)vw,
\end{equation*}
with $f=F'$. We introduce also
\begin{equation}
  \label{eq:E''}
  \langle E''(u)v,w \rangle = \langle A^uv,w\rangle + 2 \int_\Omega f'(u^2)u^2vw,\qquad \forall v,w\in H^1(\Omega).
\end{equation}
 The properties of the function $F$ will be similar to those in
 Ref.~\refcite{Cances2010} and have already been introduced above. We recall
 them here:
 \begin{subequations}
   \begin{align}
     \label{eq:cond1} &F\in C^1([0,+\infty),\mathbb{R})\cap C^\infty((0,+\infty),\mathbb{R})\text{ and }F''>0\text{ in }(0,+\infty),\\
     \label{eq:cond2} &\exists q\in[0,2), \exists C\in\mathbb{R} : \forall t\geq 0, |F'(t)| \leq C(1+t^q), \\
     \label{eq:cond3} &F''(t)t \text{ locally bounded in } [0,+\infty), 
 \end{align}
and we suppose that $\forall R>0, \exists C_R\in\mathbb{R}_+: \forall t_1\in(0,R], \forall t_2 \in \mathbb{R}$,
\begin{equation}
  \label{eq:cond4}
\lvert F'(t_2^2)t_2 - F'(t_1^2)t_2 - 2F''(t_1^2)(t_1^2)(t_2-t_1)\rvert \leq C_R(1+|t_2|^s)|t_2-t_1|^r
\end{equation}
\end{subequations}
 for $r\in(1,2]$ and $s\in [0, 5-r)$.
 As an example of nonlinearities satisfying \eqref{eq:cond1} to \eqref{eq:cond4},
   we mention, following Ref. \refcite{Cances2010}, functions modeling
   repulsive interaction in Bose--Einstein condensates ($F(t) \propto t^2$) and
   the Thomas--Fermi kinetic energy functional ($F(t) \propto t^{5/3}$). We refer
 the reader to the discussion in the mentioned reference.
 
Finally, we suppose that the potential $V$ is such that 
\begin{subequations}
  \begin{equation}
  \label{eq:hyp-potential-1}
V\in L^{p_V}(\Omega)
\end{equation}
with $p_V > \max(1, d/2)$ and that there exists $0<\varepsilon < 1$ such that
\begin{equation}
  \label{eq:hyp-potential-2}
  V \in \mathcal{K}^{\varpi,{}\infty}_{-2+\varepsilon}(\Omega, \mathfrak{C}) .
\end{equation}
\end{subequations}
 For $d=2,3$, \eqref{eq:hyp-potential-2}
 implies \eqref{eq:hyp-potential-1} as long as $p_V < d/(2-\epsilon)$. A
 consequence of \eqref{eq:hyp-potential-1} is, in particular, that for $u, v\in H^1(\Omega)$,
\begin{equation*}
  (Vu,v) \leq C \|u\|_{H^1(\Omega)}\|v\|_{H^1(\Omega)}.
\end{equation*}
where the constant $C$ depends on $V$ and on the domain.
\begin{remark}
    As an example of a class of potentials satisfying
    \eqref{eq:hyp-potential-2}, we mention the functions $V:\R^d \to \R$ such that
    \begin{equation*}
      V(x) = \sum_{k\in\Z^d} \frac{1}{\exp(d(x,k)^{\beta}) - 1},
    \end{equation*}
    for $\beta\in(0,2)$.
  \end{remark}

We have also the
following regularity result, which follows from \eqref{eq:cond2} and
\eqref{eq:hyp-potential-2} and the regularity result obtained in Ref.~\refcite{linear}.

\begin{lemma}
  \label{lemma:u-reg}
  The solution $u$ to \eqref{eq:eq-cont}, under hypotheses \eqref{eq:cond1} to \eqref{eq:hyp-potential-2},
 belongs to the space
\begin{equation}
  \label{eq:u-reg}
  u \in \mathcal{J}^2_{d/2+\alpha}(\Omega)
\end{equation}
for any $\alpha<\varepsilon$, with $\varepsilon$ as in \eqref{eq:hyp-potential-2}.
\end{lemma}
\begin{proof}
  We have $u\in L^\infty(\Omega)$, see Ref.~\refcite{Cances2010}. Hence, $u$ is the
  solution of $\left( -\Delta + V \right) u = \lambda u + f(u^2)u$,
  with right hand side belonging to the space
  $\mathcal{J}^0_{d/2+\epsilon-2}(\Omega)$. Since the operator $-\Delta +V$ is
  an isomorphism from $\mathcal{J}^2_{d/2+\alpha}(\Omega)$ to
  $\mathcal{J}^0_{d/2+\alpha-2}(\Omega)$ for $0<\alpha< \epsilon$, see Ref.~\refcite{linear}, we obtain the thesis.
\end{proof}

\subsection{Numerical method}
\label{sec:numerical-method}
In this section we introduce the \hp{} discontinuous Galerkin method.
Concerning the design of the \hp{} space, the setting is the one from
Refs.~\refcite{Guo1986a,Guo1986b}.
Let $\mathcal{T}$ be a triangulation of axiparallel quadrilateral ($d=2$) or
  hexahedral ($d=3$) elements of $\Omega$, such that $\bigcup_{K\in
  \mathcal{T}} \overline{K}= \overline{S}$, whose properties will be
specified later. A $d-1$ dimensional face (edge, when $d=2$) is defined as the nonempty interior
of $\partial K_{\sharp}\cap \partial K_{\flat}$ for two adjacent elements
$K_{\flat}$ and $K_\sharp$. Let $\mathcal{E}$ be the set of all faces/edges.
 We denote 
\begin{equation*}
  (u,v)_\mathcal{T} = \sum_{K \in \mathcal{T}} (u, v)_K
\end{equation*}
and, similarly,
\begin{equation*}
  (u,v)_\mathcal{E} = \sum_{e \in \mathcal{E}} (u, v)_e.
\end{equation*}

We suppose that for any $K\in\mathcal{T}$ there exists
an affine transformation $\Phi:K\to\hK$ to the $d$-dimensional cube $\hK$ such
that $\Phi(K) = \hK$, that the mesh is shape and
contact regular.\footnote{If $h_K$ is the diameter of an element
  $K\in\mathcal{T}$ and $\rho_K$ is the radius of the largest ball inscribed in
  $K$, a mesh sequence is \emph{shape regular} if there exists $C$ independent of the
  refinement level such that $h_K\leq C r_K$ for all $K\in \mathcal{T}$. 
  The mesh sequence is \emph{contact regular} if for all
  $K\in\mathcal{T}$, the number of elements adjacent to $K$ is uniformly bounded
and there exists a constant $C$ independent from the refinement level such that
for every face/edge $e$ of $K$, $h_K\leq C h_e$, where $h_e$ is the diameter of $e$.}

We now introduce meshes that are isotropically and geometrically graded around the
points in $\mathfrak{C}$: for simplicity, we consider the case where $\fC
  = \{\fc\}$. Then, we
  fix a refinement ratio $\sigma\in (0, 1/2)$ and, for all $\ell\in
    \mathbb{N}$, we introduce a 
  mesh $\mathcal{T}^{\ell}$ that can be partitioned into disjoint mesh layers $\Layer^\ell_j$, $j=1,
  \dots, \ell$ such that $\mathcal{T}^{\ell} = {\bigcup}_{j=1}^\ell \Layer^\ell_j$.
  For all $K\in \Layer^\ell_j$, we suppose that
\begin{equation*}
  h_K \simeq h_j = \sigma^{j} \qquad d(\fc, K) \simeq h_K,
\end{equation*}
for all $j=1, \dots, \ell$ and with constants uniform in $\mathcal{T}$ and $\ell$.
Finally, we suppose that mesh refinement happens only at the singularity,
  i.e., given
   $ \mathcal{T}^\ell = \Layer^\ell_1 \cup\ldots \cup \Layer^\ell_{\ell}$,
  then
    $\mathcal{T}^{\ell+1} = \Layer^\ell_1 {\cup}\ldots {\cup}\Layer^\ell_{\ell-1}{\cup} \Layer^{\ell+1}_{\ell}{\cup}\Layer^{\ell+1}_{\ell+1}$.
The generalization to the case of $\fC$ containing multiple points follows from
the construction of a graded mesh around each point.

We will allow for $1$-irregular edges/faces, i.e., given two neighboring
elements $K_\flat$ and $K_\sharp$, that share an edge/face $e =
\partial{K}_\sharp\cap \partial{K}_\flat$, we require that $e$ is an entire
edge/face of at least one between $K_\sharp$ and $K_\flat$. We refer to Section
  \ref{sec:numerical} (specifically, to Figure \ref{subfig:2d-mesh}) for a
  visualization of such a mesh.
We introduce on this mesh the \hp{} space with  linear polynomial slope $s$, i.e., for an element $K\in \mathcal{T}^{\ell}$ such that $K\in\Layer^{\ell}_j$,
\begin{equation*}
   p_K \simeq p_j = p_0+\slope(\ell-j),
\end{equation*}
where $h_K$ is the diameter of the element $K$ and $p_K$ is the polynomial order
whose role will be specified in \eqref{eq:discrete-space}. We introduce the
discretization parameter $0 < \delta\leq 1$ such that $\delta \to 0$ when $\ell
\to \infty$, and
the discrete space 
\begin{equation}
  \label{eq:discrete-space}
  \de{X} = \left\{ \de{v} \in L^2(\Omega):(v_{|_K}\circ\Phi^{-1})\in \mathbb{Q}_{p_K}(\hK),\, \forall K\in\mathcal{T}^{\ell}\right\},
\end{equation}
where $\mathbb{Q}_p$ is the space of polynomials of maximal degree $p$ in any
variable and denote
  \begin{equation*}
    X(\delta) = X + \Xd.
  \end{equation*}
Then, $\mathcal{E}^{\ell}$ is the set of the edges (for $d=2$) or faces ($d=3$) of the elements in $\mathcal{T}^{\ell}$ and
\begin{equation*}
  \dhe= \min_{K\in \mathcal{T}^{\ell}: e\cap\partial K \neq \myempty}h_K , \qquad
  \dpe = \max_{K\in \mathcal{T}^{\ell}: e\cap\partial K \neq \myempty}p_K .
\end{equation*}
On an edge/face between two elements $K_\sharp$ and $K_\flat$, i.e., on
$e\subset\partial{K}_\sharp\cap \partial{K}_\flat$, the average $\avg{\cdot}$
and jump $\jump{\cdot}$ operators for a function $w\in X(\delta)$  are defined by
\begin{equation*}
  \avg{w} = \frac{1}{2}\left( w_{|_{K_\sharp}}+w_{|_{K_\flat}}\right), \qquad \jump{w} =   w_{|_{K_\sharp}}\mathbf{n}_\sharp+w_{|_{K_\flat}}\mathbf{n}_\flat,
\end{equation*}
where $\mathbf{n}_\sharp$ (resp. $\mathbf{n}_\flat$) is the outward normal to
the element $K_\sharp$ (resp. $K_\flat$).
If $e$ is an edge/face of an element $K$ and it lies on part of the boundary
  where homogeneous Dirichlet boundary
  conditions are imposed (as it will be the case in Section
  \ref{sec:numerical}),  the expression above is replaced by
\begin{equation*}
  \avg{w} =  w_{|_{K}}, \qquad \jump{w} =   w_{|_{K}}\mathbf{n},
\end{equation*}
where $\mathbf{n}$ is the normal pointing outwards from $K$.

We now introduce the discrete versions of the operators defined in Section \ref{sec:statement-sub}.
First, the bilinear form $a_\delta$ over $X_\delta \times X_\delta$ is given by
\begin{equation}
  \begin{aligned}
  \label{eq:bilin-discont}
  a_\delta (u_\delta,v_\delta) = (\nabla u_\delta, \nabla v_\delta)_{\mathcal{T}^{\ell}} &- (\avg{\nabla u_\delta},\jump{v_\delta})_{\mathcal{E}^{\ell}}- (\avg{\nabla v_\delta},\jump{u_\delta})_{\mathcal{E}^{\ell}} \\ & + \sum_{e \in \mathcal{E}^{\ell}}\alpha_e \frac{\dpe^2}{\dhe} (\jump{u_\delta}, \jump{v_\delta})_{e} + \int_\Omega V u_\delta v_\delta.
\end{aligned} 
\end{equation}
Furthermore,
\begin{equation}
  \label{eq:energy-discont}
  \de{E}(\de{v}) = \frac{1}{2}\de{a}(\de{v},\de{v}) + \frac{1}{2} \int_\Omega F(\de{v}^2).
\end{equation}
Let $u_\delta$ be a minimizer of
\eqref{eq:energy-discont} over $\Xd$, with unitary norm constraint. Then, there exists an eigenvalue $\lambda_\delta\in
\mathbb{R}$ such that
\begin{equation}
  \label{eq:eq-discont}
  \langle A_\delta^{u_\delta} u_\delta - \lambda_\delta u_\delta, v_\delta \rangle = 0\quad \forall v_\delta \in X_\delta
\end{equation}
where 
\begin{equation*}
    \langle A_\delta^{u_\delta} v_\delta , w_\delta \rangle  = a_\delta (v_\delta, w_\delta) + \int_\Omega f(u_\delta^2)v_\delta w_\delta.
\end{equation*}
Finally, $E_\delta''$, defined on $\Xd$, is
obtained by replacing $A^u$ with $\Ad^u$ in \eqref{eq:E''}.

\begin{remark}[Symmetry of the numerical method]
  The dG method with bilinear form \eqref{eq:bilin-discont} is the
  symmetric interior penalty (SIP) method.
  The requirement of symmetry in the bilinear form of the numerical method is a strong
one, and will be used without explicit mention throughout the proofs.

This could be seen as a limitation; nonetheless, from a practical point of view, there is little interest in 
approximating a symmetric eigenvalue problem with a non symmetric numerical
method.
Non symmetric methods tend to exhibit, in the linear case, lower rates of convergence than symmetric
ones \cite{Antonietti2006}.
Furthermore, the solution of the finite dimensional problem
would be more problematic, since algebraic eigenvalue problems are more easily
treated for symmetric matrices \cite{Saad2011}.
\end{remark}
We introduce the mesh dependent norms that will be used in this section. 
First, for a $v\in X(\delta)$,
\begin{equation}
  \label{eq:dgnorm}
  \dgnorm{v}^2 = \sum_{K\in \mathcal{T}^{\ell}} \| v\|_{\mathcal{J}^1_1(K)}^2 + \sum_{e\in\mathcal{E}^{\ell}} \dhe^{-1}{\dpe^2} \|  \jump{v}\|_{L^2(e)}^2 .
\end{equation}
Remark that on $X$, this norm is equivalent to the
$\mathcal{J}^1_1(\Omega)=H^1(\Omega)$ norm, since functions in $X$ have no face
discontinuity, implying $\jump{v} = 0$.
Then, on $X(\delta)$ we introduce, when $d=3$,
\begin{equation}
  \label{eq:fulldgnorm}
  \fullnorm{u}^2 
= \sum_{K\in \mathcal{T}^{\ell}} 
\| u\|_{H^1(K)}^2
+  \sum_{e\in\mathcal{E}^{\ell}} \dpe^{2} \dhe^{-1}\|  \jump{u}\|_{L^2(e)}^2  + \sum_{e\in\mathcal{E}^{\ell}}  \dpe^{-2}\| r^{1/2} \nabla u \cdot n_e \|_{L^2(e)}^2 ,
\end{equation}
where $n_e$ denotes the normal to face $e$.
If $d=2$, we denote by $\mathcal{E}^{\ell}_c$ the set of edges abutting at the
singularity, and write (note that on $\mathcal{E}^{\ell}_c$, $\dpe = p_0$)
\begin{multline}
  \label{eq:fulldgnorm-2d}
  \fullnorm{u}^2 
= \sum_{K\in \mathcal{T}^{\ell}} 
\| u\|_{H^1(K)}^2
+  \sum_{e\in\mathcal{E}^{\ell}} \dpe^{2} \dhe^{-1}\|  \jump{u}\|_{L^2(e)}^2  
+ \sum_{e\in\mathcal{E}^{\ell}\setminus \mathcal{E}^{\ell}_c}  \dpe^{-2}\| r^{1/2} {\nabla
  u}\cdot n_e\|_{L^2(e)}^2
\\ + \sum_{e\in \mathcal{E}^{\ell}_c}  \dhe^{q-1}\| \nabla u\cdot n_e\|_{L^q(e)}^q,
\end{multline}
where $q$ is fixed and such that $1<q<2/(3-\gamma)$, see Remark \ref{remark:normbound}.

For any triangulation $\mathcal{T}$ of $\Omega$, let us also introduce the broken space
\begin{equation*}
\mathcal{J}^{s,p}_\gamma(\Omega, \mathcal{T}) = \left\{ v : v\in\mathcal{J}^{{s,p}}_\gamma(K), \forall K\in \mathcal{T} \right\}.
\end{equation*}
\begin{remark}
  \label{remark:normbound}
When $d=3$, by the definition of the weighted spaces, see Ref. \refcite{Mazya2010}, for
$e \subset\partial K$, and since $\gamma > 3/2$,
\begin{equation*}
  \| r^{1/2}\nabla v \|_{L^2(e)} \leq C\| v \|_{\mathcal{J}^{2}_\gamma(K)},
\end{equation*}
then $\fullnorm{v}$ \eqref{eq:fulldgnorm} is bounded on $\mathcal{J}^2_\gamma(\Omega, \mathcal{T}^{\ell})$. Since furthermore $X(\delta)
\subset \mathcal{J}^{2}_\gamma(\Omega, \mathcal{T}^{\ell})$, $\fullnorm{v}$ as defined
in \eqref{eq:fulldgnorm} is
bounded when $d=3$ for all $v\in X(\delta)$.

When $d=2$, we consider the definition of the norm \eqref{eq:fulldgnorm-2d}. Let
$\gamma > 1$: then, for any $q < 2/(3-\gamma)$ and writing $t =
(1/q-1/2)^{-1}$, there exist $C_1, C_2$ such that, for all $v\in \mathcal{J}^2_\gamma(\Omega)$,
\begin{equation*}
  \sum_{\alpham=2}\| \dalpha v \|_{L^q(\Omega)} \leq C_1 \|r^{\gamma-2}\|_{L^t(\Omega)} \sum_{\alpham=2}\|r^{2-\gamma}\dalpha v\|_{L^2(\Omega)}\leq C_2\| v\|_{\mathcal{J}^2_\gamma(\Omega)}.
\end{equation*}
Hence, if $v\in \mathcal{J}^2_\gamma(\Omega)$ with $\gamma>1$, $v\in W^{2, q}(\Omega)$, and $\| \nabla v\cdot n\|_{L^q(e)}$ is well
defined. Therefore, $\fullnorm{v}$ as defined in \eqref{eq:fulldgnorm-2d} is bounded when
$d=2$ for all $v\in X(\delta)$.
\end{remark}

\begin{remark}
By proving continuity as in Lemma \ref{lemma:estimates} (see the part of the
proof referring to inequality \eqref{eq:estimates-lemma1b}) and thanks to Remark
\ref{remark:normbound}, it can be shown that the bilinear form $a_\delta$
defined in \eqref{eq:bilin-discont} over $\Xd\times\Xd$ can be extended over $X(\delta) \times X_\delta$. 
\end{remark}

\begin{remark}
  \label{remark:edge-elem}
Note that on $X_\delta$ and for $d\leq3$, the two norms \eqref{eq:dgnorm} and
\eqref{eq:fulldgnorm} are uniformly (with respect to the refinement level)
equivalent, since for any $K\in\mathcal{T}^{\ell}$, $r_{|_K}\lesssim h_K$ and
thanks to the discrete trace inequality \cite{DiPietro2011}
\begin{equation}
  \label{eq:edge-elem}
h_e^{(1-d)/p+d/2} \| w_\delta \|_{L^p(e)} \leq C_{d,p}   \| w_\delta \|_{L^2(K)},
\end{equation}
valid for $e\in\partial K$ and for all $w_\delta \in X_\delta$. The constant
$C_{d,p}$ depends on the dimension $d$, on $p$, and on the polynomial order
  $\dpe$, but is independent of $\dhe$. Furthermore, $C_{d,p}$ is bounded by
  $\dpe$ if $p=2$.
\end{remark}

\begin{remark}
    \label{remark:alphae}
    The coercivity of the SIP discrete bilinear form associated to the Laplacian, i.e., the existence of $c>0$ such that, for all $\vd\in\Xd$
    \begin{equation*}
      (\nabla v_\delta, \nabla v_\delta)_{\mathcal{T}} - (\avg{\nabla v_\delta},\jump{v_\delta})_{\mathcal{E}}- (\avg{\nabla v_\delta},\jump{v_\delta})_{\mathcal{E}} + \sum_{e \in \mathcal{E}}\alpha_e \frac{\dpe^2}{\dhe} (\jump{v_\delta}, \jump{v_\delta})_{e} \geq c \dgnorm{\vd}^2,
    \end{equation*}
    is verified if for all $e\in\mathcal{E}$, $\alpha_e > \alpha^* > 0$, for
    some $\alpha^*$ that depends on the mesh. This is shown in
    Theorem 4.4 of Ref.~\refcite{Schotzau2013a} for geometrically graded meshes;
    for quasi-uniform meshes, see
    Refs.~\refcite{alphae-tria,alphae-tria2,alphae-quad} for explicit
    expressions of $\alpha^*$.
  \end{remark}

We conclude this section by introducing the discrete approximation to the solution of the linear problem, i.e.~the function $u_\delta^*\in X_\delta$ such that
\begin{equation}
  \label{eq:uds-def}
\langle A_\delta^{u} u_\delta^* - \lambda_\delta^* u_\delta^*, v_\delta \rangle = 0\quad \forall v_\delta \in X_\delta
\end{equation}
for an eigenvalue $\lambda_\delta^*$. Note that, since $u$ is an eigenfunction
of $A^u$ and the associated eigenspace is of dimension $1$ \cite{Cances2010},
we have that
\begin{equation}
  \label{eq:convergence-linear}
  \begin{aligned}
    \lVert u_\delta^* - u \rVert_\mathrm{DG} &\lesssim \inf_{v_\delta\in X_\delta}\fullnorm{u-v_\delta},\\
    | \lambda_\delta^* - \lambda | &\lesssim \inf_{v_\delta\in X_\delta}\fullnorm{u-v_\delta}^2,
  \end{aligned}
\end{equation}
 see Ref.~\refcite{linear}, and the eigenspace associated with $u_\delta^*$ is of
 dimension one, for a sufficiently large number of degrees of freedom \cite{Antonietti2006}.

 The isotropically refined \hp{} finite element space $\Xd$ defined here
 provides approximations that converge with exponential rate to the function in
 the weighted analytic class, as stipulated in the following statement, see
 Theorem 4.2 and Proposition 5.13 of Ref.~\refcite{Schotzau2013b}.
 \begin{proposition}
   \label{prop:best-appx}
Let $\gamma > d/2$ and $v\in
\mathcal{J}^\varpi_\gamma(\Omega)$. There exists two
constants $C, b>0$  such that for all $\ell \in \mathbb{N}$
\begin{equation}
  \label{eq:hp-exponential}
  \inf_{\vd\in\Xd} \fullnorm{v-\vd} \leq C e^{ -b \ell }.
\end{equation}
Here, $\ell$ is the number of refinement steps, and $\ell = N^{1/(d+1)}$, with
$N$ denoting the number of degrees of freedom of $\Xd$.
\end{proposition}

\section{Abstract a priori error bounds}
\label{sec:apriori}
In this section we prove some a priori estimates on the convergence of the
numerical eigenfunction and eigenvalue. We start by giving some continuity and coercivity
estimates, then we provide an auxiliary estimate on a scalar product where
we construct an adjoint problem, and we conclude by proving convergence and
quasi optimality for the eigenfunctions. The rate of convergence proven for the
eigenvalues is smaller than what is obtained in the linear case: in the following
it will be shown that under some additional hypothesis we can recover the 
rate typically obtained in the approximation of solutions to linear elliptic
operators with singular potentials \cite{linear}.

Since our main focus here is on isotropically refined \hp{} methods, the
approach we take uses the assumption that finite element space and the
underlying mesh are those of an \hp{} discontinuous Galerkin method, as
described in the previous sections. It is important to remark, nonetheless, that
the results of this section can be extended, with minimal effort, to the
analysis of a general discontinuous Galerkin approximation. The novelty of
the approach we use in this section lies, indeed, more into the treatment of the
nonconformity of the 
method than in the aspects related to the \hp{} space. The 
modification necessary to get a proof that applies to a classical
$h$-type discontinuous Galerkin finite element method, for example, would be
related to the continuity and coercivity estimates, since those would need not to use
the hypothesis that $r\simeq h$.

For the aforementioned reason, and for the sake of generality, we prove our
results for an $F$ as general as possible, even though the \hp{} method shows
its full power (i.e., exponential rate of convergence) only in a less general setting.

To conclude, we mention the fact that we will mainly write our proofs so that they work
for $d=3$, even though this sometimes means using a suboptimal strategy for the
case $d\leq 2$. Consider for example the bound
\begin{equation*}
  \| v \|_{L^p(\Omega)}\leq C \| v\|_{H^1(\Omega)},
\end{equation*}
for a $v\in H^1(\Omega)$: we will always use it for $p$ such that $1\leq p\leq 6$, even if for $d=2$ any
$1\leq p<\infty$ would be acceptable.
\subsection{Continuity and coercivity}
We start with an auxiliary lemma, where we prove the continuity, positivity and
coercivity of some operators. As mentioned before, we use the numerical
eigenvalue $\lds$ obtained from the numerical approximation of the linear
problem as a lower bound of the operators over the discrete space $\Xd$.
\begin{lemma}
\label{lemma:estimates}
Recall the definition of the operators $A^u_\delta$ and $E''_{\delta}(u)$,
of the spaces $X_\delta$ and $X(\delta)$, and of $\lambda^*_\delta$ provided in
Section \ref{sec:statement}. Suppose that for all $e\in \mathcal{E}$, $\alpha_e \geq \alpha^*>0$
for sufficiently large $\alpha^*$ (see Remark \ref{remark:alphae}).There exist $\delta_0>0$ and a constant
  $C_{\mathrm{cont},1}>0$ such that
\begin{subequations}
\begin{align}
  \label{eq:estimates-lemma1b}
 \lvert  \langle \left( A^u_\delta - \lambda_\delta^* \right) v, v_\delta \rangle  \rvert &\leq C_{\mathrm{cont}} \fullnorm{v}\dgnorm{v_\delta} &\forall v \in X(\delta), \forall v_\delta \in X_\delta, \forall \delta\leq 1\\
  \label{eq:estimates-lemma1a}
  \langle \left( A^u_\delta - \lambda_\delta^* \right) v_\delta, v_\delta \rangle & \geq 0 &\forall v_\delta \in X_\delta, \forall\delta\leq \delta_0.
\end{align}
\end{subequations}
Furthermore, there exist $C_{\mathrm{pos}}, C_{\mathrm{coer}}, C_{\mathrm{cont},2}>0$
\begin{equation}
\label{eq:estimates-lemma2}
 \langle \left( A^u_\delta - \lambda_\delta^* \right) (u_\delta - u_\delta^*), (u_\delta-u_\delta^*) \rangle  \geq C_{\mathrm{pos}}\lVert u_\delta - u_\delta^*\rVert^2_\mathrm{DG},  \qquad \forall\delta\leq \delta_0,
\end{equation}
\begin{subequations}
  and
\begin{align}
\label{eq:estimates-lemma3a}
    \langle \left( E''_\delta(u) - \lambda_\delta^* \right) v_\delta, v_\delta \rangle & \geq C_{\mathrm{coer}} \lVert v_\delta\rVert_\mathrm{DG}^2 &\forall v_\delta \in X_\delta, \forall \delta \leq 1\\
\label{eq:estimates-lemma3b}
    \lvert \langle \left( E''_\delta(u) - \lambda_\delta^* \right) v, v_\delta \rangle \rvert &  \leq C_{\mathrm{cont}, 2}\fullnorm{v}\dgnorm{v_\delta}  &\forall v\in X(\delta), \forall v_\delta \in X_\delta, \forall \delta \leq 1.
\end{align}
\end{subequations}
\end{lemma}
\begin{proof}
Let us first consider the continuity inequality \eqref{eq:estimates-lemma1b}. The proof when $d=2$ is classical, see in particular \cite[Lemma 4.30]{DiPietro2011}
for the bound on the edges in $\mathcal{E}_c$, and the same arguments that we
use here for the bounds on the rest of the elements and edges. We
restrict then ourselves here to the case $d=3$, where we use a slightly different
norm than usual. 
Consider a function $v\in X(\delta)$. We can decompose $v=\tv+ \de{\tv}$, where
$v\in X$ and $\de{\tv}\in \de{X}$. Consider an edge/face $e\in \mathcal{E}^{\ell}$.
Then, $\jump{v}_{|_e} = \jump{\de{\tv}}_{|_e}$. If $\mathfrak{C}\cap \bar{e} =
\emptyset$, then $\dhe\simeq r_{|e}$, uniformly with respect to $e$ and
$\ell$. If instead there exists a $\fc\in \mathfrak{C}\cap S$
such that $\fc$ is one of the vertices of $e$, then
$\jump{\de{\tv}}_{|_e}\in\mathbb{Q}_{p_0}(e)$, which is a finite dimensional
space, whose dimension is independent of the refinement level $\ell$, which
  implies that $\dhe^{-1} \| \jump{\cdot}\|^2_{L^2(e)} \simeq 
    \| r^{-1/2}\jump{\cdot}\|^2_{L^2(e)}$, uniformly with respect to $\ell$. Therefore on $X(\delta)$ we have the
uniform equivalence
  \begin{equation}
  \label{eq:jump-eq-3d}
  \dhe^{-1} \| \jump{\cdot}\|^2_{L^2(e)} \simeq 
    \| r^{-1/2}\jump{\cdot}\|^2_{L^2(e)}, \qquad \forall e\in \mathcal{E}^\ell,
  \end{equation}
  if $d=3$, with hidden constant independent of $e$ and of $\ell$.
The continuity estimate \eqref{eq:estimates-lemma1b} can be obtained through multiple applications of H\"older's inequality: we consider the terms in the bilinear form separately. First, 
we exploit the fact that, as shown in Theorem 2.1
and Remark 2.3 of Ref. \refcite{Lasis2003}, for all $v \in 
H^1(\mathcal{T}^\ell):=\lbrace v: v_{|_K} \in H^1(K),\,\forall K\in
\mathcal{T}^\ell\rbrace$, there exists $C>0$ depending only on the domain $\Omega$
such that
\begin{equation*}
  \lVert v\rVert^2_{L^q(\Omega)} \leq C\left( \sum_{K\in \mathcal{T}^\ell}\|v\|^2_{H^1(K)}
    + \sum_{e\in\mathcal{E}^\ell} \| \jump{v} \|^2_{L^{2d/(d-2)}(e)}\right)
\end{equation*}
with $q\leq 2d/(d-2)$ if $d\geq 3$. Now, for any
function $v \in X(\delta)\subset H^1(\mathcal{T}^\ell)$ there
holds, for all $e\in \mathcal{E}^\ell$, $\jump{w}_{|_{e}}\in\mathbb{Q}_{\dpe}(e) $. By
Bernstein's inequality, then,
\begin{equation}
\label{eq:discrete-embedding}
  \lVert v\rVert_{L^q(\Omega)} \lesssim \lVert v\rVert_\mathrm{DG} \quad \forall v \in X(\delta)
\end{equation}
with $q\leq 2d/(d-2)$ if $d\geq 3$ and $q\in [1,\infty)$ if $d=2$.
Thus
  \begin{equation*}
    \left|\sum_{K\in\mathcal{T}}(\nabla v,\nabla \vd)_K+ (Vv, \vd)_K\right| \lesssim  \dgnorm{v}\dgnorm{\vd}
  \end{equation*}
  Secondly, 
\begin{align*}
    \left|\sum_e (\avg{\nabla { v}},\jump{{\vd}})_e \right| 
& \lesssim \sum_e \dpe^{-1}\|r^{1/2}\avg{\nabla  v}\|_{L^2(e)}  \dpe \| r^{-1/2} \jump{\vd}\|_{L^2(e)}\\
& \lesssim \sum_e \dpe^{-1}\|r^{1/2}\avg{\nabla  v}\|_{L^2(e)}  \dpe\dhe^{-1/2} \|  \jump{\vd}\|_{L^2(e)}\\
                            &\lesssim \left(\sum_e \dpe^{-2}\|r^{1/2}\avg{\nabla  v}\|^2_{L^2(e)}\right)^{1/2}\left(\sum_e  \dpe^2\dhe^{-1}\| \jump{\vd}\|^2_{L^2(e)}\right)^{1/2}
  \end{align*}
  where the second inequality follows from  \eqref{eq:jump-eq-3d}. Similarly,
  \begin{align*}
    \left|\sum_e (\avg{\nabla { \vd}},\jump{{v}})_e \right| &
\lesssim \left(\sum_e \dpe^{-2}\dhe\|\avg{\nabla  \vd}\|^2_{L^2(e)}\right)^{1/2}\left(\sum_e  \dpe^2\dhe^{-1}\|  \jump{v}\|^2_{L^2(e)}\right)^{1/2}\\
&\lesssim \left(\sum_K \|\nabla \vd\|^2_{L^2(K)}\right)^{1/2}\left(\sum_e  \dpe^2\dhe^{-1}\|  \jump{v}\|^2_{L^2(e)}\right)^{1/2},
  \end{align*}
  using \eqref{eq:edge-elem} in the second line.
  Then, 
  \begin{align*}
    \left|\sum_{e \in \mathcal{E}}\alpha_e \frac{\dpe^2}{\dhe} (\jump{v}, \jump{{\vd}})_{e}\right| 
        &\lesssim C \sum_e \left(\dpe \dhe^{-1/2}\| \jump{u} \|_{L^2(e)}\right) \left(\dpe \dhe^{-1/2}\| \jump{\vd} \|_{L^2(e)}\right)\\
        &\lesssim C \left(\sum_e\dpe^{2}\dhe^{-1}\| \jump{v} \|^2_{L^2(e)}\right)^{1/2}\left(\sum_e\dpe^{2}\dhe^{-1}\| \jump{\vd} \|^2_{L^2(e)}\right)^{1/2}.
  \end{align*}
  Thanks to the H\"{o}lder inequality, Sobolev imbeddings, hypothesis \eqref{eq:cond2}, and \eqref{eq:discrete-embedding}, 
  \begin{align*}
    \left| \int_\Omega f(u^2) v  \vd \right| 
&\lesssim \| 1 + u^{2q}\|_{L^{3/2}(\Omega)} \| v\|_{L^6(\Omega)}\| \vd\|_{L^6(\Omega)}\\
&\lesssim \| u\|^{2q}_{H^{1}(\Omega)} \| v\|_{H^1(\Omega)}\dgnorm{ \vd}\\
&\lesssim  \| v\|_{H^1(\Omega)}\dgnorm{ \vd}.
  \end{align*}
Since then, by \eqref{eq:convergence-linear} and
  Proposition \ref{prop:best-appx}, $\lds \to \lambda$ as $\delta\to 0$, we have that $|\lds(v,\vd)|\leq C \|v\|\|\vd\|$ and this, combined with the above inequalities, proves \eqref{eq:estimates-lemma1b}.

We now consider \eqref{eq:estimates-lemma1a}. As already stated,
$\lambda_\delta^*$ is a simple eigenvalue for a sufficient number of degrees of
freedom and therefore $A^u_\delta -\lds$ is coercive on the subspace of $\Xd$
$L^2$-orthogonal to $\uds$. Hence, since $\|\uds\|=1$ and $A^u_\delta$ is symmetric,
\begin{equation}
  \label{eq:semipos}
  \begin{aligned}
  \langle \left( A^u_\delta - \lambda_\delta^* \right) v_\delta, v_\delta \rangle 
 &=  \langle \left( A^u_\delta - \lambda_\delta^* \right) (\vd-(\vd,\uds)_\Omega\uds), \vd-(\vd,\uds)_\Omega\uds \rangle \\
&\gtrsim \lVert v_\delta\rVert^2 - (u_\delta^*, v_\delta)^2 \geq 0 ,
  \end{aligned}
\end{equation}
for all $\vd \in \Xd$.
We may then prove \eqref{eq:estimates-lemma2} following the same reasoning as in
Ref. \refcite{Cances2010}. We recall it here for ease of reading. We choose, without loss
of generality, $\uds$ such that $(\uds, \ud)\geq 0$. From the above
inequality we have (recall that $\|\uds\|=\|\ud\|=1$)
\begin{equation}
  \label{eq:A-l2norm}
  \begin{aligned}
  \langle \left( A^u_\delta - \lambda_\delta^* \right) (u_\delta - u_\delta^*), (u_\delta-u_\delta^*) \rangle  
& \gtrsim \| \ud-\uds\|^2 - (\uds, \ud-\uds)^2 \\
& = \| \ud-\uds\|^2 - (1 + (\uds, \ud)^2 - 2(\uds,\ud))\\
& = 1 - (\uds, \ud)^2 \\
& \geq \frac{1}{2} \lVert u_\delta - u_\delta^*\rVert^2,
  \end{aligned}
\end{equation}
and this proves \eqref{eq:estimates-lemma2}.
To prove \eqref{eq:estimates-lemma3a}, we note that
\begin{equation}
  \label{eq:secondE}
  \langle \left( E''_\delta(u) - \lambda_\delta^* \right) v_\delta, v_\delta \rangle \geq \langle \left( A^u_\delta - \lambda_\delta^* \right) v_\delta, v_\delta \rangle + \int_\Omega f'(u^2)u^2 v_\delta^2.
\end{equation}
Suppose we negate \eqref{eq:estimates-lemma3a}: then, there has to be a sequence
$\{v_\delta^j\}_j\subset \Xd$ such that $\|v_\delta^j\|=1$ and $\langle \left(
  E''_\delta(u) - \lambda_\delta^* \right) v_{\delta}^j, v_\delta^j \rangle \to
0$. Since $\int_\Omega f'(u^2)u^2 (v_\delta^j)^2>0$, from \eqref{eq:semipos} we have
that
\begin{align*}
  \frac{1}{2} \| v_\delta^j - \uds\|^2
  &= \| v_\delta^j\|^2 - (v_\delta^j, \uds)^2 \\
&\lesssim \langle  (E''(u) - \lds) v_\delta^j, v_\delta^j\rangle,
\end{align*}
thus, $v_\delta^j\to \uds$ in $L^2(\Omega)$. Now, since $\uds$ converges towards $u$ in the
$\mathrm{DG}$ norm, and using \eqref{eq:cond3} and the positivity of $f'$, we
can show that there exists an $\alpha>0$ such that, for a sufficient
number of degrees of freedom,
\begin{equation*}
  \int_\Omega f'(u^2)u^2 (\uds)^2 > \alpha.
\end{equation*}
This negates the contradiction hypothesis that $\langle \left(
  E''_\delta(u) - \lambda_\delta^* \right) v_{\delta}^j, v_\delta^j \rangle \to
0$, hence there exists a constant $C>0$, independent of $\delta$, such that
\begin{equation}
  \label{eq:E2L2}
  \langle \left( E''_\delta(u) - \lambda_\delta^* \right) v_\delta, v_\delta \rangle \geq 
  C \|\vd\|^2
\end{equation}
for all $\vd\in\Xd$.
Then, using the fact that there exists $c>0$ such that for
all $\vd\in \Xd$,
\begin{equation*}
  (\nabla v_\delta, \nabla v_\delta)_{\mathcal{T}} - (\avg{\nabla v_\delta},\jump{v_\delta})_{\mathcal{E}}- (\avg{\nabla v_\delta},\jump{v_\delta})_{\mathcal{E}} + \sum_{e \in \mathcal{E}}\alpha_e \frac{\dpe^2}{\dhe} (\jump{v_\delta}, \jump{v_\delta})_{e} \geq c \dgnorm{\vd}^2,
\end{equation*}
see Remark \ref{remark:alphae}, combined with the estimate from the proof of \cite[Lemma 1]{Cances2010}, we can
show that there exist constants $\alpha, C>0$ independent of $\delta$ such that 
\begin{equation}
\label{eq:A-semicoer}
   \langle \left( A^u_\delta - \lambda_\delta^* \right) v_\delta, v_\delta \rangle \geq \alpha \lVert v_\delta \rVert^2_\mathrm{DG} - C \lVert v_\delta \rVert^2 \qquad \forall \vd\in\Xd.
\end{equation}
The coercivity estimate \eqref{eq:estimates-lemma3a} then follows from
\eqref{eq:E2L2} and \eqref{eq:A-semicoer}.

Finally, \eqref{eq:estimates-lemma3b} follows directly from the definition of
$E''_\delta(u)$, the continuity estimate \eqref{eq:estimates-lemma1a} and 
the fact that $|f'(u^2)u^2| \leq C$.
\end{proof}
  \subsection{Estimates on the adjoint problem}
  In this section we develop an estimate on the scalar product between a
  function and the error $u-\ud$, whose
  interest lies mainly in the $L^2(\Omega)$ convergence estimate given in
  Theorem \ref{theorem:convergence1}.
  The estimate is based on the introduction of the adjoint problem \eqref{eq:nonlin-adjoint}.
\begin{lemma}
  \label{lemma:adjoint}
  Let $\udso= \lbrace \vd\in \Xd: (\vd, \uds) = 0 \rbrace$ be the space of functions $L^2(\Omega)$-orthogonal to $\uds$ and let $\psiwd$ be the solution to the problem
  \begin{equation}
    \label{eq:nonlin-adjoint}
    \begin{aligned}
      &\text{find }\psiwd\in \udso\text{ such that}\\
      &\langle \left( E_\delta''(u) - \lds \right) \psiwd , \vd \rangle = \langle \wdel, \vd \rangle, \forall \vd\in\udso
    \end{aligned}
  \end{equation}
  for $\wdel$ in $L^2(\Omega)$.
  Then, if hypotheses \eqref{eq:cond1} to \eqref{eq:cond4} hold,
  \begin{equation}
\label{eq:scalar-prod-estim}
    \begin{aligned}
      \left| \langle  \wdel, u_\delta - \uds \rangle \right|
&\lesssim \| \ud - \uds\|^r_{L^{6r/(5-s)}} \dgnorm{\psiwd} + | \lambda_\delta - \lds| \| u_\delta - \uds \| \|\psiwd \| 
  \\  &\quad  + \| u-\uds\|\|\psiwd\|
  + \| \ud - \uds \|^2 \|\psi_\wdel\| 
 + \| \ud - \uds \|^2 \| \wdel \|, 
    \end{aligned}
  \end{equation}
\end{lemma}
\begin{proof}
We break $\ud -\uds$ into two parts, one parallel to $\uds$ and one
perpendicular to it. Those are given respectively by
\begin{equation*}
(\ud-\uds, \uds)\uds = -\frac{1}{2}\|\ud - \uds \|^2 \uds \quad\text{ and }\quad \ud-(\ud,\uds)\uds \in \udso.
\end{equation*}
Then,
\begin{equation}
  \label{eq:lemma-adj1}
\begin{aligned}
  \langle \wdel, \ud - \uds \rangle & = (\wdel, \ud - (\ud,\uds)\uds) - \frac{1}{2}\|\ud-\uds\|^2(\wdel, \uds)\\
  & = \langle  \left( E_\delta''(u) - \lds \right) \psi_\wdel, \ud - (\ud,\uds)\uds \rangle - \frac{1}{2} \| \ud - \uds \|^2 (\wdel, \uds) \\
& =  \langle  \left( E_\delta''(u) - \lds \right) (\ud-\uds), \psi_\wdel \rangle - \frac{1}{2} \| \ud - \uds \|^2 \langle  \left( E_\delta''(u) - \lambda\right) \uds,  \psi_\wdel \rangle \\ 
&\quad  - \frac{1}{2} \| \ud - \uds \|^2 (\wdel, \uds) \\ 
& =  \langle  \left( E_\delta''(u) - \lds \right) (\ud-\uds), \psi_\wdel \rangle - \| \ud - \uds \|^2 \int_\Omega  f'(u^2) u^2\uds \psi_\wdel \\ 
&\quad - \frac{1}{2} \| \ud - u \|^2 (\wdel, \uds) . 
\end{aligned}
\end{equation}
We consider the first term:
\begin{equation}
  \label{eq:lemma-adj2}
\begin{aligned}
  \langle  \left( E_\delta''(u) - \lds \right) (\ud-\uds), \psi_{\wdel} \rangle &=
  \langle (\Ad^u-\lds)\ud, \psiwd \rangle + 2\int_\Omega f'(u^2)u^2\psiwd(\ud -\uds)\\
& = -\int_\Omega \left[ f(\ud^2)\ud - f(u^2)\ud - 2f'(u^2)u^2(\ud-u)\right] \psiwd \\
& \quad+ (\ld -\lds) (\ud - \uds, \psiwd)\\
& \quad+ 2\int_\Omega f'(u^2)u^2 \psiwd (u-\uds).
\end{aligned}
\end{equation}
Thanks to \eqref{eq:cond4}, combining \eqref{eq:discrete-embedding}, \eqref{eq:lemma-adj1}, and \eqref{eq:lemma-adj2} we can infer that
\begin{align*}
  \left| \langle  \wdel, u_\delta - \uds \rangle \right|
  &\lesssim \| \ud - \uds\|^r_{L^{6r/(5-s)}} \dgnorm{\psiwd} + | \lambda_\delta - \lds| \| u_\delta - \uds \| \|\psiwd \| 
  \\  &\quad  + \| u-\uds\|\|\psiwd\|
  + \| \ud - \uds \|^2 \int_\Omega  \left| f'(u^2) u^2\uds \psi_\wdel\right| 
\\  &\quad  + \| \ud - \uds \|^2 \left| (\wdel, \uds) \right|, 
\end{align*}
which gives the thesis.
\end{proof}

\subsection{Convergence}
\label{sec:convergence}
At this stage, we are able to prove the convergence result for the
numerical eigenfunction and eigenvalue. We work mainly in the discrete setting,
in order to avoid the issues due to the nonconformity of the method.
The analysis is carried out for the symmetric interior penalty discontinuous
Galerkin method, but it holds for any nonconforming symmetric
method, as long as the results of Lemma \ref{lemma:estimates} hold for such a method.
Furthermore, the remark made at the beginning of Section \ref{sec:apriori} still
holds, in that the result can be adapted with few modifications to a classical
$h$-type discontinuous Galerkin finite element method.

In general, the goal is to prove that the numerical eigenvalue-eigenfunction
couple obtained as solution to the nonlinear problem converges as fast as for
linear elliptic operators. In this section, we obtain this result for the eigenfunction,
which is shown to converge quasi optimally. We prove that the eigenvalue converges at least as fast as the eigenfunction.

The following theorem gives then the above mentioned estimates on the
convergence of the eigenfunction and eigenvalue. We start by showing the
convergence to zero of the error, and use this result to show that the estimate
is quasi optimal. We then show that the eigenvalue converges, with the basic
rate mentioned above, and conclude by showing an estimate on the $L^2(\Omega)$
norm of the error.
\begin{theorem}
  \label{theorem:convergence1}
  If the hypotheses \eqref{eq:cond1} to \eqref{eq:cond4} on $F$ hold and the
  hypotheses on the potential $V$ \eqref{eq:hyp-potential-1}, \eqref{eq:hyp-potential-2} hold, then
    \begin{equation}
    \label{eq:convergence}
    \lVert u - \ud\rVert_\mathrm{DG} \to 0, \qquad \text{as }\delta\to 0.
  \end{equation}
 In particular, there exists a $\delta_0>0$ such that we have the quasi-optimal convergence
    \begin{equation}
    \label{eq:optimality}
    \lVert  u -\ud \rVert_\mathrm{DG} \lesssim \inf_{v_\delta\in X_\delta}\fullnorm{u-\vd}\qquad \forall\delta\leq\delta_0.
  \end{equation}
  Furthermore, 
  \begin{equation}
    \label{eq:lambda}
|\lambda -\ld| \lesssim \inf_{v_\delta\in X_\delta}\fullnorm{u-\vd},   \qquad \forall\delta\leq\delta_0
  \end{equation}
  and
  \begin{equation}
    \label{eq:l2}
  \| u - \ud\| \lesssim
  \| u - \uds\|^r_{L^{6r/(5-s)}}  +\| u - \ud\|^r_{L^{6r/(5-s)}}  +  \| u-\uds\|,\qquad \forall\delta\leq\delta_0,
 \end{equation}
 where $r$ is defined in \eqref{eq:cond4} and $\uds$ is the solution of the
 linear eigenvalue problem defined in \eqref{eq:uds-def}.
\end{theorem}
\begin{proof}
We start by proving \eqref{eq:convergence}, i.e.~the convergence of the
numerical solution towards the exact one. We have
 \begin{align*}
     2\left( E_\delta(u_\delta)- E(u) \right)&= \langle A_\delta^u u_\delta, u_\delta \rangle -  \langle A^u u, u\rangle + \int_\Omega \left( F(u_\delta^2)-F(u^2)-f(u^2)(u_\delta^2-u^2) \right) \\
                           & =  \langle \left( A_\delta^u -\lambda_\delta^* \right) (u_\delta - u_\delta^*), u_\delta - u_\delta^* \rangle - \lambda+\lambda_\delta^*\\ & \quad
+ \int_\Omega \left( F(u_\delta^2)-F(u^2)-f(u^2)(u_\delta^2-u^2) \right) \\
                            & \gtrsim \lVert u_\delta - u_\delta^* \rVert^2_\mathrm{DG} - |\lambda -\lambda_\delta^*| + \int_\Omega \left( F(u_\delta^2)-F(u^2)-f(u^2)(u_\delta^2-u^2) \right).
  \end{align*} 
Therefore, exploiting the convexity of $F$ and the convergence of $\lambda$ towards $\lambda_\delta^*$, we have that
\begin{equation}
\begin{aligned}
  \label{eq:norm-control}
 \lVert u_\delta - u_\delta^* \rVert^2_\mathrm{DG} &\lesssim  E_\delta(u_\delta)- E(u) + |\lambda -\lambda_\delta^*| \\ &  \leq E_\delta(\Pi_\delta u) - E_\delta(u) +  |\lambda -\lambda_\delta^*| \to 0.
\end{aligned}
\end{equation} 
Considering that $u_\delta^*$ converges towards $u$ in the DG norm, \eqref{eq:norm-control} implies \eqref{eq:convergence}.
Note then that
\begin{equation}
  \label{eq:l-lds-1}
  \begin{aligned}
  \lambda_\delta - \lds &= \langle A^{u}_\delta u_\delta, u_\delta \rangle - \lds + \int_\Omega \left[ f(u_\delta^2)-f(u^2)\right]u_\delta^2\\
                                 &=  \langle \left( A^{u}_\delta - \lds\right) \left( u_\delta-\uds\right), u_\delta -\uds \rangle  + \int_\Omega\left[ f(u_\delta^2)-f(u^2)\right] u_\delta^2.
  \end{aligned}
\end{equation} 
Remarking, as in the proof of Theorem 1 of Ref.~\refcite{Cances2010}, that
\begin{equation*}
  \int_\Omega\left[ f(u_\delta^2)-f(u^2)\right] u_\delta^2 \leq \|1+\ud^{2q+1}\|_{L^{6/(2q+1)}(\Omega)} \dgnorm{u-\ud}
\end{equation*}
and using \eqref{eq:discrete-embedding} and \eqref{eq:convergence}, we can conclude that
\begin{equation}
\label{eq:temp-lambda}
  | \lambda -\ld| \lesssim |\lambda - \lds| + \dgnorm{\ud - \uds}^2 + \dgnorm{u-\ud}.
\end{equation}
Now, from \eqref{eq:estimates-lemma3a} we have
\begin{align*}
  \dgnorm{\ud- \uds}^2 
&\lesssim \langle \left( E_\delta''(u)-\lds \right) (\ud-\uds), \ud-\uds\rangle \\
&=  \langle \left( \Ad^u-\lds \right) (\ud-\uds), \ud-\uds\rangle + 2\int_\Omega f'(u^2)u^2(\ud -\uds)^2 \\
&=  \langle \left( \Ad^u-\ld \right) \ud, \ud-\uds\rangle + (\ld -\lds)\|\ud -\uds\|^2+2\int_\Omega f'(u^2)u^2(\ud -\uds)^2 \\
  & = \int_\Omega \left[ \left( f(u^2) - f(\ud^2)  \right)\ud +2f'(u^2)u^2(\ud-\uds)\right] (\ud-\uds)
    \\ & \qquad + (\ld-\lds) \|\ud -\uds\|^2.
\end{align*}
Consider the first term: hypothesis \eqref{eq:cond3} gives
\begin{equation*}
  \int_\Omega f'(u^2)u^2(\ud-\uds)^2 \lesssim \int_\Omega f'(u^2)u^2(\ud - u)(\ud-\uds) +  \| u -\uds\| \| \ud-\uds\|.
\end{equation*}
The two above equations and \eqref{eq:cond4} thus give
\begin{multline*}
  \dgnorm{\ud- \uds}^2  \lesssim  \|1+ |\ud|^s\|_{L^{6/s}(\Omega)}\|\ud-u
  \|_{L^{6r/(5-s)}(\Omega)}^r\dgnorm{\ud -\uds} \\
  + |\ld-\lds| \| \ud - \uds\|^2
  + \| u - \uds \| \|\ud - \uds\|
\end{multline*}
and, since $r>1$ and $6r/(5-s)\leq 6$, we can conclude that
\begin{equation*}
  \dgnorm{u-\ud} \lesssim \dgnorm{u-\uds}.
\end{equation*}
The quasi optimality of $\uds$ then
implies \eqref{eq:optimality}. Additionally, we can use this estimate in
\eqref{eq:temp-lambda} and, considering that
\begin{equation*}
  |\lambda -\lds|  \lesssim \fullnorm{u-\uds}^2\lesssim \inf_{\vd \in \Xd} \fullnorm{u-\vd}^2,
\end{equation*}
we conclude that
\begin{equation*}
  |\lambda -\ld| \lesssim \inf_{\vd \in \Xd}\fullnorm{u-\vd}.
\end{equation*}
Note that this result can be a bit sharper if $q$ in $\eqref{eq:cond2}$ is
significantly smaller than $2$; we write it this way for ease of reading. As
already mentioned, we
will prove a sharper result under some additional conditions in the following sections.

We finish by showing the estimate for the $L^2$ norm of the error. This follows
from Lemma \ref{lemma:adjoint}, since \eqref{eq:scalar-prod-estim} implies
\begin{multline}
  \label{eq:L2-1}
  \| \ud - \uds\|^2 \lesssim
  \| \ud - \uds\|^r_{L^{6r/(5-s)}} \dgnorm{\psi_{\ud-\uds}} + | \lambda_\delta - \lds| \| u_\delta - \uds \| \|\psi_{\ud-\uds} \| 
  \\   + \| u-\uds\|\|\psi_{\ud-\uds}\|
   + \| \ud - \uds \|^2 \|\psi_{\ud-\uds}\| 
  + \| \ud - \uds \|^3 
\end{multline}
for $\psi_{\ud-\uds}\in \Xd$ defined as in \eqref{eq:nonlin-adjoint}, with $\wdel =
\ud-\uds$. Now, the coercivity of $\langle (E''(u)-\lds)\cdot ,\cdot\rangle$
over $\Xd$ shown in \eqref{eq:estimates-lemma3a} and a Cauchy-Schwarz inequality imply
\begin{equation}
  \label{eq:psi-stab}
  \dgnorm{\psi_{\ud-\uds}} \lesssim \| \ud -\uds\|.
\end{equation}
Hence, from the combination of \eqref{eq:L2-1}, \eqref{eq:psi-stab}, 
and the convergences of $\ld$ towards $\lds$ and of $\ud$ towards $\uds$ in the
$L^2(\Omega)$ norm, we derive
\begin{equation*}
  \| \ud - \uds\| \lesssim
  \| \ud - \uds\|^r_{L^{6r/(5-s)}}  +  \| u-\uds\|.
\end{equation*}
Noting that
\begin{equation*}
  \| u -\ud\| \leq \| u -\uds\|+\|\ud-\uds\|
\end{equation*}
we conclude the proof.
\end{proof}

\section{Weighted analytic regularity for polynomial nonlinearities}
\label{chap:nonlinear-regularity}
This section is centered on the proof of analytic-type estimates on
the norms of the solution to nonlinear elliptic problems.  Specifically, we
consider the nonlinear Schr\"{o}dinger equation and prove
that, under some conditions on the coefficients of the operator, the solution
belongs to $\mathcal{J}^{\varpi}_\gamma(\Omega)$, for the same $\gamma$ as in
the linear case seen in Ref.~\refcite{linear}. Since the singularities we
consider are internal to the domain, we suppose that the domain $\Omega$ is a
compact domain without boundary, e.g., $\Omega = (-1,1)^d/2\mathbb{Z}$.
The extension of the theory to the case of a bounded domain with smooth boundary
can be done using the classical tools used in the analysis of elliptic problems
in Sobolev spaces, as long as $r_{|_{\partial \Omega}} \simeq 1$, i.e., the
singularity is bounded away from the boundary.

First, in Section \ref{sec:localelliptic} we prove the local elliptic estimate in weighted Sobolev spaces that will allow for the derivation of the bounds on higher order derivatives from those obtained on lower order ones.  Then, in order to estimate the norms of the nonlinear terms, we follow the proof technique used in Ref.~\refcite{DallAcqua2012}. 
The idea is to proceed by induction and to consider $L^p$ norms in nested balls
and with a big enough $p$. Let $L_{\mathrm{lin}}$ be an elliptic linear operator
and consider the operator $L_u u  = -L_{\mathrm{lin}} u + |u|^{n-1}u$,
where $n=2,3,4$: the $L^p$ norms of the nonlinear terms can then be broken
up into products of $L^{n p}$ norms by a Cauchy-Schwarz inequality.
In order to get back to $L^p$ norms, in Ref.~\refcite{DallAcqua2012} the authors use an
interpolation inequality where $L^{n p}$ is contained in an interpolation
space between $L^p$ and $W^{1,p}$. Since in our case we need to deal with weighted spaces, in Section \ref{sec:thetaprod} we derive the weighted version of this inequality, via a dyadic decomposition of the domain near the singular points.

The proof of the analytic bound on a nonlinear scalar elliptic eigenvalue problem is then given in Section \ref{sec:nonlinsch-analytic}, in the case of the nonlinear Schr\"{o}dinger equation up to a quartic nonlinear term. Starting from a basic regularity assumption, we are able to treat the potential and the nonlinear term thanks to the results presented in the preceding sections.

We suppose, for the sake of simplicity, the presence of a single point
  singularity, i.e., that $\fC = \{\fc\}$. For any $L>0$, we indicate by
    $B_L = B_L(\fc)$ the ball of radius $L$ centered at $\fc$.

We denote the commutator by square brackets, i.e., we write
\begin{equation*}
  \left[ A, B \right] = AB-BA.
\end{equation*}

\subsection{Local elliptic estimate}
\label{sec:localelliptic}
We start by proving a local seminorm estimate in weighted Sobolev spaces. This has been already established in Ref.~\refcite{Costabel2012a}, as an intermediate estimate leading to the proof of another regularity result. We restate it here fully, in the specific form that will be needed in the sequel. The goal is to control the weighted norm of a higher order derivative of a function with the weighted norm of its Laplacian and of lower order derivatives in a bigger domain, while giving an explicit dependence of the constants on the distance between the domains.
\begin{proposition}
  \label{lemma:elliptic-weighted}
  Let $1<p<\infty$, $R>0$ such that $B_R \subset \Omega$, and 
  let $\gamma \in \mathbb{R}$. There exists $\Creg>0$ such that
  for all $k\in\mathbb{N}$, all $\rho\in(0, \frac{R}{2(k+1)})$, all $j\in \mathbb{N}$ such that
  $1\leq j \leq k$, and all $u$ with bounded $|\cdot|_{\mathcal{K}^{i,
        p}_\gamma(B_{R-j\rho})}$, $i=k-1, k, k+1$
    seminorms,
  \begin{multline}
    \label{eq:elliptic-weighted}
    \sum_{|\alpha| = k+1} \| r^{k+1-\gamma} \dalpha u\|_{L^p(B_{R-(j+1)\rho})}
    \leq \Creg\left( \sum_{|\beta| = k-1} \| r^{k+1-\gamma}\dbeta(\Delta u)\|_{L^p(B_{R-j\rho})} \right.\\ \left.
+ \sum_{|\alpha| = k}\rho^{-1}\|r^{|\alpha|-\gamma}\dalpha u\|_{L^p(B_{R-j\rho})}
+ \sum_{|\alpha| = k-1}\rho^{-2}\|r^{|\alpha|-\gamma}\dalpha u\|_{L^p(B_{R-j\rho})}  \right).
  \end{multline}
\end{proposition}
In order to prove this Proposition we introduce, for each $j\in \mathbb{N}$, a smooth cutoff function  
$\psi_{j}\in C^\infty_0(B_{R-j\rho})$ such that 
\begin{equation}
  \label{eq:psidef}
  \begin{gathered}
  0\leq \psi_{j} \leq 1, \qquad \psi_{j}=1 \text{ on }B_{R-(j+1)\rho},\\
  \exists C_{\psi} > 0 \text{ such that }\left| \dalpha\psi_{j}\right|\leq C_{\psi} \rho^{-|\alpha|}, \quad \forall \alpha \in \mathbb{N}_{0}^d,
\alpham \leq 2,
  \end{gathered}
\end{equation}
for all $\rho \in (0, R/(j+1))$ and with $C_\psi$ independent of $j$. Furthermore, we derive an auxiliary estimate.
\begin{lemma}
  \label{lemma:commutator1}
  Let $1<p<\infty$, $R>0$ such that $B_R\subset \Omega$, and $\gamma \in \mathbb{R}$ .
There exists $\Ccomm>0$ such that, for all
  $\beta \in \mathbb{N}^d_{0}$, for all $\rho\in(0,
  \frac{R}{2(|\beta|+2)}]$, for all $j\in \mathbb{N}$ such that $1 \leq j \leq
  |\beta|+1$, and for all $u$ such that $|u|_{\mathcal{K}^{i,
      p}_{\gamma}(B_{R-j\rho})} <\infty$, $i=\betam, \betam+1$,
  \begin{equation}
    \label{eq:commutator1}
    \sum_{|\alpha|=2}\|\left[ \dalpha, r^{|\beta|+2-\gamma}\right] \psi_j \dbeta u\|_{L^p(B_{R-j\rho})}\leq 
\Ccomm \sum_{|\alpha| \leq 1} \rho^{-2+|\alpha|}\| r^{|\beta|+|\alpha|-\gamma} \dab u\|_{L^p(B_{R-j\rho})},
  \end{equation}
  and $\Ccomm$ depends only on $\gamma$ and $R$.

\end{lemma}
\begin{proof}
  First, let us fix $i,k\in \{1, \dots, d\}$ such that $\dalpha =
  \dfrac{\partial^2}{\partial x_i \partial x_k} = \partial_{ik}=
  \partial_i\partial_k$.
  Then,
  \begin{multline*}
      \left[ \dalpha, r^{|\beta|+2-\gamma}\right] \psi_{j} \dbeta u= 
\left( \partial_{ik} r^{|\beta|+2-\gamma}\right) \psi_{j} \dbeta u
  +\left( \partial_{i} r^{|\beta|+2-\gamma}\right) \partial_k\left(\psi_{j} \dbeta u\right)\\
  +\left( \partial_{k} r^{|\beta|+2-\gamma}\right) \partial_i\left(\psi_{j} \dbeta u\right)
    \end{multline*}
Writing $(\star) = \|\left[ \dalpha,
    r^{|\beta|+2-\gamma}\right] \psi_{j} \dbeta u\|_{L^p(B_{R-j\rho})}$, we have that
  \begin{align*}
    (\star)
    &\leq
      \|\left( \partial_{ik} r^{|\beta|+2-\gamma}\right) \psi_{j} \dbeta u\|_{L^p(B_{R-j\rho})} 
      +\|\left( \partial_{i} r^{|\beta|+2-\gamma}\right) \partial_k\left(\psi_{j} \dbeta u\right)\|_{L^p(B_{R-j\rho})}
    \\
    &\qquad +\|\left( \partial_k r^{|\beta|+2-\gamma}\right) \partial_i\left(\psi_{j} \dbeta u\right)\|_{L^p(B_{R-j\rho})}
    \\
    &\leq  
      {\big|\betam + 2 -\gamma\big|\bigg(\left(\big|\betam-\gamma\big|+1\right)} \| r^{|\beta|-\gamma}\psi_{j} \dbeta u\|_{L^p(B_{R-j\rho})}\\
    &\qquad
      +  \| r^{|\beta|+1-\gamma}\partial_k\left(\psi_{j} \dbeta u\right)\|_{L^p(B_{R-j\rho})}+ \|  r^{|\beta|+1-\gamma} \partial_i\left(\psi_{j} \dbeta u\right)\|_{L^p(B_{R-j\rho})} \bigg)
      \\
&\leq  
      (2+|\gamma|)^2\bigg(\left(\betam+1\right)^2 \| r^{|\beta|-\gamma}\psi_{j} \dbeta u\|_{L^p(B_{R-j\rho})}\\
    &\qquad
      + (\betam+1) \left(\| r^{|\beta|+1-\gamma}\partial_k\left(\psi_{j} \dbeta u\right)\|_{L^p(B_{R-j\rho})}+ \|  r^{|\beta|+1-\gamma} \partial_i\left(\psi_{j} \dbeta u\right)\|_{L^p(B_{R-j\rho})}   \right)\bigg).
  \end{align*}
  Now, for $\iota\in\{i, k\}$,
  \begin{align*}
 &  \| r^{|\beta|+1-\gamma}\partial_{\iota}\left(\psi_{j} \dbeta u\right)\|_{L^p(B_{R-j\rho})}
     \\  &\qquad \qquad\leq
  \| r^{|\beta|+1-\gamma}\psi_{j} \partial_{\iota}\dbeta u\|_{L^p(B_{R-j\rho})}
 + \| r^{|\beta|+1-\gamma}\left(\partial_{\iota} \psi_{j}\right)\dbeta u\|_{L^p(B_{R-j\rho})}  
    \\ & \qquad \qquad \leq
  \| r^{|\beta|+1-\gamma}\partial_{\iota}\dbeta u\|_{L^p(B_{R-j\rho})}
 + C_\psi R \rho^{-1}\| r^{|\beta|-\gamma}\dbeta u\|_{L^p(B_{R-j\rho})}
  \end{align*}
  were we have also used the definition of $\psi_{j}$ given in \eqref{eq:psidef}. We obtain
  \begin{multline*}
    (\star) \leq
    (2+|\gamma|)^2
  \left( (|\beta|+1)^2+(|\beta|+1)2C_\psi R\rho^{-1}\right) \| r^{|\beta|-\gamma} \dbeta u\|_{L^p(B_{R-j\rho})}
    \\ 
         +  \sum_{\iota \in \{i, k\}}(|\beta|+1)\| r^{|\beta|+1-\gamma}\partial^{\beta}\partial_{\iota} u\|_{L^p(B_{R-j\rho})}.
  \end{multline*}
  Summing over all multi indices $|\alpha|=2$, it follows that there exists a
  constant $C>0$ that depends only on $\gamma$ and $R$ such that
  \begin{align*}
    \sum_{|\alpha| = 2}\|\left[ \dalpha, r^{|\beta|+2-\gamma}\right] \psi_{j} \dbeta u\|_{L^p(B_{R-j\rho})}
    &\leq
  C\bigg(  \left( (|\beta|+1)^2+(|\beta|+1)\rho^{-1}\right) \| r^{|\beta|-\gamma}\dbeta u\|_{L^p(B_{R-j\rho})}\\
 & \qquad +  \sum_{|\alpha|=1} (|\beta|+1)\| r^{|\beta|+1-\gamma}\partial^{\beta+\alpha} u\|_{L^p(B_{R-j\rho})} \bigg).
  \end{align*}
  Since $\rho\in(0, \frac{R}{2(|\beta|+2)}]$ implies $|\beta|+1\leq R \rho^{-1}$, we can conclude with \eqref{eq:commutator1}.
\end{proof}
We can now prove estimate \eqref{eq:elliptic-weighted}.
\begin{proof}(Proof of Proposition \ref{lemma:elliptic-weighted})
 Let us consider a multi index $\beta\in \mathbb{N}^d_0$ such that $\betam
 = k-1$.
 First,
 \begin{multline}
   \label{eq:ellipticreg1}
   \sum_{|\alpha|=2} \|r^{|\beta|+2-\gamma }\dab u\|_{L^p(B_{R-(j+1)\rho})} \leq 
\sum_{|\alpha|=2}\left\{ \|\dalpha \left(r^{|\beta|+2-\gamma} \dbeta u\right)\|_{L^p(B_{R-(j+1)\rho}) } \right. \\ 
\left.+ \|\left[\dalpha, r^{|\beta|+2-\gamma} \right]\dbeta u\|_{L^p(B_{R-(j+1)\rho}) }\right\}.
 \end{multline}
 We consider the first term at the right hand side: using \eqref{eq:psidef}
 \begin{equation*}
\sum_{|\alpha|=2} \|\dalpha \left(r^{|\beta|+2-\gamma} \dbeta u\right)\|_{L^p(B_{R-(j+1)\rho}) } \leq 
\sum_{|\alpha|=2} \|\dalpha \left(r^{|\beta|+2-\gamma} \psi_{j}\dbeta u\right)\|_{L^p(B_{R-j\rho}) }.
 \end{equation*}
 By elliptic regularity and using the triangle inequality, since, for all
   $j\in \mathbb{N}$ and for all $\beta\in\mathbb{N}^d_0$,
   $\psi_{j}\dbeta u $ has compact support in $B_R$, there exists a
   constant $C_\Delta>0$ that depends only on $p$ and $R$ such that
\begin{align*}
&\sum_{|\alpha|=2} \|\dalpha \left(r^{|\beta|+2-\gamma} \psi_{j}\dbeta u  \right)\|_{L^p(B_{R-j\rho}) }
                 \\
&\qquad \leq
C_{\Delta}  \|\Delta \left(r^{|\beta|+2-\gamma} \psi_{j}\dbeta u\right)\|_{L^p(B_{R-j\rho}) }\\
&\qquad \leq 
C_{\Delta}\bigg(  \|r^{|\beta|+2-\gamma} \psi_{j}\Delta \dbeta u\|_{L^p(B_{R-j\rho}) } 
+ 
                                                                                                   \|\left[\Delta,r^{|\beta|+2-\gamma}\right] \psi_{j}\dbeta u\|_{L^p(B_{R-j\rho}) } \\
  &\qquad \qquad 
+
\|r^{|\beta|+2-\gamma}\left[\Delta,\psi_{j}\right] \dbeta u\|_{L^p(B_{R-j\rho}) } \bigg).
\end{align*}
Combining the last inequality with \eqref{eq:ellipticreg1} we obtain
\begin{align*}
&   \sum_{|\alpha|=2} \|r^{|\beta|+2-\gamma }\dab u\|_{L^p(B_{R-(j+1)\rho})} 
 \leq 
  C_\Delta\|r^{|\beta|+2-\gamma} \psi_{j}\dbeta\left(\Delta u\right)\|_{L^p(B_{R-j\rho}) } 
   \\ & \qquad \qquad
   +C_\Delta\sum_{i=1}^d\left\{ \| r^{|\beta|+2-\gamma} \left(\partial_{ii}\psi_{j}\right) \dbeta u\|_{L^p(B_{R-j\rho}) } \right.
 \left.+ \| r^{|\beta|+2-\gamma} \left(\partial_{i}\psi_{j}\right) \partial^{\beta}\partial_i u\|_{L^p(B_{R-j\rho}) } \right\} 
   \\ & \qquad \qquad
+ (1+C_\Delta)\sum_{|\alpha|=2}\|\left[\dalpha, r^{|\beta|+2-\gamma} \right]\psi_{j}\dbeta u\|_{L^p(B_{R-j\rho}) }.
\end{align*}
The bounds on the derivatives of $\psi_{j}$ given in \eqref{eq:psidef} and \eqref{eq:commutator1} then give
\begin{align*}
  & \sum_{|\alpha|=2} \|r^{|\beta|+2-\gamma }\dab u\|_{L^p(B_{R-(j+1)\rho})} 
  \leq
  C_\Delta\|r^{|\beta|+2-\gamma} \psi_{j}\dbeta\left(\Delta u\right)\|_{L^p(B_{R-j\rho}) } \\
  &\qquad \qquad + C_\Delta C_\psi
    \sum_{|\alpha|\leq 1} \rho^{-2+|\alpha|}R^{2-\alpham}d \|r^{|\beta|+|\alpha| -\gamma} \dab u\|_{L^p(B_{R-j\rho})}.
    \\ & \qquad \qquad + (1+C_{\Delta})\Ccomm
    \sum_{|\alpha|\leq 1} \rho^{-2+|\alpha|} \|r^{|\beta|+|\alpha| -\gamma} \dab u\|_{L^p(B_{R-j\rho})},
\end{align*}
where $\Ccomm$ is the constant introduced in Lemma \ref{lemma:commutator1} and
  is independent of $\beta$, $u$, $\rho$, and $j$.
We can now sum over all multi indices $\beta$ such that $|\beta| = k-1$ to obtain the thesis \eqref{eq:elliptic-weighted}.
\end{proof}
\subsection{Weighted interpolation estimate}
  \label{sec:thetaprod}
\begin{lemma}
  \label{lemma:thetaprod}
  Let $0 < R_0 < R_1$,  $m>1$, $\gamma-d/p >
  2/(1-m)$, and $p\geq d(1-1/m)$. Then, the following
  ``interpolation'' estimate holds: there exists $\Cinterp>0$ such that
for all $R_0\leq D \leq R_1$, for all $\beta\in \mathbb{N}_0^d$, and for all $u$ with bounded
    $|\cdot|_{\mathcal{K}^{i}_{\gamma}(B_D)}$, $i = \betam, \betam+1$ seminorms,
  \begin{multline}
  \label{eq:thetaprod}
  \|r^{\frac{2-\gamma}{m}+|\beta|}\dbeta u \|_{L^{m p}(B_D)} \leq \Cinterp\|r^{\betam-\gamma} \dbeta u\|^{1-\theta}_{L^p(B_D)}
\left\{ \vphantom{\sum_{i=1}^d} (\betam+1)^\theta \| r^{\betam-\gamma}\dbeta u\|^\theta_{L^p(B_D)} \right. \\ 
\left.+ \sum_{\alpham=1} \| r^{\betam+1-\gamma} \partial^{\beta+\alpha} u \|^\theta_{L^p(B_D)}\right\},
  \end{multline}
   with $\theta = \frac{d}{p}\left(1-\frac{1}{m}\right)$.
\end{lemma}
\begin{proof}
  We start by proving \eqref{eq:thetaprod} with $D=1$.
  Consider a dyadic decomposition of $B_1$ given by the sets 
  \begin{equation*}
    B^j = \left\{ x\in B_1 : 2^{-j} \leq r(x)\leq 2^{-j+1} \right\}, \; j\in \mathbb{N}.
  \end{equation*}
   Let us introduce the
   linear maps $\phi_j : B^1\to B^j$ and indicate the pullback of functions
   by $\phi_j^{-1}$ as, e.g., $\hr = r\circ\phi_j^{-1}$ and $\hB = \phi_j^{-1}(B^j)$. Then, 
   \begin{equation*}
     \|r^{\frac{2-\gamma}{m}+\betam} \dbeta u\|_{L^{m p}(B^j)} \leq 
     2^{\frac{j}{m}(\gamma-2-d/p)}
     \|\hr^{\frac{2-\gamma}{m}+\betam} \hat{\partial}^\beta \hat{u}\|_{L^{m p}(\hB)} 
   \end{equation*}
   Now, for any set $S\subset\mathbb{R}^d$ satisfying the cone
     condition, for all $m\geq 1$, and all $1<p<\infty$, there exists $\CSinterp> 0$ such that, for all  $v\in W^{1, p}(S)$,
   \begin{equation*}
     \| v\|_{L^{m p}(S)} \leq \CSinterp \|v\|^{1-\theta}_{L^p(S)} \|v\|^\theta_{W^{1, p}(S)},
   \end{equation*}
   with $\theta$ defined as
   above, see Ref.~\refcite{DallAcqua2012}. Therefore,
   \begin{equation}
     \label{eq:interp-ref}
     \begin{multlined}[][\arraycolsep]
     \|r^{\frac{2-\gamma}{m}+\betam} \dbeta u\|_{L^{m p}(B^j)} 
      \\ \leq \CBinterp
     2^{\frac{j}{m}(\gamma-2-d/p)}
     \|\hr^{\frac{2-\gamma}{m}+\betam} \hat{\partial}^\beta \hat{u}\|^{1-\theta}_{L^{p}(\hB)} 
      \sum_{\alpham = 1} \|\hat{\partial}^\alpha \hr^{\frac{2-\gamma}{m}+\betam}\hat{\partial}^\beta \hat{u}\|^\theta_{L^p(\hB)}.
     \end{multlined}
   \end{equation}
   Let us now consider the first norm in the product above. Since $\hr :\hB\to
   (1/2,1)$, there holds
   \begin{equation*}
     \hr^{\frac{2-\gamma}{m}}\leq \max(1, 2^{\frac{2}{m} +|\gamma|(1-\frac{1}{m}) }) \hr^{-\gamma},
   \end{equation*}
   therefore,
     \begin{equation*}
     \|\hr^{\frac{2-\gamma}{m}+\betam} \hat{\partial}^\beta \hat{u}\|^{1-\theta}_{L^{p}(\hB)}  \leq  \max(1, 2^{\frac{2}{m} +|\gamma|(1-\frac{1}{m}) }) \|\hr^{\betam-\gamma} \hat{\partial}^\beta \hat{u}\|^{1-\theta}_{L^{p}(\hB)} .
\end{equation*}
We now compute more explicitly the second norm in the product in \eqref{eq:interp-ref}:
\begin{multline*}
      \sum_{\alpham = 1} \|\hat{\partial}^\alpha \hr^{\frac{2-\gamma}{m}+\betam}\hat{\partial}^\beta \hat{u}\|^\theta_{L^p(\hB) }\\
  \leq \left( \betam + \frac{2-\gamma}{m}\right)^\theta \|\hr^{\frac{2-\gamma}{m}+\betam-1} \hat{\partial}^\beta \hat{u}\|^\theta_{L^p(\hB) } 
+ \sum_{\alpham=1} \|\hr^{\frac{2-\gamma}{m}+\betam} \hat{\partial}^{\beta +\alpha}\hat{u}\|^\theta_{L^p(\hB) }
\end{multline*}
and we may adjust the exponents of $\hr$ and the term in
$\frac{2-\gamma}{m}$ introducing a constant $\Cinterpone$ that depends
only on $\gamma$, $m$, $d$ and $p$, obtaining
\begin{multline*}
      \sum_{\alpham = 1} \|\hat{\partial}^\alpha \hr^{\frac{2-\gamma}{m}+\betam}\hat{\partial}^\beta \hat{u}\|^\theta_{L^p(\hB) } 
      \\
\leq C\left( 
  \left( \betam + 1\right)^\theta \|\hr^{\betam -\gamma} \hat{\partial}^\beta \hat{u}\|^\theta_{L^p(\hB) } 
+ \sum_{\alpham=1} \|\hr^{\betam-\gamma+1} \hat{\partial}^{\beta +\alpha}\hat{u}\|^\theta_{L^p(\hB) } \right).
\end{multline*}
Denoting $\Cinterptwo = \Cinterpone\max(1, 2^{\frac{2}{m} +|\gamma|(1-\frac{1}{m}) }) $, scaling everything back to $B^j$ and adjusting the exponents,
\begin{multline*}
  \|r^{\frac{2-\gamma}{m}+\betam} \dbeta u\|_{L^{m p}(B^j)} \\
  \begin{aligned}
  &\leq 
  \Cinterptwo 2^{\frac{j}{m}\left((\gamma-d/p)(1-m)-2\right)}
  \|r^{\betam-\gamma} {\partial}^\beta u\|^{1-\theta}_{L^{p}(B^j)} 
\left\{ \vphantom{\sum_{i=1}^d}\left( \betam + 1\right)^\theta \|r^{\betam -\gamma} {\partial}^\beta {u}\|^\theta_{L^p(B^j) }  \right.
\\ & \qquad \left. + \sum_{\alpham=1} \|r^{\betam-\gamma+1} {\partial}^{\beta +\alpha}{u}\|^\theta_{L^p(B^j) }\right\}.
  \end{aligned}
\end{multline*}
If $\gamma-d/p > 2/(1-m)$ and therefore 
  $2^{\frac{j}{m}\left((\gamma-d/p)(1-m)-2\right)}\leq 1$, we can sum
  over all $j\in \mathbb{N}$ thus obtaining the estimate
  \eqref{eq:thetaprod} on the whole ball $B_{1}$. Indeed, denoting
    $\tau =\frac{1}{m}\left((\gamma-d/p)(1-m)-2\right) $
    \begin{align*}
      &\|r^{\frac{2-\gamma}{m}+\betam} \dbeta u\|_{L^{m p}(B_1)}\\
      &\qquad\leq \sum_{j=1}^\infty
      \|r^{\frac{2-\gamma}{m}+\betam} \dbeta u\|_{L^{m p}(B^j)}\\
&\qquad\begin{aligned}
  &\leq 
  \Cinterptwo \left( \sum_{j=1}^\infty2^{j\tau} \right)
  \max_{j\in \mathbb{N}}\bigg( \|r^{\betam-\gamma} {\partial}^\beta u\|^{1-\theta}_{L^{p}(B^j)} 
\left\{ \vphantom{\sum_{i=1}^d}\left( \betam + 1\right)^\theta \|r^{\betam -\gamma} {\partial}^\beta {u}\|^\theta_{L^p(B^j) }  \right.
\\ & \qquad \left. + \sum_{\alpham=1} \|r^{\betam-\gamma+1} {\partial}^{\beta +\alpha}{u}\|^\theta_{L^p(B^j) }\right\}\bigg).
\end{aligned} \\
&\qquad\begin{aligned}
  &\leq 
  \Cinterptwo \frac{1}{1-2^\tau}
  \|r^{\betam-\gamma} {\partial}^\beta u\|^{1-\theta}_{L^{p}(B_1)} 
\left\{ \vphantom{\sum_{i=1}^d}\left( \betam + 1\right)^\theta \|r^{\betam -\gamma} {\partial}^\beta {u}\|^\theta_{L^p(B_1) }  \right.
\\ & \qquad \left. + \sum_{\alpham=1} \|r^{\betam-\gamma+1} {\partial}^{\beta +\alpha}{u}\|^\theta_{L^p(B_1) }\right\}.
  \end{aligned}
    \end{align*}
    Then, writing $\Cinterpthree = \Cinterptwo \frac{1}{1-2^\tau}$,
    by scaling, we
    have that,  for all $D>0$,
    \begin{align*}
 & \|r^{\frac{2-\gamma}{m}+\betam} \dbeta u\|_{L^{m p}(B_D)}\\
  &\qquad \leq  
  D^{-\frac{2-\gamma}{n} -\frac{d}{np} +\gamma + \frac{d}{p}} \Cinterpthree
\|r^{\betam-\gamma} \dbeta u\|^{1-\theta}_{L^p(B_D)}
    \left\{ \vphantom{\sum_{i=1}^d} (\betam+1)^\theta \| r^{\betam-\gamma}\dbeta u\|^\theta_{L^p(B_D)} \right.
      \\ 
& \qquad \qquad\left.+ \sum_{\alpham=1} \| r^{\betam+1-\gamma} \partial^{\beta+\alpha} u \|^\theta_{L^p(B_D)}\right\}.
    \end{align*}
    Defining
    \begin{equation*}
      \Cinterp = \Cinterpthree \max\left( R_0^{-\frac{2-\gamma}{n} -\frac{d}{np} +\gamma + \frac{d}{p}} ,R_1^{-\frac{2-\gamma}{n} -\frac{d}{np} +\gamma + \frac{d}{p}}  \right)
    \end{equation*}
    concludes the proof.
\end{proof}
\subsection{Analyticity of solutions}
\label{sec:nonlinsch-analytic}
We now consider the nonlinear Schr\"{o}dinger eigenvalue problem
\eqref{eq:eq-cont} with polynomial nonlinearity, given by
  \begin{equation}
    \label{eq:analytic-problem}
    Lu = -\Delta u + V u + |u|^{n-1} u = \lambda u.
  \end{equation}
  We suppose that the potential $V$ is singular on a finite set of discrete
  points and consider the case of an up-to-quartic nonlinearity (i.e.,
  $n\in \mathbb{N}$ and $n \leq 4$). We show, in the following
  theorem, that the results on the regularity of the solution that can be
  obtained in the linear case \cite{linear} can be extended to the nonlinear
  one. We recall that $\epsilon \in (0,1)$ so that $V\in
  \mathcal{K}^{\varpi, \infty}_{\epsilon-2}(\Omega)$.
\begin{theorem}
  \label{theorem:analytic}
  Let $\epsilon \in (0,1)$.
  Let $u$ be the solution to \eqref{eq:analytic-problem} with $V\in \mathcal{K}^{\varpi,
      \infty}_{\epsilon-2}(\Omega)$ and $n \in\{2,3,4\}$. Then,
  \begin{equation}
    \label{eq:analytic}
    u\in \mathcal{J}^{\varpi}_{\gamma}(\Omega)
  \end{equation}
  for any $\gamma < \epsilon + d/2$.
\end{theorem}
In order to prove the analyticity in weighted spaces of the function $u$ we need
to bound the nonlinear term. We will introduce some preliminary lemmas and
proceed by induction: let us specify the induction hypothesis.
\par\addvspace{7pt plus3pt minus2pt}%
\noindent\textbf{Induction Assumption.}
\hskip.5em {\itshape For $k\in\mathbb{N}$, $1<p<\infty$, $C_u, A_u>0$,
  $\gamma\in \mathbb{R}$, and $R>0$ such that $B_R\subset \Omega$,
    we say that $H_u(k, p, \gamma, C_u, A_u)$  holds in $B_R$ if for all $\rho
    \in (0, R/(2k)]$, there holds
  $u\in \mathcal{J}^{k,p}_\gamma(B_R)$
  and 
  \begin{equation}
    \label{eq:inductionhyp}
   \sum_{\alpham = j}\|r^{\alpham-\gamma}\dalpha u \|_{L^p(B_{R-k \rho})} \leq 
C_u A_u^j (k\rho)^{-j} j^j \quad \text{for all } \lfloor  \gamma-d/p \rfloor+1\leq  j \leq k.
  \end{equation}
  }
\par\addvspace{8pt plus3pt minus2pt}
\begin{lemma}
  \label{lemma:nonlin1}
  Let $R>0$ such that $B_R\subset \Omega$, let $m\in \{2,3,4\}$, let
    $C_u>0$, let $1<p <\infty$ and $\gamma\in \mathbb{R}$ such that
    $2/(1-m)<\gamma - d/p <\epsilon$ and $p\geq d(1-1/m)$. Then,
    there exists $C_1>0$ such that, for all
  $k\in \mathbb{N}$ such that $k\geq 2$, for all $A_u>0$, for all $u$ such that $H_u(k, p,
  \gamma, C_u, A_u)$ holds in $B_R$, for all $1\leq j \leq k-1$, and for all $0<\rho\leq \frac{R}{2k}$,

\begin{equation}
   \sum_{\betam=j}\| r^{\frac{2-\gamma}{n}+\betam} \dbeta u \|_{L^{m p} (B_{R-k\rho})} 
\leq C_{1} A_u^{j+\theta} (k\rho)^{-j-\theta} j^j (j+1)^\theta,
\end{equation}
with $\theta = \frac{d}{p}\left(1-\frac{1}{m}\right)$.
\end{lemma}
\begin{proof}
First, we use \eqref{eq:thetaprod} in order to go back to integrals in $L^p$:
in particular, with respect to \eqref{eq:thetaprod} we have $D=R-k\rho \in
  [R/2, R)$ and since $H_u(k, p, \gamma, C_u, A_u)$ holds in $B_R$ with $k\geq
  2$, it follows that $|u|_{\mathcal{K}^{i, p}_\gamma(B_R)}<\infty$ for
  $i=\betam, \betam+1$. Then,
   \begin{multline*}
   \| r^{\frac{2-\gamma}{m}+\betam} \dbeta u \|_{L^{m p} (B_{R-k\rho})}  \leq\Cinterp\|r^{\betam-\gamma} \dbeta u\|^{1-\theta}_{L^p(B_{R-k\rho})}
\left\{ \vphantom{\sum_{i=1}^d} (\betam+1)^\theta \| r^{\betam-\gamma}\dbeta u\|^\theta_{L^p(B_{R-k\rho})} \right. \\ 
\left.+ \sum_{i=1}^d \| r^{\betam+1-\gamma} \partial^{\beta}\partial_i u \|^\theta_{L^p(B_{R-k\rho})}\right\}.
   \end{multline*}
By the Cauchy-Schwarz inequality,
     \begin{multline*}
        \sum_{\betam = j} \| r^{\betam -\gamma}\dbeta u\|^{1-\theta}_{L^p(B_{R-k\rho})}
(\betam+1)^\theta \| r^{\betam-\gamma}\dbeta u\|^\theta_{L^p(B_{R-k\rho})}
         \\
\leq \left(\sum_{\betam=j} \| r^{\betam -\gamma}\dbeta u\|_{L^p(B_{R-k\rho})}  \right)^{1-\theta}
              \left(\sum_{\betam=j}
  (\betam+1) \| r^{\betam-\gamma}\dbeta u\|_{L^p(B_{R-k\rho})}
  \right)^\theta
     \end{multline*}
     and, 
     \begin{multline*}
       \sum_{\betam = j} \left( \| r^{\betam -\gamma}\dbeta u\|^{1-\theta}_{L^p(B_{R-k\rho})}
\sum_{i=1}^d\| r^{\betam+1-\gamma} \partial^{\beta}\partial_i u \|^\theta_{L^p(B_{R-k\rho})} \right)
         \\
\leq \sum_{i=1}^d\left(\sum_{\betam=j} \| r^{\betam -\gamma}\dbeta u\|_{L^p(B_{R-k\rho})}  \right)^{1-\theta}
              \left(\sum_{\betam=j}
\| r^{\betam+1-\gamma} \partial^{\beta}\partial_i u \|_{L^p(B_{R-k\rho})}
  \right)^\theta
     \end{multline*}
   Then, hypothesis \eqref{eq:inductionhyp} implies, for $j=1, \dots, k-1$,
   \begin{equation*}
     \left(\sum_{\betam = j}\|r^{j-\gamma} \dbeta u\|_{L^p(B_{R-k\rho})}  \right)^{1-\theta} \leq
     C_u^{1-\theta}A_u^{{j} (1-\theta)} \rho^{-{j} (1-\theta)} \left( \frac{{j}}{k}\right)^{{j} (1-\theta)}
   \end{equation*}
   and 
   \begin{multline*}
 ( j+1)^\theta\left( \sum_{\betam = j}  \| r^{\betam-\gamma}\dbeta u\|_{L^p(B_{R-k\rho})}  \right)^{\theta}
+ \sum_{i=1}^d \left(\sum_{\betam=j}\| r^{\betam+1-\gamma} \partial^{\beta}\partial_i u \|_{L^p(B_{R-k\rho})}  \right)^{\theta}
\\
\leq C_{u}^\theta ({j}+1)^\theta A_u^{{j} \theta} \rho^{-{j}\theta} \left( \frac{{j}}{k}\right)^{{j}\theta}
+dC_{u}^\theta A_u^{({j} +1)\theta} \rho^{-({j}+1)\theta} \left( \frac{{j}+1}{k}\right)^{({j}+1)\theta}.
   \end{multline*}
   Therefore, multiplying the right hand sides of the last two inequalities,
   \begin{equation*}
   \sum_{\betam = j}\| r^{\frac{2-\gamma}{m}+\betam} \dbeta u \|_{L^{m p} (B_{R-k\rho})}  \leq
   \Cinterp(d+1)C_u A_u^{{j} +\theta} (k\rho)^{-{j} -\theta} {j}^{{j} (1-\theta)} ({j}+1)^{({j}+1)\theta}.
   \end{equation*}
   We finally need to bound the last two terms in the multiplication above. By
   Stirling's formula, there exists $C>0$ such that
   \begin{equation*}
     {j}^{{j} (1-\theta)} ({j}+1)^{({j}+1)\theta} \leq C {j} ! {j}^{-1/2} e^{j} (j+1)^{\theta/2} j^{\theta/2}, \qquad \forall j\in \mathbb{N},
   \end{equation*}
   and another application of Stirling's formula gives the thesis.
\end{proof}
In order to estimate the $L^p$ weighted norms of derivatives of $u^{n}$ we
will use Leibniz's rule and break the $L^p$ norms into multiple $L^{n p}$
norms. Lemma \ref{lemma:nonlin1} then allows to go back to the induction
hypothesis. We continue by estimating the weighted norms of $u^2$ through the
procedure we just outlined. For two multi indices $\alpha = (\alpha_1, \dots,
\alpha_d)\in \mathbb{N}^d_0$ and $\beta = (\beta_1, \dots,
\beta_d)\in\mathbb{N}^d_0$, we write $\alpha! =\alpha_1!\cdots \alpha_d!$, $\alpha
+ \beta = (\alpha_1+\beta_1, \dots, \alpha_d+ \beta_d)$, and
\begin{equation*}
  \binom{\alpha}{\beta} = \frac{\alpha!}{\beta! (\alpha-\beta)!}.
\end{equation*}
Furthermore, we write $\alpha < \beta$ if $\alpha_i\leq \beta_i$ for all
  $i\in \{1, \dots, d\}$ and there exists at least one index $j\in \{1, \dots,
  d\}$ such that $\alpha_j < \beta_j$. We write $\alpha > 0$ if there exists at
  least one $j\in \{1, \dots, d\}$ such that $\alpha_j>0$ (i.e., if
  $\alpham>0$). We indicate
  \begin{equation*}
    \sum_{0<\alpha<\beta}= \sum_{\substack{\alpha\in \mathbb{N}^d_0 : \\\alpha>0\land\alpha<\beta}}.
  \end{equation*}

Recall \cite{Kato1996} that
\begin{equation}
  \label{eq:binom}
  \sum_{\substack{\betam = n\\\beta\leq\alpha}} \binom{\alpha}{\beta} = \binom{\alpham}{n}.
\end{equation}
We introduce the following statement on sums over multi indices.
\begin{lemma}
For all $k\in \mathbb{N}$, $k\geq 2$, and for all $a, b:
  \mathbb{N}^d_0 \to \mathbb{R}$, there holds  
  \begin{equation}
  \label{eq:sums}  
  \sum_{\betam = k} \sum_{0<\zeta<\beta} a(\zeta) b(\beta-\zeta)
  =   \sum_{j=1}^{k-1}\sum_{\zetam = j} \sum_{\xim = k-j} a(\zeta) b(\xi).
  \end{equation}
\end{lemma}
\begin{proof}
  First, remark that
    \begin{align*}
      \sum_{\betam = k} \sum_{0<\zeta<\beta} a(\zeta) b(\beta-\zeta)
      &= \sum_{\betam=k}\sum_{j=1}^{k-1} \sum_{\substack{\zeta: \zeta<\beta\\ \zetam=j}}a(\zeta) b(\beta-\zeta)
      = \sum_{j=1}^{k-1}\sum_{\betam=k} \sum_{\substack{\zeta: \zeta<\beta\\ \zetam=j}}a(\zeta) b(\beta-\zeta)\\
      &= \sum_{j=1}^{k-1}\sum_{\zetam = j} \sum_{\substack{\beta: \beta>\zeta\\ \betam = k}}a(\zeta) b(\beta-\zeta).
    \end{align*}
    Then, \eqref{eq:sums} follows by replacing $\beta$ by $\xi+\zeta$ in the
    last sum at the right hand side, and remarking that, for all $k\in
    \mathbb{N}$, $k\geq 2$ and for all $\zeta\in\mathbb{N}^d_0$ such that $1\leq
    \zetam\leq k-1$,
    \begin{equation*}
      \left\{\beta \in \mathbb{N}^d_0: \betam = k, \beta>\zeta\right\}
      =
      \left\{\zeta+\xi \in \mathbb{N}^d_0: \xim = k - \zetam\right\}.
    \end{equation*} \end{proof}
\begin{lemma}
  \label{lemma:nonlin2}
  Let $R>0$ such that $B_R\subset \Omega$, let $m\in \{2,3,4\}$, let
    $C_u>0$, let $1<p <\infty$ and $\gamma\in \mathbb{R}$ such that
    $2/(1-m)<\gamma - d/p <\epsilon$ and $p\geq d(1-1/m)$. Then,
    there exists $C_2>0$ such that, for all
  $k\in \mathbb{N}$ such that $k\geq 2$, for all $A_u>0$, for all $u$ such that $H_u(k, p,
  \gamma, C_u, A_u)$ holds in $B_R$, for all $1\leq j \leq k-1$, and for all $0<\rho\leq \frac{R}{2k}$,
 \begin{equation}
   \label{eq:nonlin2}
   \sum_{\alpham =j}\| r^{2\frac{2-\gamma}{m}+\alpham} \dalpha (u^2) \|_{L^{m p/2} (B_{R-k\rho})} 
\leq C_2 A_u^{{j}+2\theta} \rho^{-{j}-2\theta} \left(\frac{{j}}{k}\right)^{j} {j}^{1/2}.
 \end{equation}
\end{lemma}
\begin{proof}
 By Leibniz's rule and the Cauchy-Schwarz inequality,
 \begin{multline}
   \label{eq:leibniz1}
   \sum_{\alpham =j}\| r^{2\frac{2-\gamma}{m}+\alpham} \dalpha (u^2) \|_{L^{m p/2} (B_{R-k\rho})} 
   \\
   \begin{aligned}
&\leq  
\sum_{\alpham =j}\sum_{0< \beta < \alpha} \binom{\alpha}{\beta} \| r^{\frac{2-\gamma}{m}+\betam} \dbeta u \|_{L^{m p} (B_{R-k\rho})} \| r^{\frac{2-\gamma}{m}+\alpham-\betam} \partial^{\alpha-\beta} u \|_{L^{m p} (B_{R-k\rho})} 
\\
&\quad +  2\sum_{\alpham =j}\| r^{2\frac{2-\gamma}{m}+\alpham} \dalpha u \|_{L^{m p/2} (B_{R-k\rho})} \|  u \|_{L^{\infty} (B_{R-k\rho})} \\
   \end{aligned}
 \end{multline}
 Considering the sum over $0<\beta <\alpha$, Lemma \ref{lemma:nonlin1}, \eqref{eq:sums}, and Stirling's inequality give
 \begin{multline*}
\sum_{\alpham =j}\sum_{0<\beta <\alpha} \binom{\alpha}{\beta} \| r^{\frac{2-\gamma}{m}+\betam} \dbeta u \|_{L^{m p} (B_{R-k\rho})} \| r^{\frac{2-\gamma}{m}+\alpham-\betam} \partial^{\alpha-\beta} u \|_{L^{m p} (B_{R-k\rho})} \\
\begin{aligned}
&\leq C A_u^{{j} +2\theta} \rho^{-{j}-2\theta} \sum_{i=1}^{{j}-1} \binom{{j}}{i} i!({j} - i)! e^{j} \frac{(i+1)^\theta ({j}-i+1)^\theta}{k^{{j}+2\theta}}\frac{1}{ \sqrt{i({j}-i)}}\\
&\leq C A_u^{{j} +2\theta} \rho^{-{j}-2\theta} \frac{{j}!e^{j}}{k^{j}}\\
&\leq C A_u^{{j} +2\theta} \rho^{-{j}-2\theta} \left(\frac{{j}}{k}\right)^{j}{j}^{1/2}.
\end{aligned}
 \end{multline*}
The second term at the right hand side of \eqref{eq:leibniz1} is controlled
using Lemma \ref{lemma:nonlin1}, since $\gamma -d/p< 2$, and the injection
$\mathcal{J}^2_{d/2+\eta}(\Omega)\hookrightarrow L^\infty(\Omega)$, valid for
any $\eta>0$ \cite{Kozlov1997}.
\end{proof}
With the same proof as above, we can deal with a cubic nonlinear term, as we show in the following lemma.
\begin{lemma}
  \label{lemma:nonlin3}
Let $R>0$ such that $B_R\subset \Omega$, let $m\in \{3,4\}$, let
    $C_u>0$, let $1<p <\infty$ and $\gamma\in \mathbb{R}$ such that
    $2/(1-m)<\gamma - d/p <\epsilon$ and $p\geq d(1-1/m)$. Then,
    there exists $C_3>0$ such that, for all
  $k\in \mathbb{N}$ such that $k\geq 2$, for all $A_u>0$, for all $u$ such that $H_u(k, p,
  \gamma, C_u, A_u)$ holds in $B_R$, for all $1\leq j \leq k-1$, and for all $0<\rho\leq \frac{R}{2k}$,
 \begin{equation}
   \label{eq:nonlin3}
  \sum_{\alpham =j} \| r^{3\frac{2-\gamma}{m}+\alpham} \dalpha (u^3) \|_{L^{m p/3} (B_{R-k\rho})} 
\leq C_3 A_u^{{j}+3\theta} \rho^{-{j}-3\theta} \left(\frac{{j}}{k}\right)^{j} {j}.
 \end{equation}
\end{lemma}
\begin{proof}
We have  
 \begin{multline*}
  \sum_{\alpham =j}   \| r^{3\frac{2-\gamma}{m}+\alpham} \dalpha (u^3) \|_{L^{m p/3} (B_{R-k\rho})} 
   \\ \leq C 
\sum_{\alpham =j}\sum_{\beta \leq \alpha} \binom{\alpha}{\beta} \| r^{\frac{2-\gamma}{m}+\betam} \dbeta u \|_{L^{m p} (B_{R-k\rho})} \| r^{2\frac{2-\gamma}{m}+\alpham-\betam} \partial^{\alpha-\beta} (u^2) \|_{L^{m p/2} (B_{R-k\rho})} .
 \end{multline*}
 Using \eqref{eq:nonlin2} we follow the same procedure as in the proof of Lemma \ref{lemma:nonlin2}. When $0<\beta<\alpha$ in the sum above,
\begin{align*}
&\sum_{\alpham =j} \sum_{0< \beta <\alpha} \binom{\alpha}{\beta} \| r^{\frac{2-\gamma}{m}+\betam} \dbeta u \|_{L^{m p} (B_{R-k\rho})} \| r^{2\frac{2-\gamma}{m}+\alpham-\betam} \partial^{\alpha-\beta} (u^2) \|_{L^{m p/2} (B_{R-k\rho})} \\
&\qquad\leq C A_u^{{j} +3\theta} \rho^{-{j}-3\theta} \sum_{i=1}^{{j}-1} \binom{{j}}{i} i!({j} - i)! e^{j} \frac{(i+1)^\theta ({j}-i+1)^\theta}{k^{{j}+2\theta}}\frac{\sqrt{{j} -i}}{ \sqrt{i({j}-i)}}\\
&\qquad\leq C A_u^{{j} +3\theta} \rho^{-{j}-3\theta} \frac{{j}!e^{j}\sqrt{{j}}}{k^{j}}\\
&\qquad\leq C A_u^{{j} +3\theta} \rho^{-{j}-3\theta} \left(\frac{{j}}{k}\right)^{j}{j}.
 \end{align*}
 As before, the terms in the sum where $\beta = 0$ and $\beta=\alpha$ give the same bound.
\end{proof}
The proof of the next lemma, in which we control a quartic term, amounts to a
repetition of the arguments above; we show its proof for completeness.
\begin{lemma}
  \label{lemma:nonlin4}
Let $R>0$ such that $B_R\subset \Omega$, let
    $C_u>0$, let $1<p <\infty$ and $\gamma\in \mathbb{R}$ such that
    $-2/3<\gamma - d/p <\epsilon$ and $p\geq 3d/4$. Then,
    there exists $C_4>0$ such that, for all
  $k\in \mathbb{N}$ such that $k\geq 2$, for all $A_u>0$, for all $u$ such that $H_u(k, p,
  \gamma, C_u, A_u)$ holds in $B_R$, for all $1\leq j \leq k-1$, and for all $0<\rho\leq \frac{R}{2k}$,
 \begin{equation}
   \label{eq:nonlin4}
  \sum_{\alpham =j}  \| r^{{2-\gamma}+\alpham} \dalpha (u^4) \|_{L^{p} (B_{R-k\rho})} 
\leq C_4 A^{{j}+4\theta} \rho^{-{j}-4\theta} \left(\frac{{j}}{k}\right)^{j} {j}^{3/2}.
 \end{equation}
\end{lemma}
\begin{proof}
There holds
 \begin{multline*}
  \sum_{\alpham =j}   \| r^{{2-\gamma}+\alpham} \dalpha (u^4) \|_{L^{p} (B_{R-k\rho})} 
   \\ \leq C 
 \sum_{\alpham =j} \sum_{\beta \leq \alpha} \binom{\alpha}{\beta} \| r^{\frac{2-\gamma}{4}+\betam} \dbeta u \|_{L^{4 p} (B_{R-k\rho})} \| r^{3\frac{2-\gamma}{4}+\alpham-\betam} \partial^{\alpha-\beta} (u^3) \|_{L^{4p/3} (B_{R-k\rho})} .
 \end{multline*}
 We can now use the result of Lemma \ref{lemma:nonlin3} with $m= 4$. When $0<\beta<\alpha$ in the sum above,
\begin{align*}
& \sum_{\alpham =j} \sum_{0< \beta <\alpha} \binom{\alpha}{\beta} \| r^{\frac{2-\gamma}{4}+\betam} \dbeta u \|_{L^{4p} (B_{R-k\rho})} \| r^{3\frac{2-\gamma}{4}+\alpham-\betam} \partial^{\alpha-\beta} (u^3) \|_{L^{4p/3} (B_{R-k\rho})} \\
&\qquad\leq C A_u^{{j} +4\theta} \rho^{-{j}-4\theta} \sum_{i=1}^{{j}-1} \binom{{j}}{i} i!({j} - i)! e^{j} \frac{(i+1)^\theta ({j}-i+1)^\theta}{k^{{j}+2\theta}}\frac{\sqrt{{j} -i}}{ \sqrt{i({j}-i)}}\\
&\qquad\leq C A_u^{{j} +4\theta} \rho^{-{j}-4\theta} \left(\frac{{j}}{k}\right)^{j}{j}.
 \end{align*}
 The direct application of Lemmas \ref{lemma:nonlin1} and \ref{lemma:nonlin3}
 let us obtain the estimate for the terms in the sum where $\beta= \alpha$ and
 $\beta =0$, respectively.
\end{proof}
The lemma below is, then, a finite weighted regularity estimate that we use as
the basis for induction  in the proof of Theorem \ref{theorem:analytic}.
\begin{lemma}
  \label{lemma:base-reg}
  Let $\epsilon\in(0,1)$.
  The solution $u$ to problem \eqref{eq:analytic-problem}, with $n\in\{2,3,4\}$
  and $V\in \mathcal{K}^{\varpi, \infty}_{\epsilon-2}(\Omega)$,
  is such that $u\in\mathcal{J}^{2,p}_\gamma(\Omega)$, for all $1<
  p<\infty$ and for all $0 < \gamma - d/p < \epsilon$.
\end{lemma}
\begin{proof}
 The operator $L_\mathrm{lin} = -\Delta +V$ is an isomorphism 
\begin{equation}
  \label{eq:isomorph}
  \mathcal{J}^{k+2}_{\xi}(\Omega) \to \mathcal{J}^{k}_{\xi-2}(\Omega)
\end{equation}
for any $0<\xi-d/2 < \epsilon$ and all $k\in\mathbb{N}$, since $\Omega$ is a
compact set without boundary, see Lemma 3.1 of
Ref.~\refcite{linear}.
Since we can also show that $u\in L^\infty(\Omega)$ \cite{Stampacchia1965}, the solution to \eqref{eq:analytic-problem} is such that $u\in \mathcal{J}^2_{\xi}(\Omega)$. 
Iterating this line of reasoning, we have, from $u\in L^\infty(\Omega)\cap
H^1(\Omega)$ that $u^n\in H^1(\Omega)\subset \mathcal{J}^1_{\xi-2}(\Omega)$, hence
$u\in\mathcal{J}^3_{\xi}(\Omega)$ for all $0<\xi-d/2<\epsilon$.

We now
  claim that, by injection, for all $p>1$ and all $0< \gamma -
  d/p<\epsilon$ there holds
$u\in\mathcal{J}^{1,p}_{\gamma}(\Omega)$. Remark that, since
  $u\in\mathcal{J}^3_{\xi}(\Omega)$, we have \cite{Costabel2010a} that
  $u-u(\fc)\in \cK^3_{\xi}(\Omega)$. Then, by Lemma 1.2.2 of Ref.~\refcite{Mazya2010},
  $u-u(\fc) \in \cK^{1, p}_{\gamma}(\Omega)$ for all $p>1$ and for all
  $\gamma\leq \xi-d/2+d/p$.
  Since furthermore $|u(\fc)|\leq C \| u \|_{\cJ^2_{\eta}(\Omega)}$ for a $C$
  independent of $u$ and for any $\eta>d/2$ \cite{Kozlov1997}, we obtain that
  $u\in\mathcal{J}^{1,p}_{\gamma}(\Omega)$, for all $p>1$ and all $0< \gamma -
  d/p<\epsilon$.

 We now show that $u\in \mathcal{J}^4_\xi(\Omega)$ for all $0<\xi-d/2<\epsilon$.
 This will follow from the fact that $L_{\mathrm{lin}}$ is an isomorphism on the
 spaces \eqref{eq:isomorph} and from the fact that $u^n\in \mathcal{J}^2_{2-\xi}(\Omega)$.
 Indeed, we have already shown that $u^n \in \mathcal{J}^1_{2-\xi}(\Omega)$; in
 addition, for any $\alpha\in \mathbb{N}^d_0$, $\alpham=2$,
 \begin{align*}
   &\|r^{4-\xi} \dalpha (u^n) \|_{L^2(\Omega)} \\
   & \qquad \leq n(n-1) \|u\|_{L^\infty(\Omega)}^{n-2} \|r^{4-\xi} (\dbeta u)(\dzeta u)\|_{L^2(\Omega)} + n \|u\|^{n-1}_{L^\infty(\Omega)}\|r^{4-\xi} \dalpha u\|_{L^2(\Omega)} \\
   &\qquad  = (I) + (II),
 \end{align*}
 for suitable $\betam = \zetam = 1$. The norm in $(II)$ is bounded since $u\in
 \mathcal{J}^2_\xi(\Omega)$; to treat term $(I)$, we note that
 \begin{align*}
   \|r^{4-\xi} (\dbeta u)(\dzeta u)\|_{L^2(\Omega)} &\leq
   \|r^{2-\xi/2} \dbeta u\|_{L^4(\Omega)} \|r^{2-\xi/2} \dzeta u\|_{L^4(\Omega)}\\
   & \leq
   | u|_{\mathcal{K}^{1,4}_{\xi/2-1}(\Omega)}^2.
 \end{align*}
 Combining that, from above, $u\in \mathcal{J}^{1, 4}_\gamma(\Omega)$ for all
 $\gamma < \epsilon +d/4$ and that $\xi/2 - 1 < \epsilon/2 + d/4 - 1 <
 \epsilon+d/4$, we obtain that term $(I)$ is bounded, too. Therefore, $u^n\in
 \mathcal{J}^{2}_{\xi-2}(\Omega)$ and this implies $u\in
 \mathcal{J}^4_\xi(\Omega)$, for all $0< \xi-d/2 < \epsilon$.
 
We conclude the same procedure as above: $u - u(\fc) \in
\mathcal{K}^4_\xi(\Omega)$ by Ref.~\refcite{Costabel2010a}, hence $u-u(\fc)
\in \mathcal{K}^{2, p}$ for all $1<p<\infty$ and $\gamma < \epsilon +d/p$,
whence $u\in \mathcal{J}^{2,p}_\gamma(\Omega)$ for all $1<p<\infty$ and $0 <\gamma -d/p<\epsilon$.
\end{proof}

The proof of \eqref{eq:analytic} is now complete: we just need to bring the estimates together.
\begin{proof}(Proof of Theorem \ref{theorem:analytic})
Since $V\in \mathcal{K}^{\varpi, \infty}_{\epsilon-2}(\Omega)$, we denote by
  $C_V, A_V$ the constants such that $\sum_{\alpham = i}\|r^{\alpham-\epsilon+2}\dalpha
  V\|_{L^\infty(\Omega)}\leq C_V A_V^i i!$ for all $i\in \mathbb{N}_0$.

We proceed by induction and impose a restriction on $p$; specifically, we
  fix $\pstar$ and $\gammastar$ such that
\begin{equation}
  \label{eq:p-cond}
  \pstar \geq \min\left( 2, 2d \frac{n-1}{5-n} \right),
  \qquad
  \gammastar\in (d/p, d/p+\epsilon).
\end{equation}
Let us also fix $0< \Rstar < 1$ such that $B_\Rstar\subset \Omega$.
  By Lemma \ref{lemma:base-reg}, the induction
  assumption $H_u(2, \pstar, \gammastar, C_{0}, A_{0})$ holds
  in $B_\Rstar$ for some constants $C_0, A_0>0$.
We introduce a constant $C_u> 1$ such that
\begin{equation}
  \label{eq:Cu}
C_u \geq \max\left(  \| u\|_{\mathcal{J}^{2,\pstar}_{\gammastar}(\Omega)}, 
  \Creg(6+|\lambda|)\right),
\end{equation}
where $\Creg$ is the constant in \eqref{eq:elliptic-weighted},
and a constant $A_u >1$ such that
\begin{equation}
  \label{eq:Au}
  A_u \geq \max\left(A_V, \sqrt{2}C_V, \frac{4\pi}{\pstar(\epsilon-\gammastar)+d} C_V,  C_u  \right)\quad \text{and}\quad A_u^{2-n\theta}\geq C_ne,
\end{equation}
  where $C_n$ is the constant defined in one of the Lemmas \ref{lemma:nonlin2} to
  \ref{lemma:nonlin4} and $\theta = \frac{d}{\pstar}\left(
    1-\frac{1}{n} \right)$. Note that, since $2-n\theta > 0$,
  a constant $A_u$ satisfying \eqref{eq:Au} exists.

Suppose now that  $H_u(k, \pstar, \gammastar, C_u, A_u)$ holds in $B_\Rstar$ for a
$k\in\mathbb{N}$, $k\geq 2$.
 We will show that $H_u(k+1, \pstar,
 \gammastar, C_u, A_u)$ holds in $B_{\Rstar}$.
 Start by considering that, for all $\rho \in (0, \frac{\Rstar}{2(k+1)}]$ there exists
   $\widetilde{\rho} =\frac{k+1}{k}\rho$ such that $\widetilde{\rho} \in (0,
   \frac{\Rstar}{2k}]$ and
   \begin{equation*}
     \sum_{\alpham = j}\|r^{\alpham-\gammastar}\dalpha u \|_{L^\pstar(B_{\Rstar-(k+1) \rho})}  
     =
     \sum_{\alpham = j}\|r^{\alpham-\gammastar}\dalpha u \|_{L^\pstar(B_{\Rstar-k \widetilde{\rho}})}  \qquad j = 1, \dots, k.
   \end{equation*}
   Hence, since $H_u(k, \pstar, \gammastar, C_u, A_u)$ holds, we have that
   \begin{equation*}
     \sum_{\alpham = j}\|r^{\alpham-\gammastar}\dalpha u \|_{L^\pstar(B_{\Rstar-(k+1) \rho})}  
     \leq 
C_u A_u^j (k\widetilde{\rho})^{-j} j^j 
\leq
C_u A_u^j ((k+1)\rho)^{-j} j^j \quad \text{for all } 1\leq  j \leq k,
   \end{equation*}
   for all $0<\rho\leq\frac{\Rstar}{2(k+1)}$. It remains to prove that,
 for all  $0<\rho\leq\frac{\Rstar}{2(k+1)}$,
\begin{equation}
  \label{eq:inductionrepeat}
  \begin{aligned}
\sum_{\alpham = k+1}\| r^{\alpham -\gammastar} \dalpha u \|_{L^\pstar(B_{\Rstar-(k+1)\rho})}  &\leq
 C_u A_u^{k+1} ((k+1)\rho)^{-(k+1)} (k+1)^{k+1}\\
 &=C_u A_u^{k+1} \rho^{-(k+1)}.
  \end{aligned}
  \end{equation}
From
\eqref{eq:elliptic-weighted} and \eqref{eq:analytic-problem}, there exists
  a constant $\Creg$ independent of $u$, $k$, and $\rho$ such that
\begin{multline}
  \label{eq:analytic-proof1}
    \sum_{|\alpha| = k+1} \| r^{k+1-\gammastar} \dalpha u\|_{L^\pstar(B_{\Rstar-(k+1)\rho})}
    \leq \Creg\left(  \sum_{|\beta| = k-1} \| r^{k+1-\gammastar}\dbeta(Vu +
      |u|^{n-1} u + \lambda u)\|_{L^\pstar(B_{\Rstar-k\rho})} \right.\\
    \left.+ \sum_{|\alpha| = k-1, k}\rho^{|\alpha|-k-1}\|r^{|\alpha|-\gammastar}\dalpha u\|_{L^\pstar(B_{\Rstar-|\alpha|\rho})} \right).
\end{multline}
We consider the term containing the potential $V$: 
\begin{multline}
  \label{eq:Vu}
 \sum_{\betam=k-1} \| r^{k+1-\gammastar}\dbeta (Vu) \|_{L^\pstar(B_{\Rstar-k\rho})}  \\
  \begin{aligned}
& \leq\sum_{\betam=k-1}\sum_{0 <\zeta < \beta} \binom{\beta}{\zeta} \| r^{2-\epsilon+\zetam} \dzeta V \|_{L^{\infty} (B_{\Rstar-k\rho})} \| r^{\epsilon-\gammastar+\betam-\zetam} \partial^{\beta-\zeta} u \|_{L^{p} (B_{\Rstar-k\rho})} \\
&\qquad+ \| r^{2-\epsilon} V \|_{L^{\infty} (B_{\Rstar-k\rho})}  \sum_{\betam=k-1}\| r^{\epsilon-\gammastar+\betam} \partial^{\beta} u \|_{L^{p} (B_{\Rstar-k\rho})} \\
&\qquad+ \sum_{\betam=k-1}\| r^{2-\epsilon+\betam} \dbeta V \|_{L^{\infty} (B_{\Rstar-k\rho})} \| r^{\epsilon-\gammastar}  u \|_{L^{p} (B_{\Rstar-k\rho})} ,
  \end{aligned}
\end{multline}
Using \eqref{eq:binom} and
  \eqref{eq:sums},
\begin{multline}
\label{eq:Vu-estim}
\sum_{\betam=k-1}\sum_{0 <\zeta < \beta} \binom{\beta}{\zeta} \| r^{2-\epsilon+\zetam} \dzeta V \|_{L^{\infty} (B_{\Rstar-k\rho})} \| r^{\epsilon-\gammastar+\betam-\zetam} \partial^{\beta-\zeta} u \|_{L^{p} (B_{\Rstar-k\rho})}  \\
\begin{aligned}
&\leq C_V \Rstar^\epsilon
 \sum_{j=1}^{k-2} \binom{k-1}{j} A_u^{k-j-1}A_V^j j!(k\rho)^{-(k-j-1)}\left( k-j-1\right)^{k-j-1}\\
&\leq C_V\Rstar^\epsilon
A_u^{k-1} (k\rho)^{-(k-1)} (k-1)! e^{k-1}\sum_{j=1}^{k-2} \frac{(k\rho/e)^j}{\sqrt{k-1-j}}\\
&\leq C_V\Rstar^\epsilon
A_u^{k-1} (k\rho)^{-(k-1)} (k-1)^{k-1}\left( \frac{k-1}{k-2} \right)^{1/2}\frac{\Rstar/(2e)}{1-\Rstar/(2e)}\\
&\leq C_V\sqrt{2}
A_u^{k-1} (k\rho)^{-(k-1)} (k-1)^{k-1}
\leq 
A_u^{k} (k\rho)^{-(k-1)}  (k-1)^{k-1}
\leq A_u^k \rho^{-k+1}
\end{aligned}
\end{multline}
where we have used Stirling's inequality twice, the fact that $k\rho\leq
  \Rstar/2$, and we have concluded using \eqref{eq:Au}. 
The bound on the second term in \eqref{eq:Vu} is straightforward and gives
\begin{align*}
  \| r^{2-\epsilon} V \|_{L^{\infty} (B_{\Rstar-k\rho})}  \sum_{\betam=k-1}\| r^{\epsilon-\gammastar+\betam} \partial^{\beta} u \|_{L^{p} (B_{\Rstar-k\rho})}
  &\leq C_VA_u^{k-1}(k-1)^{k-1}(k\rho)^{-(k-1)}
  \\ & \leq
  A_u^{k}\rho^{-k+1}
\end{align*}
For the last term we note that $\epsilon - \gammastar> -d/\pstar$ thus
\begin{equation*}
\|r^{\epsilon-\gammastar} u \|_{L^\pstar(B_{\Rstar})} \leq \frac{4\pi}{\pstar(\epsilon-\gammastar)+d} \| u\|_{L^\infty(\Omega)} R_\star^{\epsilon-\gammastar+d/\pstar},
\end{equation*}
and therefore
\begin{align*}
  \sum_{\betam=k-1}\| r^{2-\epsilon+\betam} \dbeta V \|_{L^{\infty} (B_{\Rstar-k\rho})} \| r^{\epsilon-\gammastar}  u \|_{L^{p} (B_{\Rstar-k\rho})}
  &\leq \frac{4\pi}{\pstar(\epsilon-\gammastar)+d} C_VA_V^{k-1}(k-1)! \\
  &\leq A_u^k \rho^{-k+1},
\end{align*}
due to \eqref{eq:Au} and to the fact that $k\leq \rho^{-1}$.

We now consider the nonlinear term: we use Lemma \ref{lemma:nonlin2} (with
$m=2$) if $n=2$,
Lemma \ref{lemma:nonlin3} (with $m=3$) if $n=3$, and Lemma \ref{lemma:nonlin4}
if $n=4$.
We have shown that, for
$n \in\{ 2, 3, 4\}$ (recall that $\theta = \frac{d}{\pstar}\left(
    1-\frac{1}{n} \right)$), 
\begin{equation*}
    \sum_{\betam=k-1}\| r^{{2-\gammastar}+\betam} \dbeta u^{n} \|_{L^{p} (B_{\Rstar-k\rho})} 
\leq C_n A_u^{{k-1}+n\theta} \rho^{-{(k-1)}-n\theta} \left(\frac{{k-1}}{k}\right)^{k-1} {(k-1)}^{(n-1)/2}.
\end{equation*}
In addition, $k\leq \rho^{-1}$, therefore,  
\begin{equation*}
    \sum_{\betam=k-1}\| r^{{2-\gammastar}+\betam} \dbeta u^{n} \|_{L^{p} (B_{\Rstar-k\rho})} 
\leq C_n A_u^{{k-1}+n\theta} \rho^{-{(k-1)}-n\theta-(n-1)/2} \left(\frac{{k-1}}{k}\right)^{k-1} .
\end{equation*}
If \eqref{eq:p-cond} holds, then $n \theta + (n-1)/2\leq 2$, hence,
since also $\left(\frac{k-1}{k}\right)^{k-1} \leq e$,
\begin{equation}
\label{eq:nonlin-estim}
\begin{aligned}
    \sum_{\betam=k-1}\| r^{{2-\gammastar}+\betam} \dbeta u^{n} \|_{L^{p} (B_{\Rstar-k\rho})} 
    & \leq C_ne A_u^{k-1+n\theta} \rho^{-{(k-1)}-n\theta-(n-1)/2}
    \\
    & \leq A_u^{k+1} ((k+1)\rho)^{-(k+1)}(k+1)^{k+1}
    \\ & \leq
    A_u^{k+1} \rho^{-k-1},
\end{aligned}
\end{equation}
where we have used that $A_u^{2-n\theta}\geq C_ne$. Note that for all
$d$ and $n$ considered, \eqref{eq:p-cond} is stronger than the hypothesis
$p > d(1-1/n)$ of Lemma \ref{lemma:thetaprod}.
The bound on the term in $\lambda u$ and on the second sum of the right hand side of 
\eqref{eq:analytic-proof1} can be obtained straightforwardly from the induction
hypothesis.
Hence, from \eqref{eq:analytic-proof1}, \eqref{eq:Vu-estim}, and
  \eqref{eq:nonlin-estim},
  \begin{align*}
    &\sum_{|\alpha| = k+1} \| r^{k+1-\gammastar} \dalpha u\|_{L^\pstar(B_{\Rstar-(k+1)\rho})} \\
      &\qquad \leq \Creg \bigg\{ 3 A_u^k\rho^{-k+1}
        + A_u^{k+1}\rho^{-k-1}
          +C_u A_u^k \rho^{-k-1}
          + C_u A_u^{k-1}\rho^{-k-1}\\
           &\qquad \qquad+ |\lambda|\Rstar^2 C_uA_u^{k-1}\rho^{-k+1} \bigg\}
      \\
    &\qquad \leq  \Creg(6+|\lambda|)A_u^{k+1} \rho^{-k-1}
\end{align*}
  where we have used $A_u \geq C_u$ to obtain the last line. Since $C_u \geq
    \Creg(6+|\lambda|)$ by \eqref{eq:Cu},
  we have proven \eqref{eq:inductionrepeat} and therefore the induction step, i.e., that for any $\pstar$ such that
  \eqref{eq:p-cond} holds, and any $\gammastar -d/\pstar \in (0, \epsilon)$,
  there exist $C_u, A_u>0$, such that for all $k\in \mathbb{N}$, $k\geq 2$,
  \begin{equation*}
H_u (k, \pstar, \gammastar, C_u, A_u) \implies H_u(k+1, \pstar, \gammastar, C_u, A_u).
\end{equation*}
Therefore, \eqref{eq:inductionrepeat} holds for all $k\in \mathbb{N}$;
  furthermore, since $\Rstar-k\rho \geq \Rstar/2$, by classical arguments
    in $\Omega\setminus B_{\Rstar/2}$ we find that there exist
  constants $\Ctil_u, \Atil_u>0$ such that
  \begin{equation*}
\sum_{\alpham = k}\| r^{\alpham -\gammastar} \dalpha u \|_{L^\pstar(\Omega)}  \leq
\Ctil_u \Atil_u^k  k^k,
\end{equation*}
for all $k \geq 1$.
Thanks to Stirling's formula, this is equivalent to (increasing the constant
$\Atil_u$ to $\Ahat_u$ in order to absorb the exponential and square root terms)
  \begin{equation}
    \label{eq:final-Lp}
\sum_{\alpham = k}\| r^{\alpham -\gammastar} \dalpha u \|_{L^\pstar(\Omega)}  \leq \Ctil_u \Ahat_u^k k!.
\end{equation}
Then, let $q$ be such that $1/q = 1/2 - 1/\pstar$. For any $\tgamma <  \gammastar  + d/\pstar= \gammastar - d/\pstar
+ d/2$, $\| r^{\gammastar-\tgamma}\|_{L^q(\Omega)} $ is bounded, and
\begin{equation}
  \label{eq:final-L2}
\sum_{\alpham = k}  \| r^{\alpham -\gammastar} \dalpha u \|_{L^\pstar(\Omega)}   \leq \| r^{\gammastar-\tgamma}\|_{L^q(\Omega)} \sum_{\alpham = k}\| r^{\alpham -\tgamma} \dalpha u \|_{L^2(\Omega)}.
\end{equation}
From \eqref{eq:final-Lp} and \eqref{eq:final-L2} we infer \eqref{eq:analytic}.
\end{proof}

\section{Exponential convergence for polynomial nonlinearities}
  \label{sec:exp-convergence}
  In this section, we make the same hypotheses
  on $F$ as in Section \ref{chap:nonlinear-regularity}, i.e., we consider the concrete case where $f(u^2)$ is a polynomial. Let then
  \begin{equation}
    \label{eq:f-poly}
    f(u^2) = |u|^{n-1}
  \end{equation}
  for $n = 2, 3, 4$ (the case $n = 1$ is the linear one). Remark that this
  class of functions satisfies \eqref{eq:cond1} to
  \eqref{eq:cond4}, with in particular $r=2$ in \eqref{eq:cond4}.

  We can regroup the results of the previous sections, applied to the case where
  \eqref{eq:f-poly} holds, in the following theorem.
  \begin{theorem}
    \label{theorem:eigenvalue-exponential}
    Let $u$, $\lambda$ be the solution to \eqref{eq:eq-cont} and $\ud$, $\ld$ be the
    solution to \eqref{eq:eq-discont}. Suppose that
   \eqref{eq:hyp-potential-1}, \eqref{eq:hyp-potential-2}, and
   \eqref{eq:f-poly} hold. Then, for a space $\Xd$ with $N$ degrees of freedom,
   there exists $b >0$ such that
    \begin{equation}
      \label{eq:optimality-exp}
    \lVert  u -\ud \rVert_\mathrm{DG} \leq C e^{-b N^{1/(d+1)}}
  \end{equation}
  and
  \begin{equation}
    \label{eq:lambda-exp}
    |\lambda -\ld| \leq C e^{-b N^{1/(d+1)}}.
  \end{equation}
  \end{theorem}
  \begin{proof}
    Bounds \eqref{eq:optimality-exp} and \eqref{eq:lambda-exp} are a
      consequence of the weighted analytic regularity of $u$ given by Theorem
      \ref{theorem:analytic}, of the exponential approximation properties of
      $\Xd$ stated in \eqref{eq:hp-exponential}, and of Theorem \ref{theorem:convergence1}.  
  \end{proof}
    \section{Numerical results}
\label{sec:numerical}
In this section, we show some results obtained in the approximation of the problem 
that, in its continuous form, reads:
given the $d$-dimensional cube of unitary edge
  $\Omega = (-1/2, 1/2)^d$,
find the eigenpair $(\lambda,
u)\in \mathbb{R}\times H^1(\Omega)$ such that $\| u \|_{L^2(\Omega)} = 1$ and
\begin{equation}
  \label{eq:nonlin-num}
  \begin{aligned}
    -\Delta  u + V u  + |u|^2 u & = \lambda u \text{ in }\Omega\\
    u & = 0 \text{ on }\dOmega,
  \end{aligned}
\end{equation}
In particular, we focus on the computation of the \emph{lowest eigenvalue} and
of its associated eigenvector,
corresponding, from a physical point of view, to the ground state of the system.
\begin{remark}
We consider different boundary conditions with respect to the
  setting of Sections \ref{chap:nonlinear-regularity} and
  \ref{sec:exp-convergence}, i.e., we consider here homogeneous Dirichlet
  boundary conditions instead of periodic ones. The theoretical analysis
  of this case is more complex, due to the fact that the ground state is not
  bounded from below and to the emergence of point and edge singularities
    in the solutions, but the behavior of the method with these boundary
  conditions is of computational interest.
  Indeed, models in quantum chemistry are normally posed in the whole
    space \cite{Szabo2012},
    and homogeneous Dirichlet boundary
  conditions can be used to approximate physical systems in the whole space
  $\mathbb{R}^d$, due to the often observed exponential decay of the wave
    functions,
    see,
    e.g., Ref.~\refcite{Lewin2004} for multiconfiguration equations.

  Our numerical results indicate that the
  exponential convergence shown for the periodic case
  can also be observed in this numerical experiment with
  homogeneous Dirichlet boundary conditions. In the settings and for the
    levels of refinement considered, the corner (end edge, for $d=3$)
    singularities do not seem to affect the exponential convergence of the
    solution.
  \end{remark}

We take potentials of the form $V(x) = -r^{-\alpha}$, for $\alpha = -1/2, -1, -3/2$.
We use a SIP method, and solve the nonlinearity by fixed point iterations.
The stopping criterion on the nonlinear iterations is residual based, i.e., we
stop iterating when
\begin{equation*}
  \langle (A^{\ud} - \ld)\ud , \ud\rangle \leq \epsilon_{\mathrm{tol}}
\end{equation*}
for a given computed solution $\ud\in\Xd$ and a given tolerance $\epsilon_{\mathrm{tol}}$.
We will
indicate the tolerance we use, on a case by case basis, in the following sections.

We impose homogeneous Dirichlet boundary conditions weakly, as is customary
  for discontinuous Galerkin methods.
  We compute elementwise integrals in each
  $K\in \mathcal{T}^\ell$ with
  Gauss quadrature, with $p_K+1$ points in each coordinate direction, and use an
  increased number of points for the elements abutting the singularity. The
  precise number of quadrature points in the latter elements was found not to
  influence the estimated errors. 
  We estimate errors by comparing all solutions of a given problem with one reference
  numerical solution obtained at higher refinement.
  All computations are done with double precision floating point numbers.
\subsection{Two dimensional case}
\label{sec:num-nonlin-2d}
\begin{figure}
  \centering
  \begin{subfigure}[b]{.5\textwidth}
  \centering
  \includegraphics[width=.5\textwidth]{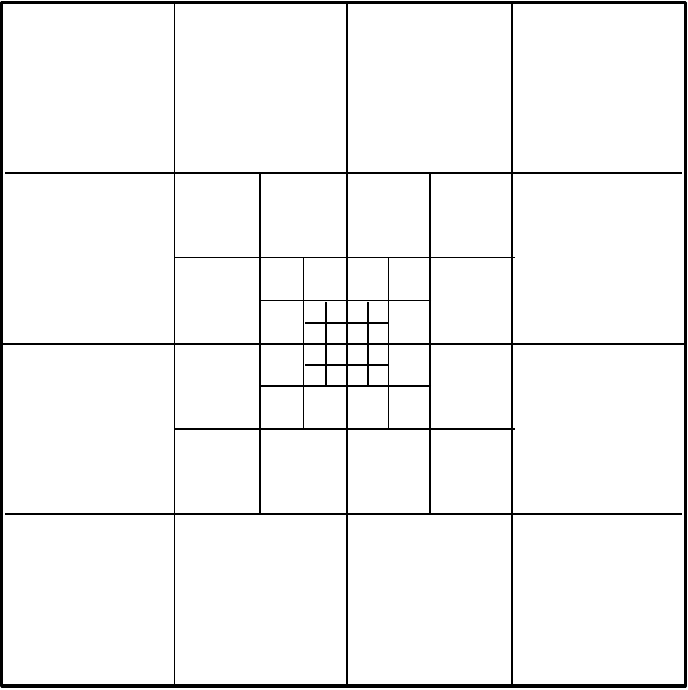}
  \caption{} \label{subfig:2d-mesh}
  \end{subfigure}%
  \begin{subfigure}[b]{.5\textwidth}
  \centering
  \includegraphics[width=.9\textwidth]{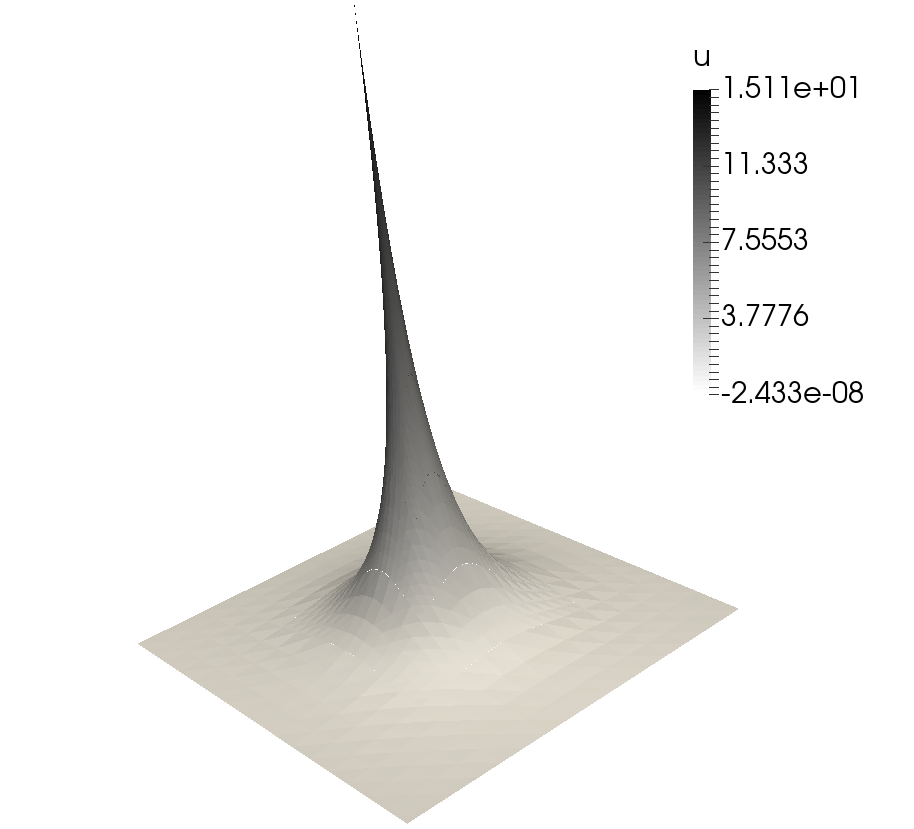}
  \caption{} \label{subfig:2d-nonlin-sol}
  \end{subfigure}
  \caption{Left: mesh for the two dimensional approximation at a fixed
    refinement step. Right: Numerical solution to \eqref{eq:nonlin-num} with $V(x) =
   - r^{-3/2}$.}
\label{fig:2d-nonlin-sol}
\end{figure}

\begin{figure}
    \centering
    \begin{subfigure}{.5\textwidth}
    \includegraphics[width=\textwidth]{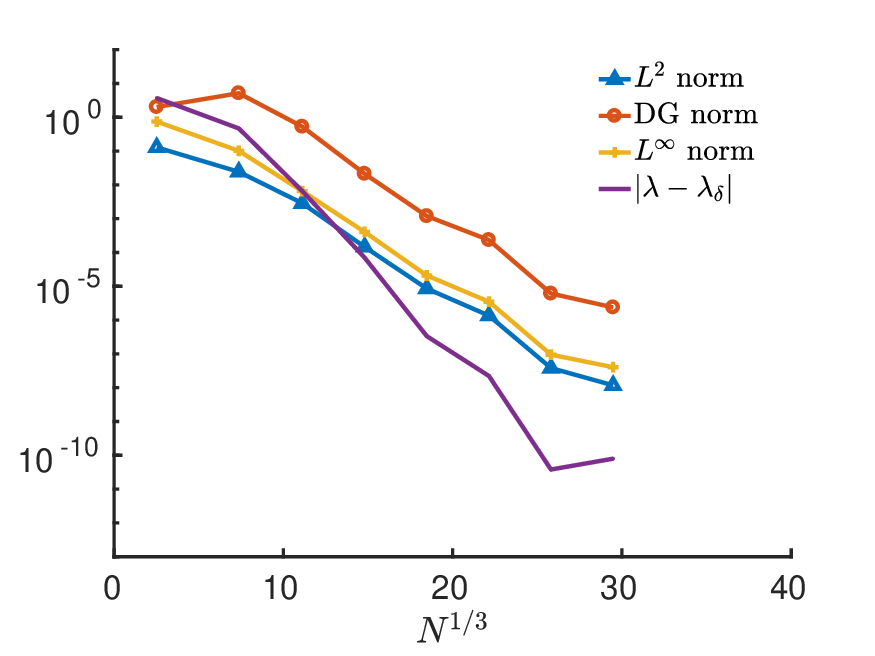}
    \caption{}\label{subfig:2d-nonlin-p050-s012_k200}
    \end{subfigure}%
    \begin{subfigure}{.5\textwidth}
    \includegraphics[width=\textwidth]{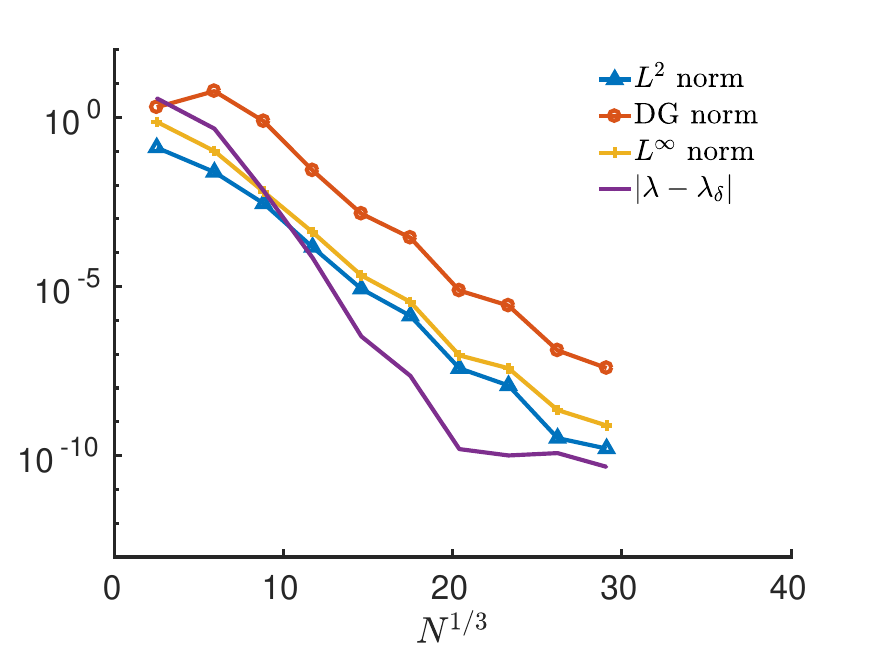}
    \caption{}\label{subfig:2d-nonlin-p050-s025_k200}
    \end{subfigure}\\
    \begin{subfigure}{.5\textwidth}
    \includegraphics[width=\textwidth]{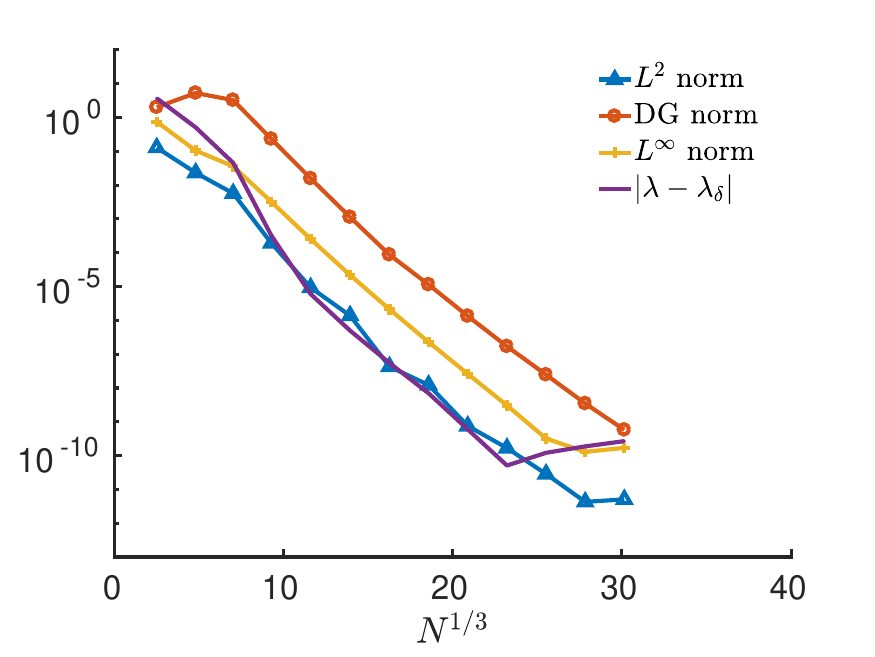}
    \caption{}\label{subfig:2d-nonlin-p050-s050_k200}
    \end{subfigure}
    \caption{Estimated errors for the numerical solution with potential $V(x) = -r^{-1/2}$.
    Polynomial slope: $\slope = 1/8$ in Figure \subref{subfig:2d-nonlin-p050-s012_k200};
    $\slope = 1/4$ in Figure \subref{subfig:2d-nonlin-p050-s025_k200} and $\slope= 1/2$ in
  Figure \subref{subfig:2d-nonlin-p050-s050_k200}.}\label{fig:2d-nonlin-p050_k200}
\end{figure}

\begin{table}
\caption{Estimated coefficients. Potential: $-r^{-1/2}$.}
\label{table:2d-nonlin-p050-b}
\centering
\pgfplotstablevertcat{\output}{tables/nonlinear_results_pot-0_50_slope0_12_k2_00_b} 
\pgfplotstablevertcat{\output}{tables/nonlinear_results_pot-0_50_slope0_25_k2_00_b} 
\pgfplotstablevertcat{\output}{tables/nonlinear_results_pot-0_50_slope0_50_k2_00_b} 
\pgfplotstabletypeset {\output}
\end{table}
\begin{figure}
    \centering
    \begin{subfigure}{.5\textwidth}
    \includegraphics[width=\textwidth]{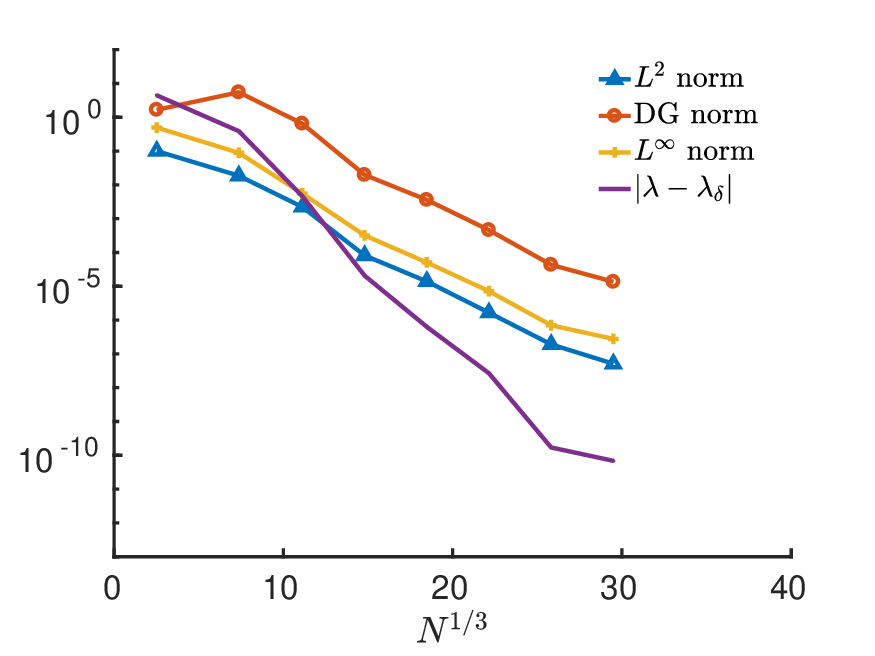}
    \caption{}\label{subfig:2d-nonlin-p100-s012_k200}
    \end{subfigure}%
    \begin{subfigure}{.5\textwidth}
    \includegraphics[width=\textwidth]{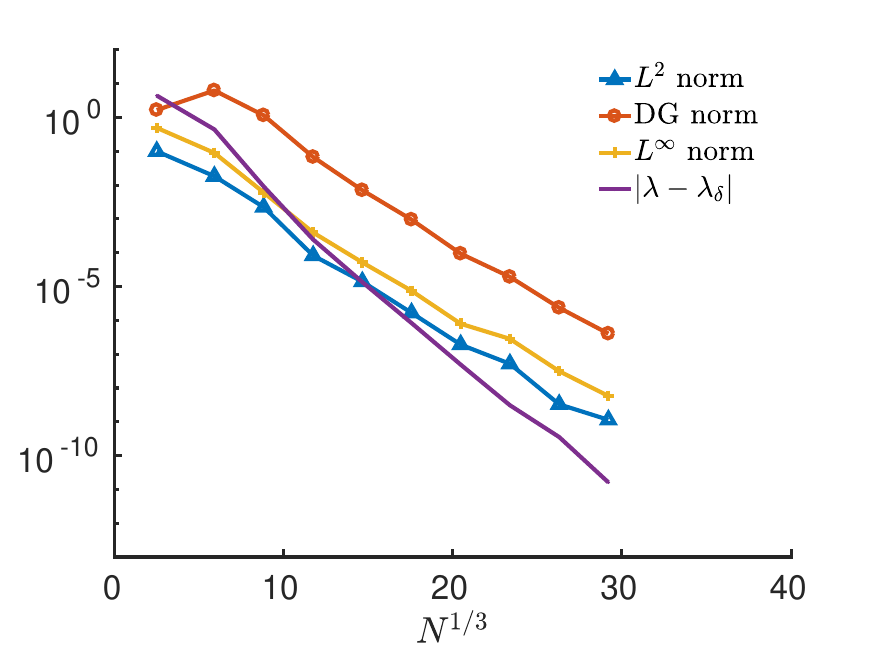}
    \caption{}\label{subfig:2d-nonlin-p100-s025_k200}
    \end{subfigure}\\
    \begin{subfigure}{.5\textwidth}
    \includegraphics[width=\textwidth]{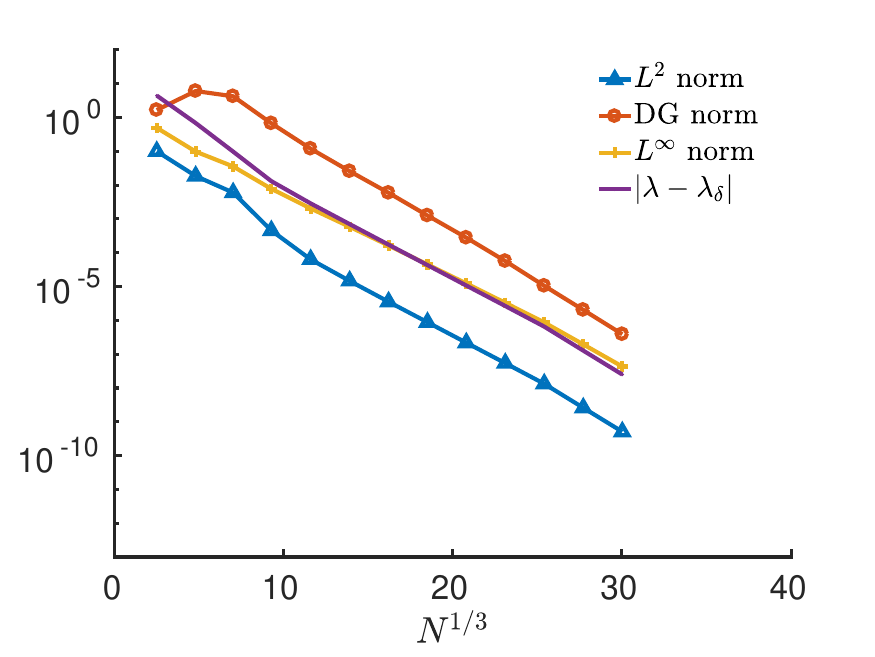}
    \caption{}\label{subfig:2d-nonlin-p100-s050_k200}
    \end{subfigure}
    \caption{Estimated errors for the numerical solution with potential $V(x) = -r^{-1}$.
    Polynomial slope: $\slope = 1/8$ in Figure \subref{subfig:2d-nonlin-p100-s012_k200};
    $\slope = 1/4$ in Figure \subref{subfig:2d-nonlin-p100-s025_k200} and $\slope= 1/2$ in
  Figure \subref{subfig:2d-nonlin-p100-s050_k200}.}\label{fig:2d-nonlin-p100_k200}
\end{figure}

\begin{table}
\caption{Estimated coefficients. Potential: $-r^{-1}$.}
\label{table:2d-nonlin-p100-b}
\centering
\pgfplotstablevertcat{\output}{tables/nonlinear_results_pot-1_00_slope0_12_k2_00_b} 
\pgfplotstablevertcat{\output}{tables/nonlinear_results_pot-1_00_slope0_25_k2_00_b} 
\pgfplotstablevertcat{\output}{tables/nonlinear_results_pot-1_00_slope0_50_k2_00_b} 
\pgfplotstabletypeset {\output}
\end{table}

\begin{figure}
    \centering
    \begin{subfigure}{.5\textwidth}
    \includegraphics[width=\textwidth]{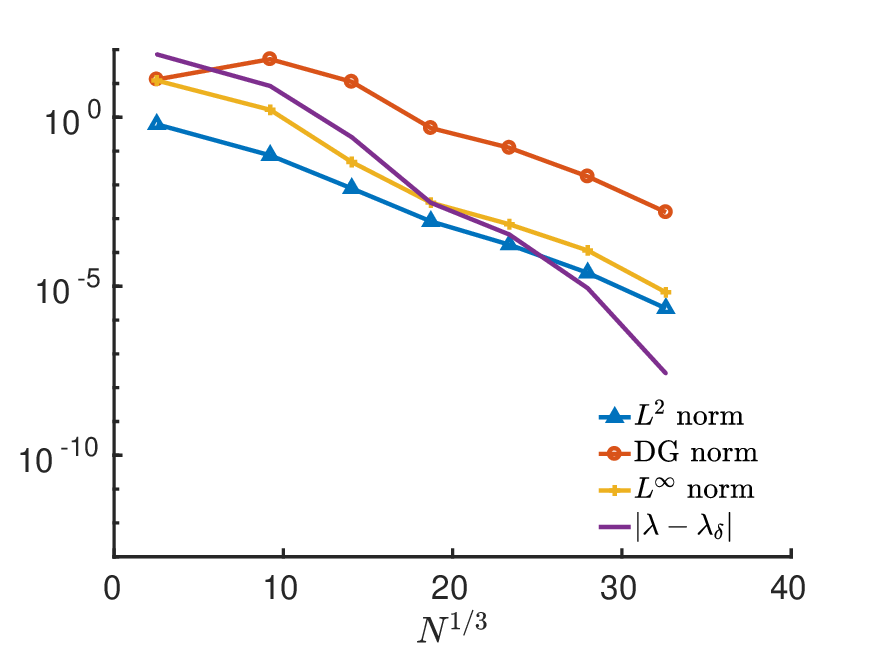}
    \caption{}\label{subfig:2d-nonlin-p150-s006_k200}
    \end{subfigure}%
    \begin{subfigure}{.5\textwidth}
    \includegraphics[width=\textwidth]{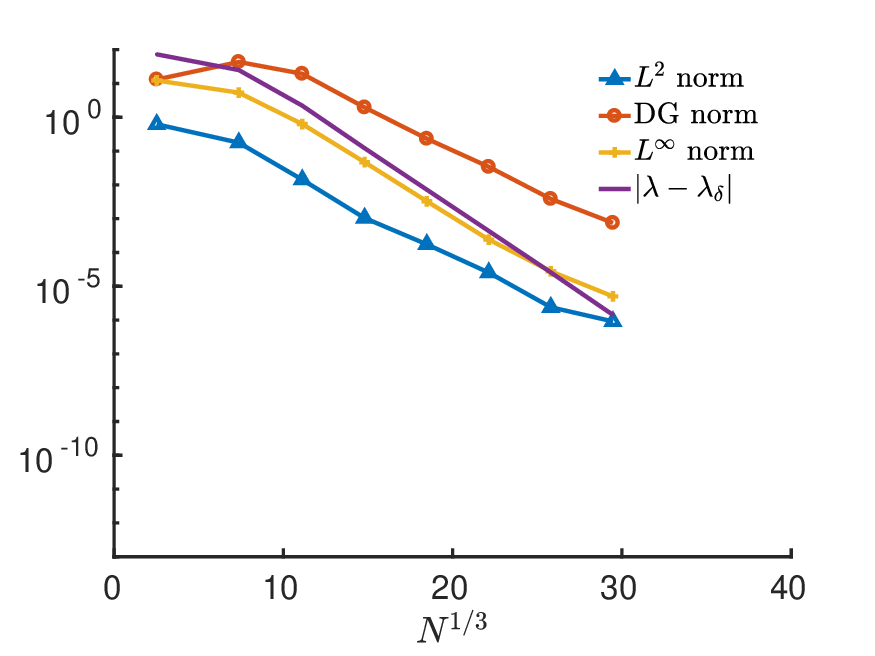}
    \caption{}\label{subfig:2d-nonlin-p150-s012_k200}
    \end{subfigure}\\
    \begin{subfigure}{.5\textwidth}
    \includegraphics[width=\textwidth]{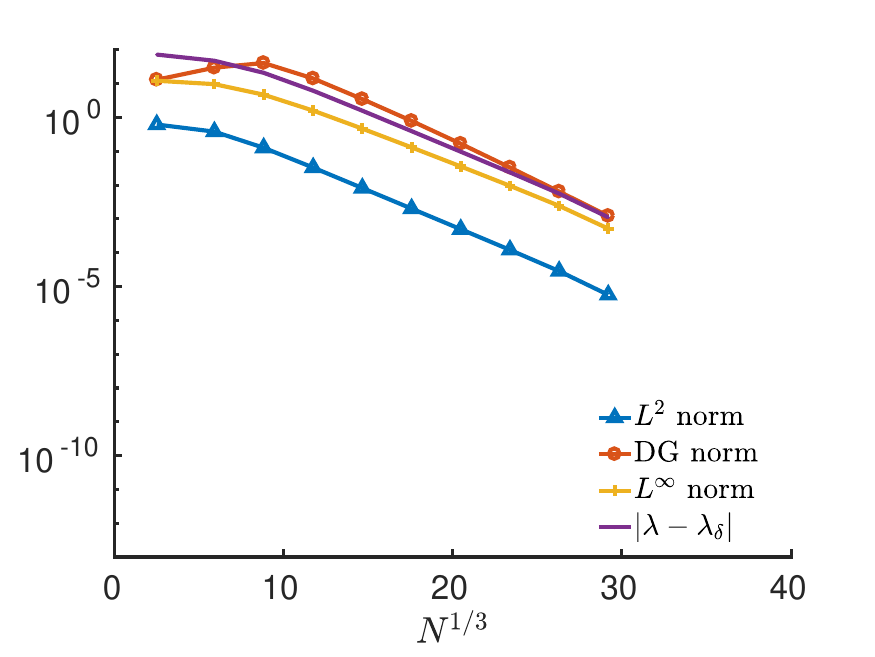}
    \caption{}\label{subfig:2d-nonlin-p150-s025_k200}
    \end{subfigure}
    \caption{Estimated errors for the numerical solution with potential $V(x) = -r^{-3/2}$.
    Polynomial slope: $\slope = 1/16$ in Figure \subref{subfig:2d-nonlin-p150-s006_k200};
    $\slope = 1/8$ in Figure \subref{subfig:2d-nonlin-p150-s012_k200} and $\slope= 1/4$ in
  Figure \subref{subfig:2d-nonlin-p150-s025_k200}.}\label{fig:2d-nonlin-p150_k200}
\end{figure}
\begin{table}
\caption{Estimated coefficients. Potential: $-r^{-3/2}$.}
\centering
\pgfplotstablevertcat{\output}{tables/nonlinear_results_pot-1_50_slope0_06_k2_00_b} 
\pgfplotstablevertcat{\output}{tables/nonlinear_results_pot-1_50_slope0_12_k2_00_b} 
\pgfplotstablevertcat{\output}{tables/nonlinear_results_pot-1_50_slope0_25_k2_00_b} 
\pgfplotstabletypeset {\output}
\label{table:2d-nonlin-p150-b}
\end{table}
In the two dimensional case, we compute the numerical solutions on meshes built
with refinement ratio $\sigma=1/2$, see Figure \ref{subfig:2d-mesh}. A visualization of the
solution (in the most singular problem we analyse) is given in Figure \ref{subfig:2d-nonlin-sol}.

Writing $V(x) = -r^{-\alpha}$, we plot the curves of the errors in Figures 
\ref{fig:2d-nonlin-p050_k200} ($\alpha=1/2$), \ref{fig:2d-nonlin-p100_k200}
($\alpha=1$), and \ref{fig:2d-nonlin-p150_k200} ($\alpha=3/2$).
In the case of the approximations with
low polynomial slopes, all errors converge exponentially in the number of
refinement steps, with the eigenvalue error converging faster than the norms of
the eigenfunction error. Estimated errors tend to reach a plateau at values around
$10^{-10}$. We conjecture this to be due to algebraic error, as the matrices
resulting from the discretization are very ill conditioned for high levels of
refinement \cite{Antonietti2011}. A comprehensive study of this effect and of the efficient
preconditioning of this problem is out of the scope of the present work.
When the polynomial slopes are higher, the quadrature
error --- not analyzed here, see Ref.~\refcite{Cances2010} for the analysis for
  $h$-type FE ---
 manifests itself more strongly and causes, in extreme cases, the total loss of
the doubling of the convergence rate.

The coefficients $b_X$, for $X= L^2(\Omega), \mathrm{DG}, L^\infty(\Omega)$ and
$\lambda$ are shown in Tables \ref{table:2d-nonlin-p050-b} to
\ref{table:2d-nonlin-p150-b}. As already discussed in Ref.~\refcite{linear}, the higher the slope, the
biggest the quadrature error and the furthest the estimated coefficients
$b_\lambda$ is from the double of the one for the $\mathrm{DG}$ norm.
\subsubsection{Corner refinement}
  Since we impose Dirichlet boundary conditions, singularities in the solution
  can emerge at the corner of the domain. To verify that those potential (milder)
  singularities do not influence, up to the precision considered, the
  convergence of the numerical scheme, we perform a comparison between the
  results with the mesh in Figure \ref{subfig:2d-mesh} and a mesh with
  refinement towards the corners of the domain (see Figure
  \ref{subfig:2d-mesh-corner}). The results are shown in Figure
  \ref{fig:2d-corner} for $\slope = 1/2$ and $V(x) = |x|^{-1}$. We compute the
  errors at the same levels of refinement for corner-refined and non
  corner-refined meshes and denote by $N_\mathrm{cor}$ the number of degrees of
  freedom of the \hp{} space on the corner-refined mesh and by $N$ the number of
  degrees of freedom of the space on the mesh without corner refinement.
  At a given level of refinement, the errors are
  approximately the same, but since $N_{\mathrm{cor}}\simeq 2 N$ the
  approximation with corner refined mesh is less efficient. This could be related
  to the superconvergence proved in Ref. \refcite{Bernardi1991}.
  Other choices of slopes and
  potentials give similar results. 
  \begin{figure}
    \centering
    \begin{subfigure}{\textwidth}
      \centering
      \includegraphics[width=.3\textwidth]{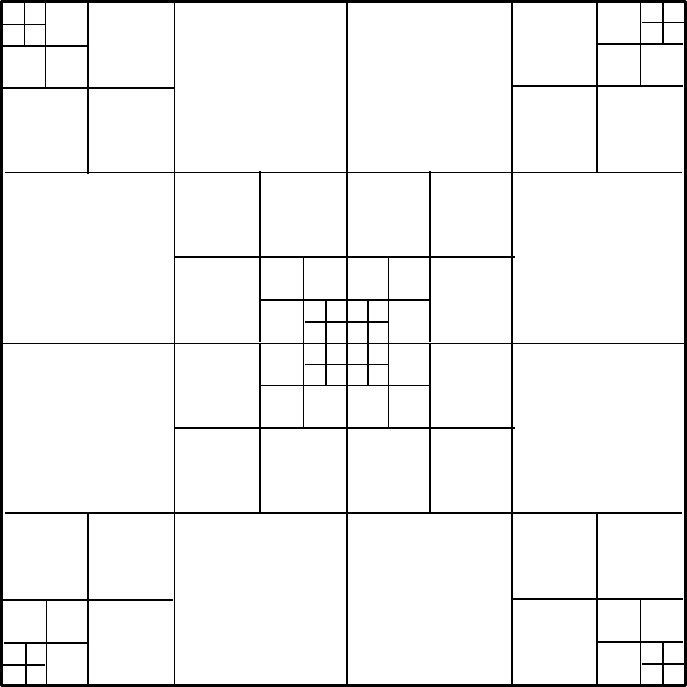}
      \caption{}
      \label{subfig:2d-mesh-corner}
    \end{subfigure}\\
    \begin{subfigure}{.5\textwidth}
      \includegraphics[width=\textwidth]{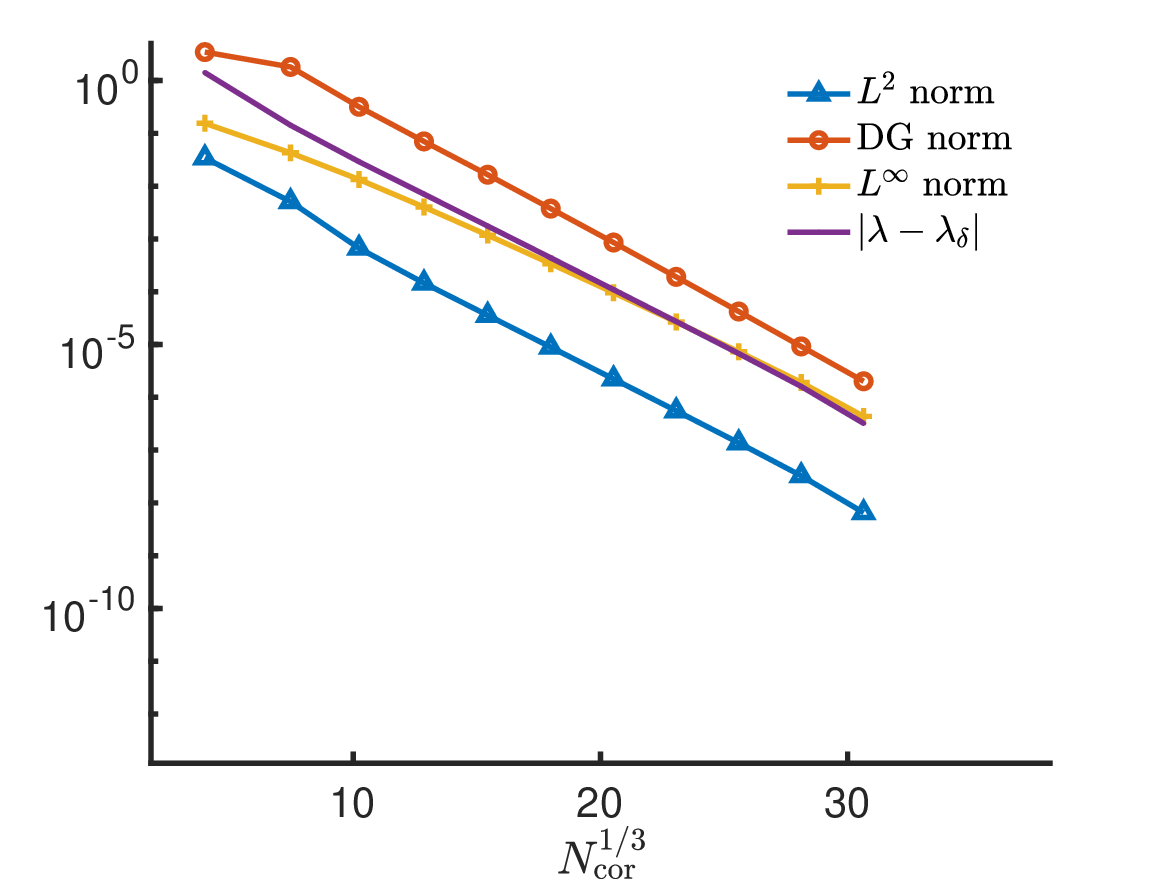}
      \caption{}
      \label{subfig:2d-corner-ref}
    \end{subfigure}%
    \begin{subfigure}{.5\textwidth}
      \includegraphics[width=\textwidth]{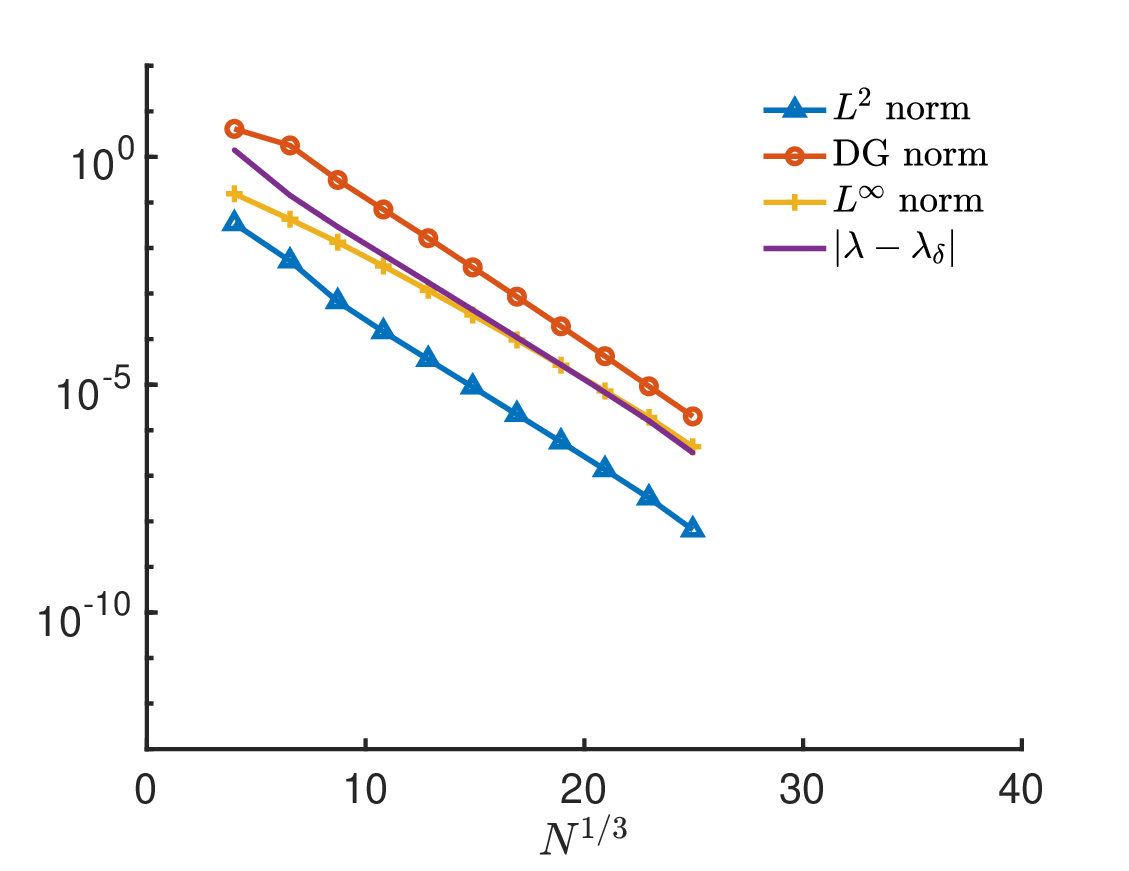}
      \caption{}
      \label{subfig:2d-corner-noref}
    \end{subfigure}
    \caption{Top: mesh with corner refinement. Bottom: comparison of results with (left) and
      without (right) corner refinement.}
    \label{fig:2d-corner}
  \end{figure}
\subsection{Three dimensional problem}
\begin{figure}[h!]
  \centering
  \begin{subfigure}[t]{.5\textwidth}
  \centering
  \includegraphics[width=\textwidth]{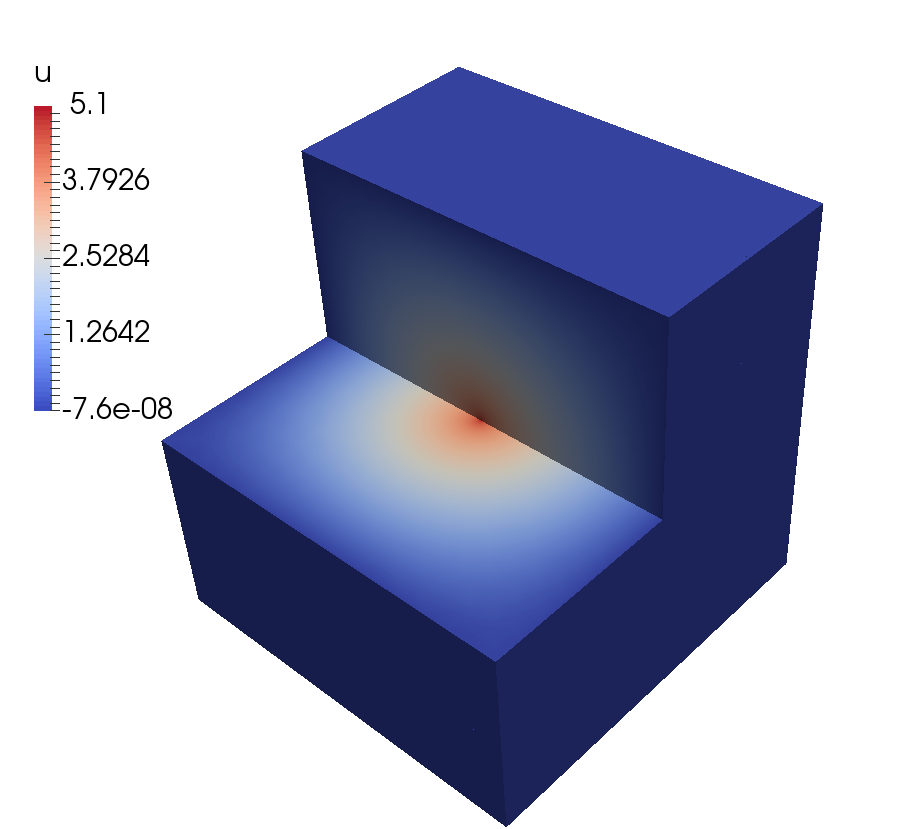}
  \end{subfigure}%
  \begin{subfigure}[t]{.5\textwidth}
  \centering
  \includegraphics[width=.7\textwidth]{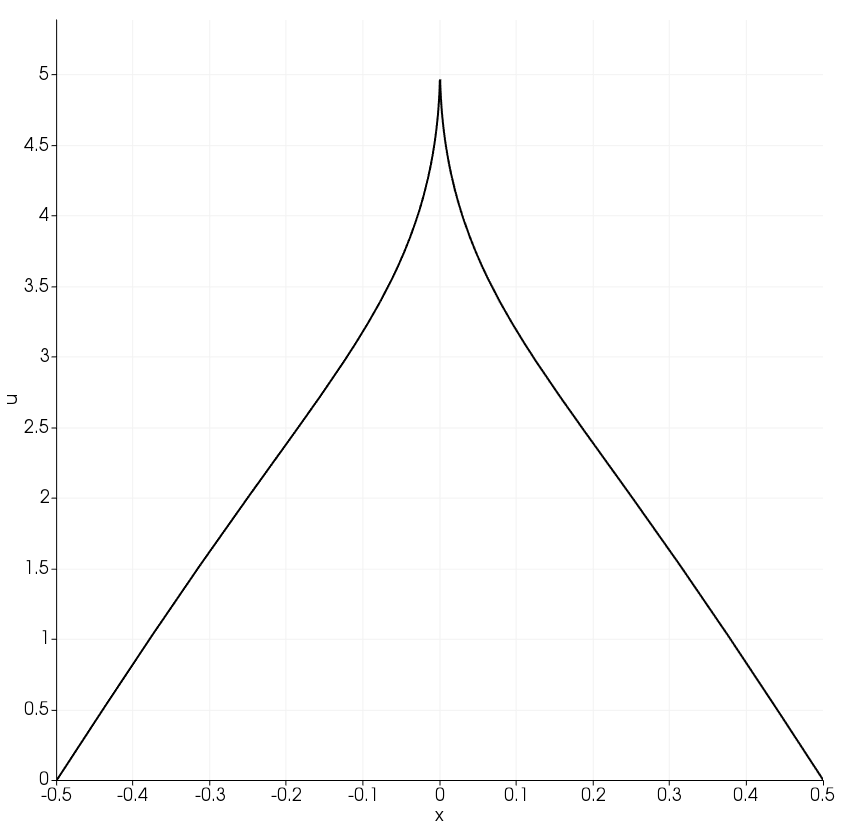}
  \end{subfigure}
  \caption{Numerical solution in the three dimensional case: solution in the
    cube, left, and close up near the origin of the restriction to the line $\{y= z= 0\}$, right}
  \label{fig:3d-nonlin-sol}
\end{figure}

\begin{figure}
    \centering
    \begin{subfigure}{.5\textwidth}
    \includegraphics[width=\textwidth]{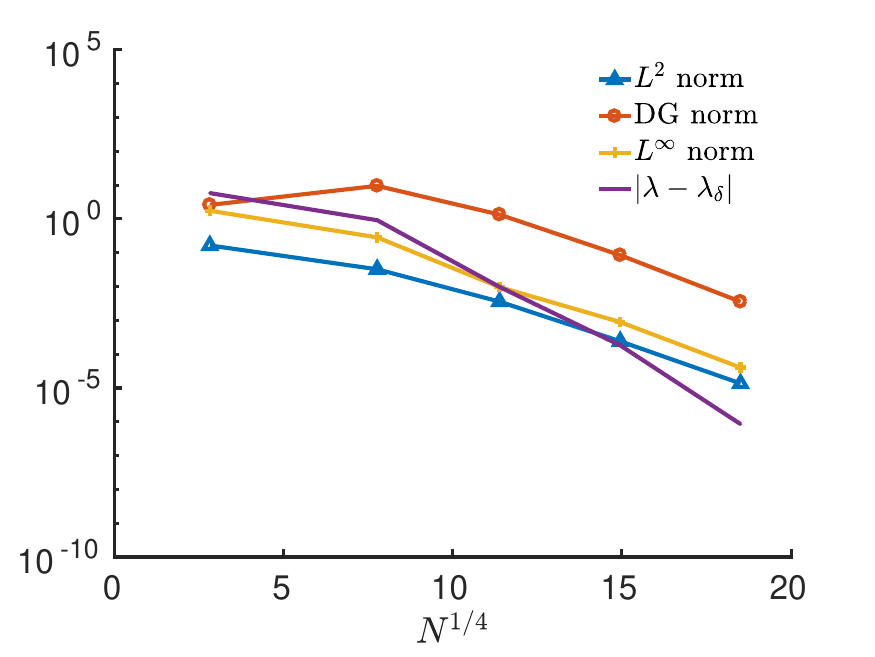}
    \end{subfigure}%
    \begin{subfigure}{.5\textwidth}
    \includegraphics[width=\textwidth]{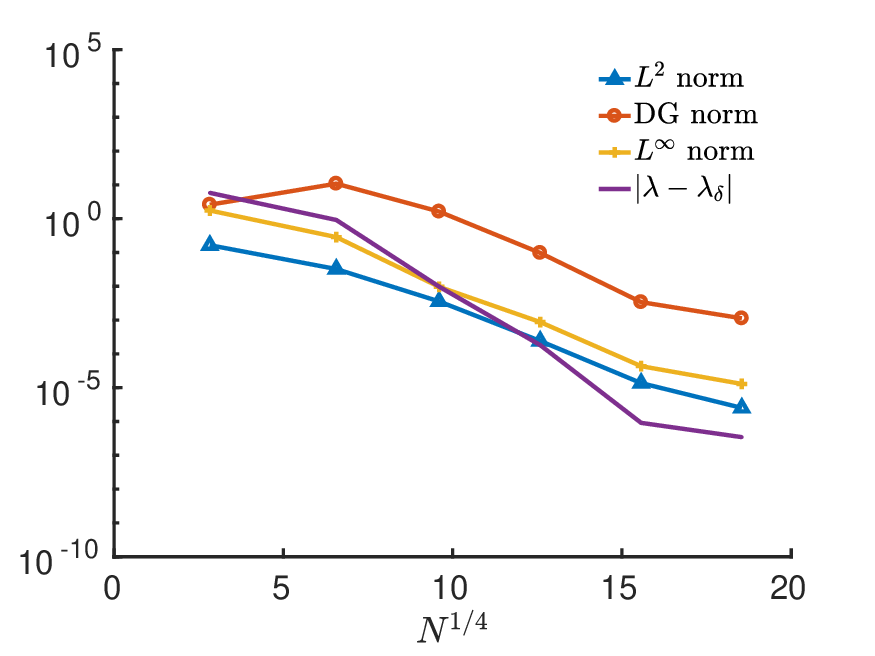}
  \end{subfigure}
  \caption{Estimated errors of the numerical solution for $V(x) = r^{-1/2}$. Polynomial
    slope $\slope = 1/8$, left and $\slope = 1/4$, right.}
  \label{fig:3d-nonlin-p050}
\end{figure}
\begin{figure}
    \centering
    \begin{subfigure}{.5\textwidth}
    \includegraphics[width=\textwidth]{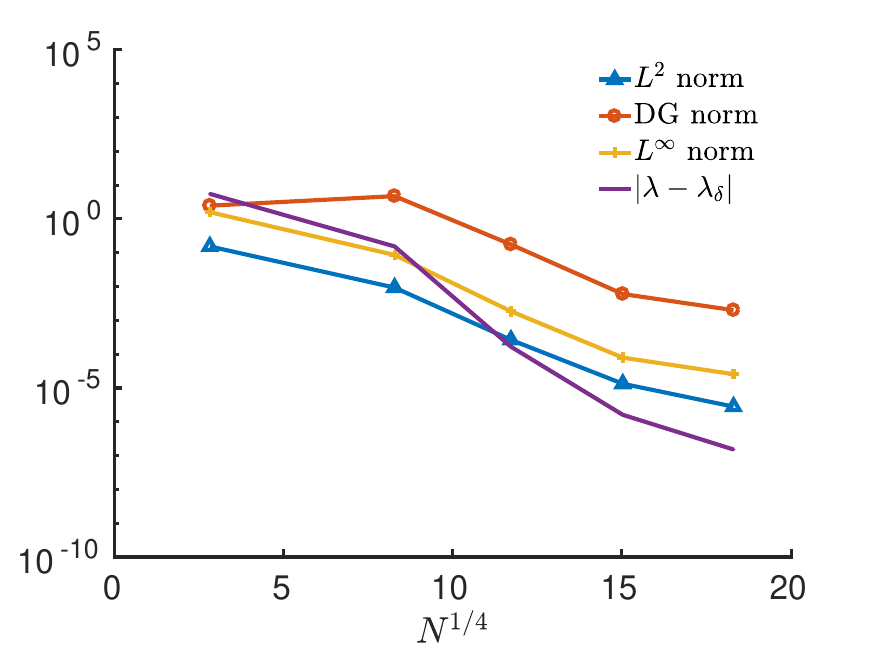}
    \end{subfigure}%
    \begin{subfigure}{.5\textwidth}
    \includegraphics[width=\textwidth]{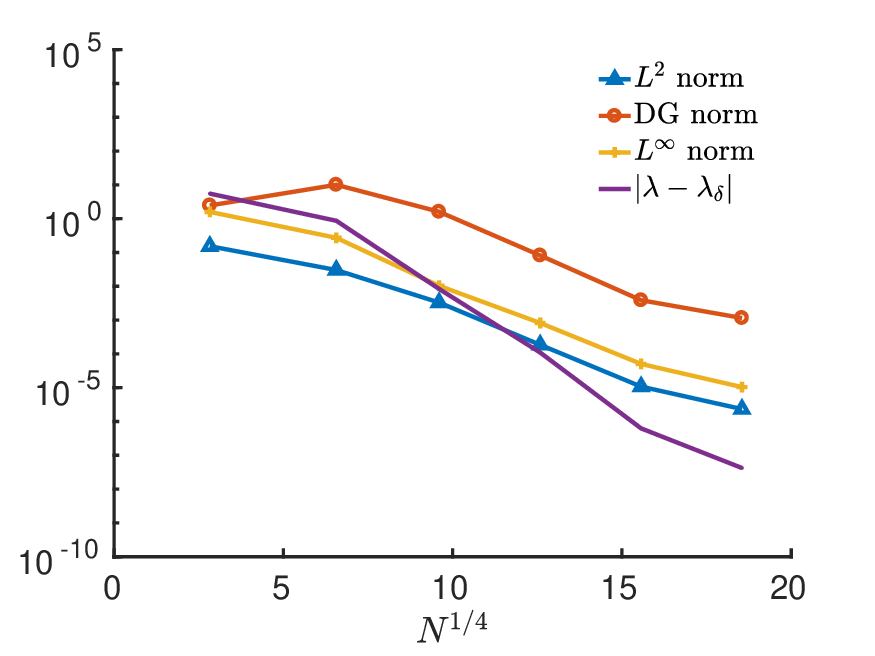}
  \end{subfigure}
    \caption{Estimated errors of the numerical solution for $V(x) = r^{-1}$. Polynomial
    slope $\slope = 1/8$, left and $\slope = 1/4$, right.}
  \label{fig:3d-nonlin-p100}
\end{figure}
\begin{figure}
    \centering
    \begin{subfigure}{.5\textwidth}
    \includegraphics[width=\textwidth]{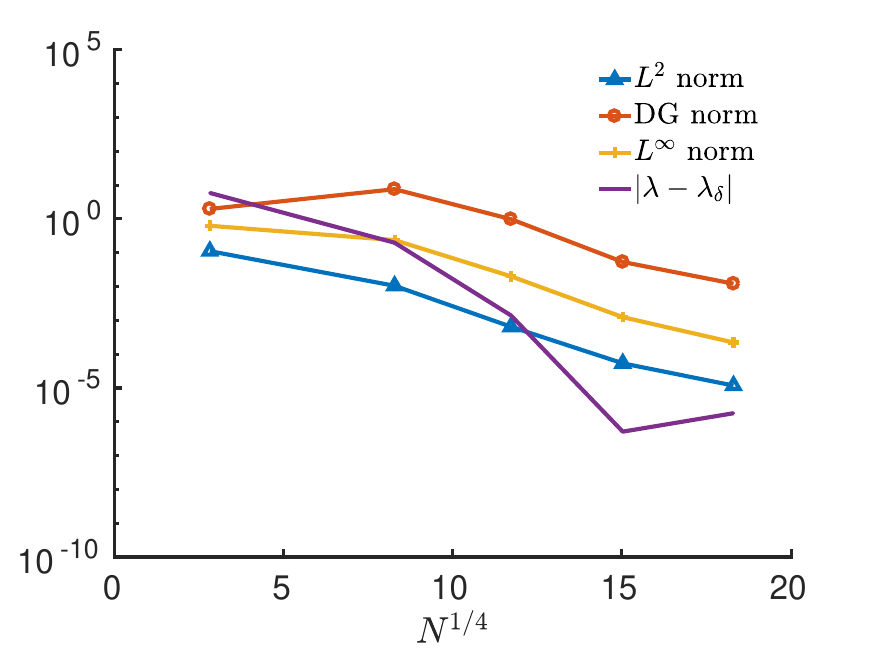}
    \end{subfigure}%
    \begin{subfigure}{.5\textwidth}
    \includegraphics[width=\textwidth]{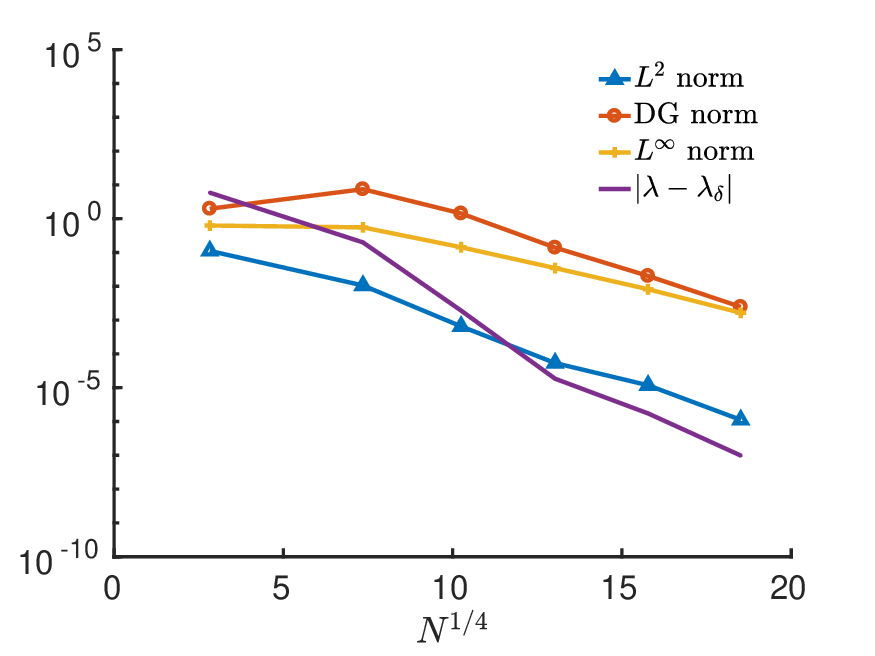}
  \end{subfigure}
    \caption{Estimated errors of the numerical solution for $V(x) = r^{-3/2}$. Polynomial
    slope $\slope = 1/8$, left and $\slope = 1/4$, right.}
  \label{fig:3d-nonlin-p150}
\end{figure}

\begin{table}
\caption{Estimated coefficients. Potential: $- r^{-1/2}$.}
  \label{table:3d-nonlin-p050}
\centering
\pgfplotstablevertcat{\output}{tables/nonlinear3d_results_pot-0_50_slope0_12_k2_00_b} 
\pgfplotstablevertcat{\output}{tables/nonlinear3d_results_pot-0_50_slope0_25_k2_00_b} 
\pgfplotstabletypeset {\output}
\end{table}
\begin{table}
\caption{Estimated coefficients. Potential: $ - r^{-1}$.}
  \label{table:3d-nonlin-p100}
\centering
\pgfplotstablevertcat{\output}{tables/nonlinear3d_results_pot-1_00_slope0_12_k2_00_b} 
\pgfplotstablevertcat{\output}{tables/nonlinear3d_results_pot-1_00_slope0_25_k2_00_b} 
\pgfplotstabletypeset {\output}
\end{table}
\begin{table}
\caption{Estimated coefficients. Potential: $ - r^{-3/2}$.}
  \label{table:3d-nonlin-p150}
\centering
\pgfplotstablevertcat{\output}{tables/nonlinear3d_results_pot-1_50_slope0_12_k2_00_b} 
\pgfplotstablevertcat{\output}{tables/nonlinear3d_results_pot-1_50_slope0_25_k2_00_b} 
\pgfplotstabletypeset {\output}
\end{table}
In the three dimensional setting, we consider the domain $(-1/2, 1/2)^3$, and a mesh
exemplified with refinement ratio $\sigma = 1/2$.
The numerical solution of the problem with $V(x) = r^{-3/2}$ is shown in Figure \ref{fig:3d-nonlin-sol}. The solution shown is obtained at
one of the highest degrees of refinement.
The algebraic
eigenproblem solver uses the Jacobi-Davidson method \cite{Sleijpen1996}, with a biconjugate
gradient method \cite{VanderVorst1992,Sleijpen1994} as the linear algebraic system solver. The fixed point nonlinear
iteration are set to a tolerance $\epsilon_{\mathrm{tol}} = 10^{-7}$.

The algebraic and quadrature errors are not as evident as in the two
dimensional case, and it can clearly be seen that an optimal slope can be chosen
to better approximate the eigenvalue.

As in the two dimensional case, we expect that corner and edge
  singularities do not influence the convergence of the method. We do not
  investigate this further.

The presence of the nonlinearity does not seem to
influence the rate of convergence with respect to the rate obtained for
  the $hp$ dG approximation of
  linear eigenvalue problem with singular potentials in
  Ref.~\refcite{linear}. This is expected, as the regularity of the solution of
  the linear problem considered in Ref.~\refcite{linear} and the regularity of
  the solution of the problem considered in this section is the same.
\section*{Acknowledgements}
The authors are grateful to the referee for their valuable and constructive comments, which have contributed to the improvement of the paper.

\FloatBarrier
\bibliographystyle{ws-m3as}
\bibliography{library}

\end{document}